\documentclass[12pt,reqno]{amsart}
\usepackage{amsmath,a4wide}
\usepackage{xfrac}
\usepackage{stmaryrd,mathrsfs,bm,amsthm,mathtools,yfonts,amssymb,color}
\usepackage{xcolor}
\usepackage{courier}
\newcommand{\cT}{{\mathcal{T}}}

\begingroup
\newtheorem{theorem}{Theorem}[section]
\newtheorem{lemma}[theorem]{Lemma}
\newtheorem{proposition}[theorem]{Proposition}
\newtheorem{corollary}[theorem]{Corollary}
\endgroup

\theoremstyle{definition}
\newtheorem{definition}[theorem]{Definition}

\newtheorem{ipotesi}[theorem]{Assumption}

\numberwithin{equation}{section}
\numberwithin{subsection}{section}

\newcommand{\sing}{{\rm Sing}}
\newcommand{\reg}{{\rm Reg}}
\newcommand\supp{{\rm spt}}
\newcommand\res{\mathop{\hbox{\vrule height 7pt width .3pt depth 0pt
\vrule height .3pt width 5pt depth 0pt}}\nolimits}
\newcommand\ser{\mathop{\hbox{\vrule height .3pt width 5pt depth 0pt
\vrule height 7pt width .3pt depth 0pt}}\nolimits}

\newcommand{\mass}{{\mathbf{M}}}
\newcommand{\cone}{{\times\hspace{-0.6em}\times\,}}
\newcommand{\bOmega}{{\mathbf{\Omega}}}
\newcommand{\de}{\partial}

\def\a#1{\left\llbracket{#1}\right\rrbracket}
\newcommand{\bE}{{\mathbf{E}}}
\newcommand{\gr}{{\rm Gr}}
\newcommand{\bh}{{\mathbf{h}}}
\newcommand{\im}{{\rm Im}}
\newcommand{\bT}{\mathbf{T}}
\newcommand{\cH}{{\mathcal{H}}}

\def\I#1{{\mathcal{A}}_{#1}}
\newcommand{\Iq}{{\mathcal{A}}_Q}
\newcommand{\Iqs}{{\mathcal{A}}_Q(\R^{n})}
\newcommand{\bG}{{\mathbf{G}}}
\newcommand{\cG}{{\mathcal{G}}}
\newcommand{\etaa}{{\bm{\eta}}}
\newcommand{\D}{\textup{Dir}}

\newcommand{\bB}{{\mathbf{B}}}
\newcommand{\bC}{{\mathbf{C}}}

\newcommand{\bP}{\mathbf{P}}

\newcommand\bS{\mathbf{S}}

\def\Xint#1{\mathchoice
{\XXint\displaystyle\textstyle{#1}}%
{\XXint\textstyle\scriptstyle{#1}}%
{\XXint\scriptstyle\scriptscriptstyle{#1}}%
{\XXint\scriptscriptstyle\scriptscriptstyle{#1}}%
\!\int}
\def\XXint#1#2#3{{\setbox0=\hbox{$#1{#2#3}{\int}$ }
\vcenter{\hbox{$#2#3$ }}\kern-.6\wd0}}
\def\mint{\Xint-}

\newcommand\Z{{\mathbb Z}}
\newcommand\N{{\mathbb N}}

\newcommand\C{{\mathbb C}}
\newcommand\R{{\mathbb R}}

\newcommand{\eps}{{\varepsilon}}
\newcommand{\bA}{\mathbf{A}}
\newcommand{\bmax}{\mathbf{m}}

\newcommand{\Lip}{{\rm {Lip}}}
\newcommand{\diam}{{\rm {diam}}}
\newcommand{\dist}{{\rm {dist}}}
\renewcommand{\d}{{\rm d}}

\newcommand{\p}{{\mathbf{p}}}
\newcommand{\q}{{\mathbf{q}}}
\renewcommand{\P}{{\mathbf{P}}}

\newcommand{\gira}{{\mathfrak{B}}}
\newcommand{\disco}{{\mathfrak{D}}}
\newcommand{\bV}{{\mathbf{V}}}

\newcommand{\bef}{\mathbf{f}}
\newcommand{\beg}{\mathbf{g}}
\newcommand{\bGam}{{\bm \Gamma}}
\newcommand\bmo{{\bm m}_0}

\newcommand{\cM}{{\mathcal{M}}}

\newcommand{\bU}{{\mathbf{U}}}
\newcommand{\bL}{{\mathbf{L}}}
\newcommand{\phii}{{\bm{\varphi}}}
\newcommand{\Phii}{{\bm{\Phi}}}
\newcommand{\bc}{{\bm{c}}}
\newcommand{\sep}{{\rm sep}}

\newcommand{\cU}{{\mathcal{U}}}
\newcommand{\cV}{{\mathcal{V}}}

\newcommand{\cJ}{{\mathcal{J}}}
\newcommand{\cL}{{\mathcal{L}}}
\newcommand{\cK}{{\mathcal{K}}}
\newcommand{\cB}{{\mathcal{B}}}

\newcommand{\cA}{{\mathcal{A}}}
\newcommand{\sC}{{\mathscr{C}}}
\newcommand{\sS}{{\mathscr{S}}}
\newcommand{\sW}{{\mathscr{W}}}
\newcommand{\sP}{{\mathscr{P}}}
\newcommand{\bW}{{\mathbf{W}}}

\newcommand{\Psii}{{\bm{\Psi}}}
\newcommand{\bD}{{\mathbf{D}}}
\newcommand{\bH}{{\mathbf{H}}}

\newcommand{\bF}{{\mathbf{F}}}

\newcommand{\bLambda}{{\mathbf{\Lambda}}}

\DeclareMathAlphabet{\mathpzc}{OT1}{pzc}{m}{it}
\newcommand{\NN}{{\mathpzc{N}}}
\newcommand{\FF}{{\mathpzc{F}}}

\title[Branched center manifold]
{Regularity theory for $2$-dimensional almost minimal currents II: branched center manifold}
\author{Camillo De Lellis, Emanuele Spadaro and Luca Spolaor}

\begin{document}

\begin{abstract}
We construct a branched center manifold in a neighborhood of a singular point of a $2$-dimensional integral current which
is almost minimizing in a suitable sense. Our construction is the first half of an argument which shows the discreteness of
the singular set for the following three classes of $2$-dimensional currents: area minimizing in Riemannian manifolds, semicalibrated and spherical cross sections of $3$-dimensional
area minimizing cones.
\end{abstract}

\maketitle

This paper is the third in a series of works aimed at establishing an optimal regularity theory
for $2$-dimensional integral currents which are almost minimizing in a suitable sense. Building upon the monumental
work of Almgren \cite{Alm}, Chang in \cite{Chang} established that $2$-dimensional area minimizing currents in
Riemannian manifolds are classical minimal surfaces, namely they are regular (in the interior) except for a discrete
set of branching singularities. The argument of Chang is however not entirely complete since a key starting point 
of his analysis, the existence of the so-called ``branched center manifold'', is only sketched in the appendix of
\cite{Chang} and requires the understanding (and a suitable modification) of the most involved portion of the monograph \cite{Alm}. 

An alternative proof of Chang's theorem has been found by Rivi\`ere and Tian in \cite{RT1} for the special case
of $J$-holomorphic curves. Later on the approach of Rivi\`ere and Tian has been generalized by Bellettini and Rivi\`ere
in \cite{BeRi} to handle a case which is not covered by \cite{Chang}, namely that of special Legendrian cycles in
$\mathbb S^5$ (see also \cite{Be3} for a further generalization).

Meanwhile the first and second author revisited Almgren's theory giving a much shorter version of his program
for proving that area minimizing currents are regular up to a set of Hausdorff codimension $2$, cf. \cite{DS1,DS2,DS3,DS4,DS5}.
In this note and its companion papers \cite{DSS2,DSS4} we build upon the latter works in order to give a
complete regularity theory which includes both the theorems of Chang and Bellettini-Rivi\`ere as special cases. 
In order to be more precise, we introduce the following terminology (cf. \cite[Definition 0.3]{DSS1}).

\begin{definition}\label{d:semicalibrated}
Let $\Sigma \subset R^{m+n}$ be a $C^2$ submanifold and $U\subset \R^{m+n}$ an open set.
\begin{itemize}
  \item[(a)] An $m$-dimensional integral current $T$ with finite mass and $\supp (T)\subset \Sigma\cap U$ is area minimizing in $\Sigma\cap U$
if $\mass(T + \partial S)\geq \mass(T)$ for any $m+1$-dimensional integral current $S$ with $\supp (S) \subset \subset \Sigma\cap U$.
 \item[(b)] A semicalibration (in $\Sigma$) is a $C^1$ $m$-form $\omega$ on $\Sigma$ such that 
  $\|\omega_x\|_c \leq 1$ at every $x\in \Sigma$, where $\|\cdot \|_c$ denotes the comass norm on $\Lambda^m T_x \Sigma$. 
  An $m$-dimensional integral current $T$ with $\supp (T)\subset \Sigma$ is {\em semicalibrated} by $\omega$ if $\omega_x (\vec{T}) = 1$ for $\|T\|$-a.e. $x$.
  \item[(c)]  An $m$-dimensional integral current $T$ supported in $\partial \bB_{\bar R} (p) \subset \R^{m+n}$ is a {\em spherical cross-section of an area minimizing cone} if ${p\cone T}$ is area minimizing. 
\end{itemize}
\end{definition}

In what follows, given an integer rectifiable current $T$, we denote by $\reg (T)$ the subset of $\supp (T)\setminus \supp (\partial T)$ consisting of those points $x$ for which there is a neighborhood $U$ such that $T\res U$ is a (constant multiple of) a $C^2$ submanifold. Correspondingly, $\sing (T)$ is
the set $\supp (T)\setminus (\supp (\partial T)\cup \reg (T))$. Observe that $\reg (T)$ is relatively open in $\supp (T) \setminus \supp (\partial T)$
and thus $\sing (T)$ is relatively closed.
The main result of this and the works \cite{DSS2,DSS4} is then the following

\begin{theorem}\label{t:finale}
Let $m=2$ and $T$ be as in (a), (b) or (c) of Definition \ref{d:semicalibrated}. Assume in addition that
$\Sigma$ is of class $C^{3,\eps_0}$ (in case (a) and (b)) and $\omega$ of class $C^{2,\eps_0}$ (in case (b)) for some positive $\eps_0$. Then $\sing (T)$
is discrete.
\end{theorem}

Clearly Chang's result is covered by case (a). As for the case of special Lagrangian cycles considered by Bellettini and Rivi\`ere in
\cite{BeRi} observe that they form a special subclass of both (b) and (c). Indeed these cycles arise as spherical cross-sections of $3$-dimensional special L.agrangian cones: as such they are then spherical cross sections of area minimizing cones but they are also semicalibrated by a specific smooth form
on $\mathbb S^5$.

Following the Almgren-Chang program, Theorem \ref{t:finale} will be established through a suitable ``blow-up argument'' which requires
several tools. The first important tool is the theory of multiple valued functions, for which we
will use the results and terminology of the papers \cite{DS1,DS2}. The second tool is a suitable approximation result for area minimizing currents with graphs of multiple valued functions, which for the case at hand has been established in the preceding
note \cite{DSS2}. 
The last tool is the so-called ``center manifold'': this will be constructed in the present paper, whereas the final argument for
Theorem \ref{t:finale} will then be given in \cite{DSS4}. We note in passing that all our arguments use heavily the uniqueness
of tangent cones for $T$. This result is a, by now classical, theorem of White for area minimizing $2$-dimensional currents in the euclidean
space, cf. \cite{Wh}. Chang extended it to case (a) in the appendix of \cite{Chang}, whereas Pumberger and Rivi\`ere covered case (b) in \cite{PuRi}. 
A general derivation of these results for a wide class of almost minimizers has been given in \cite{DSS1}: the theorems in there
cover, in particular, all the cases of Definition \ref{d:semicalibrated}.

The proof of Theorem \ref{t:finale} is based, as in \cite{Chang}, on an induction statement, cf. Theorem \ref{t:induttivo} below. This and the next paper \cite{DSS4} can be thought as the two main steps in its proof. For this reason, before detailing the construction of the branched center manifold, which is the main object of this note, we will state Theorem \ref{t:induttivo},
show how Theorem \ref{t:finale} follows from it and give a rough outline of the contributions of this and the next note
\cite{DSS4}.

\subsection{Acknowledgments}  The research of Camillo De Lellis and Luca Spolaor has been supported by the ERC grant RAM (Regularity for Area Minimizing currents), ERC 306247.

\section{Preliminaries and the main induction statement}

\subsection{Basic notation and first main assumptions}
For the notation concerning submanifolds
$\Sigma\subset \R^{2+n}$ we refer to \cite[Section 1]{DS3}. With $\bB_r (p)$ and $B_r (x)$ we denote, respectively, the open ball with radius $r$ and center $p$ in $\mathbb R^{2+n}$ and the open ball with radius $r$ and center $x$ in $\mathbb R^2$.  $\bC_r (p)$ and $\bC_r (x)$ will always denote the cylinder $B_r (x)\times \mathbb R^n$, where $p=(x,y)\in \R^2\times \R^n$. We will often need to consider cylinders whose bases are parallel to other $2$-dimensional planes, as well as balls in $m$-dimensional affine planes. We then introduce the notation $B_r (p, \pi)$ for $\bB_r (p) \cap (p+\pi)$ and $\bC_r (p, \pi)$ for $B_r (p, \pi) + \pi^\perp$. 
$e_i$ will denote the unit vectors in the standard basis,
$\pi_0$ the (oriented) plane $\R^2\times \{0\}$ and $\vec \pi_0$
the $2$-vector $e_1\wedge e_2$ orienting it.
Given an $m$-dimensional plane $\pi$, we denote by $\p_\pi$ and $\p_\pi^\perp$ the
orthogonal projections onto, respectively, $\pi$ and its orthogonal complement $\pi^\perp$. 
For what concerns integral currents we use the definitions and the notation of \cite{Sim}. Since $\pi$ is used recurrently for $2$-dimensional planes, the $2$-dimensional area of the unit circle in $\R^2$ will be denoted by $\omega_2$. 

By \cite[Lemma 1.1]{DSS2} in case (b) we can assume, without loss of generality, that the ambient manifold $\Sigma$ coincides with
the euclidean space $\R^{2+n}$. In the rest of the paper we will therefore always make the following

\begin{ipotesi}\label{ipotesi_base}
$T$ is an integral current of dimension $2$ with bounded support and it satisfies one of the three conditions (a), (b) or (c) in Definition \ref{d:semicalibrated}. Moreover
\begin{itemize}
\item In case (a),
$\Sigma\subset\R^{2+n}$ is a $C^{3, \eps_0}$
submanifold of dimension $2 + \bar n = 2 + n - l$, which is the graph of an entire
function $\Psi: \R^{2+\bar n}\to \R^l$ and satisfies the bounds
\begin{equation}\label{e:Sigma}
\|D \Psi\|_0 \leq c_0 \quad \mbox{and}\quad \bA := \|A_\Sigma\|_0
%\| D^2\bPsi \|_{L^\infty}
\leq c_0,
\end{equation}
where $c_0$ is a positive (small) dimensional constant and $\eps_0\in ]0,1[$. 
\item In case (b) we assume that $\Sigma = \R^{2+n}$ and that the semicalibrating form $\omega$ is $C^{2, \eps_0}$. 
\item In case (c) we assume that $T$ is supported in $\Sigma = \partial \bB_R (p_0)$ for some $p_0$ with $|p_0|=R$, so that $0\in \partial \bB_R (p_0)$. We assume also that $T_0 \partial \bB_R (p_0)$ is $\R^{2+n-1}$ (namely $p_0= (0, \ldots, 0, \pm |p_0|)$ and we let $\Psi: \R^{2+n-1}\to \R$ be a smooth extension to the whole space of the function which describes $\Sigma$ in $\bB_2 (0)$. We assume then that \eqref{e:Sigma} holds, which is equivalent to the requirement that $R^{-1}$ be sufficiently small.
\end{itemize}
\end{ipotesi}

In addition, since the conclusion of Theorem~\ref{t:finale} is local, by \cite[Proposition 0.4]{DSS1}
we can also assume to be always in the following situation.

\begin{ipotesi}\label{ipotesi_base_2}
In addition to Assumption \ref{ipotesi_base} we assume the following:
\begin{itemize}
\item[(i)] $\partial T \res \bC_2 (0, \pi_0)=0$;
\item[(ii)] $0\in \supp (T)$ and the tangent cone
at $0$ is given by $\Theta (T, 0) \a{\pi_0}$ where $\Theta (T, 0)\in \N\setminus \{0\}$;
\item[(iii)] $T$ is irreducible in any neighborhood $U$ of $0$ in the following sense:
it is not possible to find $S$, $Z$ non-zero integer rectifiable
currents in $U$ with $\partial S = \partial Z =0$ (in $U$), $T= S+Z$ and $\supp (S)\cap \supp (Z)=\{0\}$.
\end{itemize}
\end{ipotesi}

In order to justify point (iii), observe that we can argue as in the proof of \cite[Theorem 3.1]{DSS1}: assuming that in a certain neighborhood $U$ 
there is a decomposition $T= S+Z$ as above, it follows from \cite[Proposition 2.2]{DSS1} that both $S$ and $Z$ fall in one of the classes of Definition \ref{d:semicalibrated}. In turn this implies that $\Theta (S, 0), \Theta (Z, 0)\in \N\setminus \{0\}$ and thus $\Theta (S, 0) < \Theta (T, 0)$.
We can then replace $T$ with either $S$ or $Z$. Let $T_1=S$ and argue similarly if it is not irreducible: obviously we can apply the argument above one more time and find a $T_2$ which satisfies all the requirements and has $0<\Theta (T_2, 0) < \Theta (T_1, 0)$. This process must stop after at most $N = \Theta (T, 0)$ steps: the final current is then necessarily irreducible.

\subsection{Branching model} We next introduce an object which will play a key role in the rest of our work, because it is the basic local model of the
singular behavior of a $2$-dimensional area minimizing current: for each positive natural number $Q$ we will denote by $\gira_{Q,\rho}$ the flat Riemann surface which is a disk with a conical singularity, in the origin, of angle $2\pi Q$ and radius $\rho>0$. More precisely we have

\begin{definition}\label{d:Riemann_surface}
$\gira_{Q,\rho}$ is topologically an open $2$-dimensional disk, which we identify with the topological space $\{(z,w)\in \mathbb C^2 : w^Q=z, |z|<\rho\}$. For each $(z_0, w_0)\neq 0$ in $\gira_{Q,\rho}$ we consider the connected component $\disco (z_0, w_0)$ of $\gira_{Q,\rho}\cap \{(z,w):|z-z_0| < |z_0|/2\}$ which contains $(z_0, w_0)$. We then consider the smooth manifold given by the atlas 
\[
\{(\disco (z,w)), (x_1, x_2)): (z,w)\in \gira_{Q,\rho} \setminus \{0\}\}\, ,
\] 
where $(x_1, x_2)$ is the function which gives the real and imaginary part of the first complex coordinate of a generic point of $\gira_{Q,\rho}$. On such smooth manifold we consider the following flat Riemannian metric: on each $\disco (z,w)$ with the chart $(x_1, x_2)$ the metric tensor is the usual euclidean one $dx_1^2+dx_2^2$. Such metric will be called the {\em canonical flat metric} and denoted by $e_Q$.
\end{definition}

When $Q=1$ we can extend smoothly the metric tensor to the origin and we obtain the usual euclidean $2$-dimensional disk.
For $Q>1$ the metric tensor does not extend smoothly to $0$, but we can nonetheless complete the induced geodesic distance on $\gira_{Q,\rho}$ in a neighborhood of $0$: for $(z,w)\neq 0$ the distance to the origin will then correspond to $|z|$. The resulting metric space is a well-known object in the literature, namely a flat Riemann surface with an isolated conical singularity at the origin (see for instance \cite{Zorich}). Note that for each $z_0$ and $0 <r\leq \min \{|z_0|, \rho -|z_0|\} $ the set $\gira_{Q, \rho}\cap \{|z-z_0|<r\}$ consists then of $Q$ nonintersecting $2$-dimensional disks, each of which is a geodesic ball of $\gira_{Q, \rho}$ with radius $r$ and center $(z_0, w_i)$ for some $w_i\in \C$ with $w_i^Q = z_0$. We then denote each of them by $B_r (z_0,w_i)$ and treat it as a standard disk in the euclidean $2$-dimensional plane (which is correct from the metric point of view). We use however the same notation for the distance disk $B_r (0)$, namely for the set $\{(z,w): |z|<r\}$, although the latter is {\em not isometric} to the standard euclidean disk. Since this might be create some ambiguity, we will use the specification $\R^2 \supset B_r (0)$ (or $\R^2 \supset B_r$) when referring to the standard disk in $\R^2$. 

\subsection{Admissible $Q$-branchings}
When one of (or both) the parameters $Q$ and $\rho$ are clear from the context, the  corresponding subscript (or both) will be omitted.
We will always treat each point of $\gira$ as an element of $\C^2$, mostly using $z$ and $w$ for the horizontal and vertical complex coordinates. Often $\C$ will be identified with $\R^2$ and thus the coordinate $z$ will be treated as a two-dimensional real vector, avoiding the more cumbersome notation $(x_1, x_2)$.

\begin{definition}[$Q$-branchings]\label{d:admissible}
Let $\alpha\in ]0,1[$, $b>1$, $Q\in \N\setminus \{0\}$ and $n\in \N\setminus \{0\}$. 
An admissible $\alpha$-smooth and $b$-separated $Q$-branching in $\R^{2+n}$ (shortly a $Q$-branching) is the graph
\begin{equation}
\gr (u) := \{(z, u(z,w)): (z,w) \in \gira_{Q, 2\rho}\} \subset \R^{2+n}\, 
\end{equation}
of a map $u: \gira_{Q,2\rho} \to \mathbb R^n$ satisfying the following assumptions.
For some constant $C_i>0$ we have
\begin{itemize}
\item $u$ is continuous, $u\in C^{3,\alpha}$ on $\gira_{Q, 2\rho} \setminus \{0\}$ and $u(0)=0$;
\item $|D^j u (z,w)|\leq C_i |z|^{1-j+\alpha}$ $\forall (z,w)\neq 0$ and $j\in \{0,1,2,3\}$;
\item $[D^3 u]_{\alpha, B_r (z,w)} \leq C_i |z|^{-2}$ for every $(z,w)\neq 0$ with $|z|=2r$;
\item If $Q>1$, then there is a positive constant $c_s\in ]0,1[$ such that
\begin{equation}\label{e:separation}
\min \{|u (z,w)-u(z,w')|: w\neq w'\} \geq 4 c_s |z|^b \qquad \mbox{for all $(z,w)\neq 0$.}
\end{equation}
\end{itemize}
The map $\Phii (z,w):= (z, u (z,w))$ will be called the {\em graphical parametrization} of the $Q$-branching.
\end{definition}

Any $Q$-branching as in the Definition above is an immersed disk in $\R^{2+n}$ and can be given a natural structure as integer rectifiable current, which will be denoted by $\bG_u$. For $Q=1$ a map $u$ as in Definition \ref{d:admissible} is a (single valued) $C^{1,\alpha}$ map $u: \R^2 \supset B_{2\rho} (0)\to \R^n$. Although the term branching is not appropriate in this case, the advantage of our setup is that $Q=1$ will not be a special case in the induction statement of Theorem \ref{t:induttivo} below.  Observe that for $Q>1$
the map $u$ can be thought as a $Q$-valued map $u: \R^2 \supset B_{2\rho} (0) \to \Iqs$, setting $u(z)= \sum_{(z,w_i)\in \gira} \a{u(z,w_i)}$ for $z\neq 0$ and $u (0) = Q\a{0}$. The notation $\gr (u)$ and $\bG_u$ is then coherent with the corresponding objects defined in \cite{DS2} for general $Q$-valued maps.

\subsection{The inductive statement} Before coming to the key inductive statement, we need to introduce some more terminology.

\begin{definition}[Horned Neighborhood]\label{d:horned} Let $\gr (u)$ be a $b$-separated $Q$-branching.
For every $a>b$ we define the {\em horned neighborhood} $\bV_{u,a}$ of $\gr (u)$ to be
\begin{equation}
\bV_{u, a} := \{(x,y)\in \R^2\times \R^n: \exists (x,w) \in \gira_{Q, 2\rho} \mbox{ with } |y-u(x,w)|< c_s |x|^a\}\, ,
\end{equation}
where $c_s$ is the constant in \eqref{e:separation}.
\end{definition}

\begin{definition}[Excess]\label{d:excess}
Given an $m$-dimensional current $T$ in $\mathbb R^{m+n}$ with finite mass, its {\em excess} in the ball $\bB_r (x)$ and in the cylinder $\bC_r (p, \pi')$ with respect to the $m$-plane $\pi$ are
\begin{align}
\bE(T,\bB_r (p),\pi) &:= \left(2\omega_m\,r^m\right)^{-1}\int_{\bB_r (p)} |\vec T - \vec \pi|^2 \, d\|T\|\\
\bE (T, \bC_r(p, \pi'), \pi) &:= \left(2\omega_m\,r^m\right)^{-1}\int_{\bC_r (p, \pi')} |\vec T - \vec \pi|^2 \, d\|T\|\, .
\end{align}
For cylinders we omit the third entry when $\pi=\pi'$, i.e. $\bE (T, \bC_r (p, \pi)) := \bE (T, \bC_r (p, \pi), \pi)$.
In order to define the spherical excess we consider $T$ as in Assumption \ref{ipotesi_base} and we say that
$\pi$ {\em optimizes the excess} of $T$ in a ball $\bB_r (x)$ if
\begin{itemize}
\item In case (b)
\begin{equation}\label{e:optimal_pi}
\bE(T,\bB_r (x)):=\min_\tau \bE (T, \bB_r (x), \tau) = \bE(T,\bB_r (x),\pi);
\end{equation} 
\item In case (a) and (c) $\pi\subset T_x \Sigma$ and 
\begin{equation}\label{e:optimal_pi_2}
\bE(T,\bB_r (x)):=\min_{\tau\subset T_x \Sigma} \bE (T, \bB_r (x), \tau) = \bE(T,\bB_r (x),\pi)\, .
\end{equation} 
\end{itemize}
\end{definition}
Note in particular that, in case (a) and (c), $\bE (T, \bB_r (x))$ differs from the quantity defined in \cite[Definition 1.1]{DS5}, where, although $\Sigma$ does not coincide with the ambient euclidean space, $\tau$ is allowed to vary among {\em all} planes, as in case (b). Thus a notation more consistent with that of \cite{DS5} would be, in case (a) and (c), $\bE^\Sigma (T, \bB_r (x))$. However, the difference is a minor one and we prefer to keep our notation simpler.

Our main induction assumption is then the following

\begin{ipotesi}[Inductive Assumption]\label{induttiva}
$T$ is as in Assumption \ref{ipotesi_base} and \ref{ipotesi_base_2}. 
For some constants $\bar Q\in \mathbb N\setminus \{0\}$ and $0<\alpha< \frac{1}{2\bar Q}$ there is an $\alpha$-admissible $\bar Q$-branching
$\gr (u)$ with $u: \gira_{\bar{Q}, 2}\to \R^n$ such that
\begin{itemize}
\item[(Sep)] If $\bar Q>1$, $u$ is $b$-separated for some $b>1$; a choice of some $b>1$ is fixed also in the case
$\bar Q =1$, although in this case the separation condition is empty.
\item[(Hor)] $\supp (T) \subset \bV_{u,a}\cup \{0\}$ for some $a>b$;
\item[(Dec)] There exist $\gamma >0$ and a $C_i>0$ with the following property. Let $p = (x_0,y_0)\in \supp (T)\cap \bC_{\sqrt{2}} (0)$ and $4 d := |x_0|>0$, let $V$ be the connected component of 
$\bV_{u,a}\cap \{(x,y): |x-x_0| < d\}$ containing $p$ and let $\pi (p)$ be the plane tangent to $\gr (u)$ at the only point of the form $(x_0, u(x_0, w_i))$ which is contained in $V$. Then
\begin{equation}\label{e:effetto_energia}
\bE (T\res V,\bB_\sigma (p), \pi (p)) \leq C_i^2 d^{2\gamma -2} \sigma^2 \qquad \forall \sigma\in \left[{\textstyle{\frac{1}{2}}}d^{(b+1)/2}, d\right]\, .
\end{equation}
\end{itemize}
\end{ipotesi}

The main inductive step is then the following theorem, where we denote by $T_{p,r}$ the rescaled current $(\iota_{p,r})_\sharp T$, through the map $\iota_{p,r} (q) := (q-p)/r$.

\begin{theorem}[Inductive statement]\label{t:induttivo}
Let $T$ be as in Assumption \ref{induttiva} for some $\bar Q = Q_0$. Then,
\begin{itemize}
\item[(a)] either $T$ is, in a neighborhood of $0$, a $Q$ multiple of a $\bar{Q}$-branching $\gr (v)$;
\item[(b)] or there are $r>0$ and $Q_1>Q_0$ such that $T_{0, r}$ satisfies Assumption \ref{induttiva} with $\bar Q = Q_1$.
\end{itemize}
\end{theorem}

Theorem \ref{t:finale} follows then easily from Theorem \ref{t:induttivo} and \cite{DSS1}. 

\subsection{Proof of Theorem \ref{t:finale}} As already mentioned, without loss of generality we can assume that Assumption \ref{ipotesi_base} holds, cf. \cite[Lemma 1.1]{DSS1} (the bounds on $\bA$ and $\Psi$ can be achieved by a simple scaling argument). Fix now a point $p$
in $\supp (T)\setminus \supp (\partial T)$. Our aim is to show that $T$ is regular in a punctured neighborhood of $p$. Without loss of generality we can assume that $p$ is the origin. Upon suitably decomposing $T$ in some neighborhood of $0$ we can easily assume that (Sep) in Assumption \ref{induttiva} holds, cf. the argument of Step 4 in the proof of \cite[Theorem 3.1]{DSS1}. Thus, upon suitably rescaling and rotating $T$ we can assume that $\pi_0$ is the unique tangent cone to $T$ at $0$, cf. \cite[Theorem 3.1]{DSS1}. In fact, by \cite[Theorem 3.1]{DSS1} $T$ satisfies Assumption \ref{induttiva} with $\bar{Q}=1$: it suffices to chose $u\equiv 0$ as admissible smooth branching. If $T$ were not regular in any punctured neighborhood of $0$, we could then apply Theorem \ref{t:induttivo} inductively to find a sequence of rescalings $T_{0, \rho_j}$ with $\rho_j\downarrow 0$ which satisfy Assumption \ref{induttiva} with $\bar Q=Q_j$ for some strictly increasing sequence of integers.
It is however elementary that the density $\Theta (0, T)$ bounds $Q_j$ from above (see for instance the argument of the next section leading to Lemma \ref{l:density}), which is a contradiction.

\section{The branched center manifold}\label{s:center_manifold}

\subsection{The overall approach to Theorem \ref{t:induttivo}}
From now on we fix $T$ satisfying Assumption \ref{induttiva}. Observe that, without loss of generality, we are always free to replace $T$ by $T_{0,r}$ with $r$ sufficiently small (and ignore whatever portion falls outside $\bC_2 (0)$). Indeed we will do this several times. Hence, if we can prove that something holds in a sufficiently small neighborhood of $0$, then we can assume, without loss of generality, that it holds on $\bC_2 (0)$.
For this reason we can assume that the constant $C_i$ in Definition \ref{d:admissible} and Assumption \ref{induttiva} is as small as we want.
In turns this implies that there is a well-defined orthogonal projection $\P: \bV_{u,a}\cap \bC_1 \to \gr (u)\cap \bC_2$, which is a $C^{2,\alpha}$ map. 

By the constancy theorem, $(\P_\sharp (T\res \bC_1))\res \bC_{1/2}$ coincides with the current $Q \bG_u \res \bC_{1/2}$ (again,
we are assuming $C_i$ in Definition \ref{d:admissible} sufficiently small), where $Q\in \Z$. If $Q$ were $0$, condition (Dec) in Assumption \ref{induttiva} and a simple covering argument would imply that $\|T\| (\bC_{1/2} (0))\leq C_0 C_i^2$, where $C_0$ is a constant depending on $n$ and $\bar Q$. In particular, when $C_i$ is sufficiently small, this would violate, by the monotonicity formula, the assumption $0\in \supp (T)$. Thus $Q\neq 0$. On the other hand condition (Dec) in Assumption \ref{induttiva} implies also that $Q$ must be positive (again, provided $C_i$ is smaller than a geometric constant).

Now, recall from \cite[Theorem 3.1]{DSS1} that the density $\Theta (p, T)$ is a positive integer at any $p\in \supp (T)\setminus \supp (\partial T)$. Moreover, the rescaled currents $T_{0,r}$ converge to $\Theta (0,T) \a{\pi_0}$. It is easy to see that the rescaled currents $(\bG_u)_{0,r}$ converge to $\bar{Q} \a{\pi_0}$ and that
$(\bP_\sharp T)_{0,r}$ converges to $\Theta (0, T) \a{\pi_0}$. We then conclude that $\Theta (0, T) = \bar{Q} Q$. 

We summarize these conclusions in the following lemma, where we also claim an additional important bound on the density
of $T$ outside $0$, which will be proved later.

\begin{lemma}\label{l:density}
Let $T$ and $u$ be as in Assumption \ref{induttiva} for some $\bar  Q$ and sufficiently small $C_i$. Then the nearest point projection $\P: \bV_{u,a} \cap \bC_1 \to \gr (u)$ is a well-defined $C^{2, \alpha}$ map. In addition there is $Q\in \N\setminus \{0\}$ such that $\Theta (0, T) = Q \bar{Q}$ and the unique tangent cone to $T$ at $0$ is $Q\bar{Q} \a{\pi_0}$. Finally,
after possibly rescaling $T$, $\Theta (p, T) \leq Q$ for every $p\in \bC_2\setminus \{0\}$ and, for every $x\in B_2 (0)$, each
connected component of $(\{x\}\times \R^n) \cap \bV_{u,a}$ contains at least one point of $\supp (T)$.
\end{lemma}

Since we will assume during the rest of the paper that the above discussion applies, we summarize the relevant conclusions in the following

\begin{ipotesi}\label{piu_fogli}
$T$ satisfies Assumption \ref{induttiva} for some $\bar{Q}$ and with $C_i$ sufficiently small. $Q\geq 1$
is an integer, $\Theta (0, T) = Q \bar{Q}$ and $\Theta (p, T) \leq Q$ for all $p\in \bC_2\setminus \{0\}$. 
\end{ipotesi}

%If, for some $r>0$, $\Theta (T, p) =Q$ for every $p\in \bC_r\setminus \{0\}$, then Allard's regularity theorem would imply that $T$ is %regular in $\bC_r \setminus \{0\}$: in fact the estimates of Allard's theorem and standard elliptic regularity implies easily that conclusion %(a) of Theorem \ref{t:induttivo} holds: the details of the proof will be given in \cite{DSS4}. Hence, in order to show Theorem %\ref{t:induttivo} we can assume, without loss of generality, that there is a sequence of points $p_k\downarrow 0$ with $1 \leq \Theta (p_k, %T) <Q$. In particular $Q\geq 2$. 
The overall plan to prove Theorem \ref{t:induttivo} is then the following:
\begin{itemize}
\item[(CM)] We construct first a branched center manifold, i.e. a second admissible smooth branching $\phii$ on $\gira_{\bar{Q}}$, and a corresponding $Q$-valued map $N$ defined on the normal bundle of $\gr (\phii)$, which approximates $T$ with a very high degree of accuracy (in particular more accurately than $u$) and whose average $\etaa\circ N$ is very small;
\item[(BU)] Assuming that alternative (a) in Theorem \ref{t:induttivo} does not hold, we study the asymptotic behavior of $N$ around $0$ and use it to build a new admissible smooth branching $v$ on some $\gira_{k \bar{Q}}$ where $k\geq 2$ is a factor of $Q$: this map will then be the one sought in alternative (b) of Theorem \ref{t:induttivo} and a suitable rescaling of $T$ will lie in a horned neighborhood of its graph.
\end{itemize}
The first part of the program is the one achieved in this paper, whereas the second part will be completed in \cite{DSS4}: in the latter paper we then give the proof of Theorem \ref{t:induttivo}. Note that, when $Q=1$, from (BU) we will conclude that alternative (a) necessarily holds: this will be a simple corollary of the general case, but we observe that it could also be proved resorting to the classical Allard's regularity theorem.

\subsection{Smallness condition} In several occasions we will need that the ambient manifold $\Sigma$ is suitably flat and that the excess of the current $T$ is suitably small. This can, however, be easily achieved after scaling. More precisely we have the following

\begin{lemma}\label{l:piccolezza}
Let $T$  be as in the Assumptions \ref{induttiva}.
After possibly rescaling, rotating and modifying $\Sigma$ outside $\bC_2 (0)$ we can assume that, in case (a) and (c) of Definition \ref{d:semicalibrated},
\begin{itemize}
\item[(i)] $\Sigma$ is a complete submanifold of $\R^{2+n}$;
\item[(ii)] $T_0 \Sigma = \R^{2+\bar{n}}\times \{0\}$ and, $\forall p\in \Sigma$, $\Sigma$ is the graph of a $C^{3,\eps_0}$ map $\Psi_p: T_p \Sigma \to (T_p\Sigma)^\perp$.
\end{itemize}
Under these assumptions, we denote by $\bc$ and $\bmo$ the following quantities
\begin{align}
&\bc := \sup \{\|D\Psi_p\|_{C^{2,\eps_0}}:p\in \Sigma\}\qquad \mbox{in the cases (a) and (c) of Definition \ref{d:semicalibrated}}\\
&\bc :=\|d\omega\|_{C^{1,\eps_0}} \qquad \qquad\;\;\qquad\qquad\mbox{in case (b) of Definition \ref{d:semicalibrated}}\\
&\bmo:= \max \left\{ \bc^2, \bE (T, \bC_2, \pi_0), C_i^2, c_s^2\right\}\, ,
\end{align}
where $C_i$ and $c_s$ are the constants appearing in Definition \ref{d:admissible} and Assumption \ref{induttiva}. Then, for any $\eps_2>0$,
after possibly rescaling the current by a large factor, we can assume
\begin{equation}\label{e:smallness}
\bmo \leq \eps_2\, .
\end{equation}
\end{lemma}

We postpone the proof of this (simple) technical lemma to a later section.

\subsection{Conformal parametrization} In order to carry on the plan outlined in the previous subsection, it is convenient to use  parametrizations of $Q$-branchings which are not graphical but instead satisfy a suitable conformality property. To simplify our notation, the map $\Psi_0$ will be simply denoted by $\Psi$.

If we remove the origin, any admissible $Q$-branching is a Riemannian submanifold of $\R^{2+n}\setminus \{0\}$: this gives a Riemannian tensor $g := \Phii^\sharp e$ (where $e$ denotes the euclidean metric on $\R^{2+n}$) on the punctured disk $\gira_{Q, 2\rho}\setminus \{0\}$. Note that in $(z,w)$ the difference between the metric tensor $g$ and the canonical flat metric $e_Q$ can be estimated by (a constant times) $|z|^{2\alpha}$: thus, as it happens for the canonical flat metric $e_Q$, when $Q>1$ it is not possible to extend the metric $g$ to the origin. However, using well-known arguments in differential geometry, we can find a conformal map from $\gira_{Q, r}$ onto a neighborhood of $0$ which maps the conical singularity of $\gira_{Q,r}$ in the conical singularity of the $Q$-branching. In fact, we need the following accurate estimates for such a map.

\begin{proposition}[Conformal parametrization]\label{p:conformal}
Given an admissible $b$-separated $\alpha$-smooth $Q$-branching $\gr (u)$ with $\alpha <1/(2Q)$
there exist a constant $C_0 (Q, \alpha)>0$, a radius $r>0$ and functions $\Psii\colon \gira_{Q,r} \to \gr (u)$
and  $\lambda\colon \gira_{Q,r} \to \R_+$
such that
\begin{itemize}
\item[(i)] $\Psii$ is a homeomorphism of $\gira_{Q,r}$ with a neighborhood of $0$ in $\gr (u)$;
\item[(ii)] $\Psii\in C^{3,\alpha} (\gira_{Q,r}\setminus \{0\})$, with the estimates
\begin{align}
%|\Psii (z,w) -(z,0)|\leq & C_0 C_i |z|^{1+\alpha}\, ,\label{e:conformal1}\\
|D^l \big(\Psii(z,w) - (z,0)\big)| \leq & C_0 C_i |z|^{1+\alpha-l} \qquad \mbox{for $l=0,\dots,3$, $z\neq 0$}\, ,\label{e:conformal2}\\
[D^3 \Psii]_{\alpha, B_r (z,w)} \leq & C_0 C_i |z|^{-2} \qquad \mbox{for $z\neq 0$ and $r = |z|/2$}\, ;\label{e:conformal3}
\end{align}
\item[(iii)] $\Psii$ is a conformal map with conformal factor $\lambda$, namely, if we denote by $e$ the ambient euclidean metric
in $\R^{2+n}$ and by $e_Q$ the canonical euclidean metric of $\gira_{Q,r}$, 
\begin{equation}\label{e:conformal}
g:= \Psii^\sharp e = \lambda\, e_Q\, \qquad \mbox{on $\gira_{Q,r}\setminus \{0\}$.}
\end{equation}
\item[(iv)] The conformal factor $\lambda$ satisfies
\begin{align}
|D^l (\lambda -1) (z,w)| \leq &C_0 C_i |z|^{2\alpha-l} \qquad \mbox{for $l=0,1,\ldots, 2$}\label{e:conformal4}\\
[D^2 \lambda]_{\alpha, B_r (z,w)} \leq & C_0 C_i |z|^{\alpha-2}\qquad \mbox{for $z\neq 0$ and $r = |z|/2$}\, .\label{e:conformal5}
\end{align}
\end{itemize}
\end{proposition}

A proof of Proposition \ref{p:conformal} is given in Appendix \ref{a:conformal}.

\begin{definition}
A map $\Psii$ as in Proposition \ref{p:conformal} will be called a {\em conformal parametrization} of an admissible $Q$-branching.  
\end{definition}

\subsection{The center manifold and the approximation} We are finally ready to state the main theorem of this note. 

\begin{theorem}[Center Manifold Approximation]\label{t:cm_final}
Let $T$ be as in Assumptions~\ref{induttiva} and \ref{piu_fogli} and assume in addition that the conclusions of Lemma \ref{l:piccolezza} apply (in particular we might need to replace $T$ by $T_{0,r}$ for $r$ sufficiently small).
Then there exist $\eta_0, \gamma_0, r_0,C>0$, $b>1$, an admissible $b$-separated $\gamma_0$-smooth $\bar Q$-branching $\cM$, a corresponding conformal parametrization $\Psii: \gira_{\bar Q,2} \to \cM$ 
and a $Q$-valued map $\NN: \gira_{\bar Q,2} \to \I{Q}(\R^{2+n})$ with the following properties:
\begin{itemize}
\item[(i)] $\bar Q Q = \Theta (T, 0)$ and 
\begin{align}
|D(\Psii (z,w) -(z,0))|\leq &\, C \bmo^{1/2} |z|^{\gamma_0}\\ 
|D^2 \Psii (z,w)|+ |z| |D^3 \Psii (z,w)| \leq &\, C \bmo^{\sfrac{1}{2}} |z|^{\gamma_0-1}\, ;
\end{align}
in particular, if we denote by $A_\cM$ the second fundamental form of $\cM\setminus \{0\}$,
\[
|A_\cM (\Psii (z,w))| + |z| |D_{\cM} A_\cM (\Psii (z,w))| \leq C \bmo^{\sfrac{1}{2}} |z|^{\gamma_0-1}\, .
\]
\item[(ii)] $\NN\,_i (z,w)$ is orthogonal to the tangent plane, at $\Psii (z,w)$, to $\cM$.
\item[(iii)] If we define $S:= T_{0,r_0}$, then
% \[
% N_i(x) \perp T_x \im(\Phii) \quad \text{and} \quad x + N_i(x) \in \Sigma\quad\forall\; x \in \im(\Phii)
% \]
$\supp (S)\cap \bC_1\setminus \{0\}$ is contained in a suitable horned neighborhood of the $\bar Q$-branching $\mathcal{M}$, where the orthogonal projection
$\P$ onto it is well-defined. Moreover, for every $r\in]0,1[$ we have
\begin{align}\label{e:Ndecay}
\|\NN\vert_{B_r}\|_0 + \sup_{p\in \supp (S) \cap \P^{-1} (\Psii(B_r))} |p- \P (p)|\leq C \bmo^{\sfrac{1}{4}} r^{1+\sfrac{\gamma_0}{2}}\, .
\end{align}
\item[(iv)] If we define
\[
\bD (r) := \int_{B_r} |D\NN|^2
\quad \text{and}\quad
\bH (r) := \int_{\de B_r} |\NN|^2\, ,
\]
\[
\bF(r) := \int_0^r\frac{\bH(t)}{t^{2-\,\gamma_0}}\,dt
\quad\text{and}\quad \bLambda(r):= \bD(r)+\bF(r)\, ,
\]
then the following estimates hold for every $r\in ]0,1[$:
\begin{align}
\Lip (\NN\vert_{B_r}) \leq & C \min\{\bLambda^{\eta_0}(r), \bmo^{\eta_0}r^{\eta_0}\} \label{e:Lip_N}\\
%\bD(r)\leq C\,\bmo\, r^{2+2\gamma_0} \label{e:Dirdecay}\\
\bmo^{\eta_0} \,\int_{B_r}  |z|^{\gamma_0 -1} |\etaa \circ \NN(z,w)| 
\leq & C\, \bLambda^{\eta_0}(r)\,\bD(r)+C\,\bF(r)\, .\label{e:media_pesata}
\end{align}
\item[(v)] Finally, if we set
\[
\FF(z,w):= \sum_i\a{\Psii(z,w) + \NN\,_i(z,w)}\, ,
\]
then
\begin{align}
\|S-\bT_\FF\|\big(\P^{-1}(\Psii(B_r))\big) \leq  & C\,\bLambda^{\eta_0}(r)\,\bD(r)+ C\,\bF(r)\, . \label{e:diff masse} 
\end{align}
\end{itemize}
\end{theorem}

The rest of this note is dedicated to prove the above theorem. We first outline how the center manifold is constructed. We then construct an approximating map $N$ taking values on its normal bundle. Finally we change coordinates using a conformal parametrization $\Psii$ and prove the above theorem for the map $\NN = N \circ \Psii$.

\section{Center manifold: the construction algorithm}

\subsection{Choice of some parameters and  smallness of some other constants}
As in \cite{DS4} the construction of the center manifold involves several parameters. We start by choosing three of them which
will appear as exponents of (two) lenghtscales in several estimates.

\begin{ipotesi}\label{esponenti} Let $T$ be as in Assumptions \ref{induttiva} and \ref{piu_fogli}, assume in addition that the conclusions of Lemma \ref{l:piccolezza} apply (we might therefore need to replace $T$ with $T_{0,r}$ for a sufficiently small $r$) and in particular recall the exponents
$\varepsilon_0,\alpha, b,a$ and $\gamma$ defined therein. 
We choose the positive exponents $\gamma_0$, $\beta_2$ and $\delta_1$ (in the given order) so that
\begin{align}
&\gamma_0 < \min \{\gamma, \alpha, a- b, b - \textstyle{\frac{b+1}{2}}, \log_2 \frac{6}{5}\}\, \label{e:cost_2}\\
&\beta_2 <\min \{\eps_0, \textstyle{\frac{\gamma_0}{4}}, \frac{a}{b} -1, \frac{\alpha}{2}, \frac{\beta_0 \gamma_0}{2}\} \qquad\quad\;\;b> \frac{1+b}{2} (1+\beta_2) &\label{e:cost_0}\\
&\beta_2 - 2\delta_1 \geq \textstyle{\frac{\beta_2}{3}}\qquad \beta_0 (2-2\delta_1) - 2\delta_1 \geq 2 \beta_2\label{e:cost_1}
\end{align}
(where $\beta_0$ is the constant of \cite[Theorem 5.2]{DSS2} and in this paper we assume it is smaller than $1/2$)
\end{ipotesi}

Having fixed $\gamma_0$, $\beta_2$ and $\delta_1$ we introduce five further parameters: $M_0, N_0, C_e, C_h$ and
$\eps_2$. We will impose several inequalities upon them, but following a very precise hierarchy, which ensures that all the conditions required in the remaining statements can be met. We will use the term ``geometric'' when such conditions depend only upon $\bar{n}, n, Q, \bar{Q}, \gamma_0, \beta_2$ and $\delta_1$, whereas we keep track of their dependence
on $M_0, N_0, C_e$ and $C_h$ using the notation $C=C(M_0), C(M_0, N_0)$ and so on. $\eps_2$ is always
the last parameter to be chosen: it will be small depending upon all the other constants, but constants will never depend upon it.

\begin{ipotesi}[Hierarchy of the parameters] In all the statements of the paper
 \begin{itemize}
  \item $M_0\geq 4$ is larger than a geometric constant and $N_0$ is a natural number larger than $C(M_0)$; one such condition is recurrent and we state it here:
\begin{equation}\label{e:cost_3}
\sqrt{2} M_0 2^{10-N_0} \leq 1\, ;
\end{equation}
  \item $C_e$ is larger than $C(M_0,N_0)$;
  \item $C_h$ is larger than $C(M_0,N_0,C_e)$;
  \item $\eps_2>0$ is smaller than $c(M_0,N_0,C_e,C_h)>0$.
 \end{itemize}
\end{ipotesi}

\subsection{Whitney decomposition of $\gira_{\bar Q, 2}$}\label{ss:whitney} From now on we will use $\gira$ for $\gira_{\bar Q, 2}$, since the
positive natural number $\bar Q$ is fixed for the rest of the paper. In this section we decompose $\gira\setminus \{0\}$ in a suitable way. More precisely, a closed subset $L$ of $\gira$ will be called a dyadic square if it is a connected component of $\gira \cap (H\times \mathbb \C)$ for some euclidean dyadic square $H = [a_1, a_1 + 2\ell]\times [a_2, a_2+2\ell] \subset \R^2 = \C$ with
\begin{itemize}
\item $\ell = 2^{-j}$, $j\in \mathbb N, j\geq 3$, and $a\in 2^{1-j} \Z^2$;
\item $H\subset [-1,1]^2$ and $0\not\in H$.
\end{itemize}
Observe that $L$ is truly a square, both from the topological and the metric point of view.
$2\ell$ is the sidelength of both $H$ and $L$. Note that $\gira\cap (H\times\C)$ consists then of $\bar Q$ distinct squares $L_1, \ldots, L_{\bar Q}$. $z_H := a + (\ell, \ell)$ is the center of the square $H$. Each $L$ lying over $H$ will then contain a point $(z_H, w_L)$, which is the center of $L$. Depending upon the context  we will then use $z_{L}$ rather than $z_H$ for the first (complex) component of the center of $L$.

The family of all dyadic squares of $\gira$ defined above will be denoted by $\sC$. 
We next consider, for $j\in \mathbb N$, the dyadic closed annuli 
\[
\cA_j := \gira \cap \big(([-2^{-j}, 2^{-j}]^2 \setminus ]-2^{-j-1}, 2^{-j-1}[^2)\times \C\big)\, .
\]
Each dyadic square $L$ of $\gira$ is then contained in exactly one annulus $\cA_j$ and we define $\d (L):= 2^{-j-1}$. Moreover $\ell (L) = 2^{-j-k}$ for some $k\geq 2$. 
We then denote by $\sC^{k,j}$ the family of those dyadic squares $L$ such that $L\subset \cA_j$ and $\ell (L) = 2^{-j-k}$. Observe that, for each $j\geq 1, k\geq 2$, $\sC^{k,j}$ is a covering of $\cA_j$ and that two elements of $\sC^{k,j}$ can only intersect at their boundaries. Moreover, any element of $\sC^{k,j}$ can intersect at most 8 other elements of $\sC^{k,j}$. Finally, we set
$\sC^k:= \bigcup_{j\geq 2} \sC^{k,j}$. Observe now that $\sC^k$ covers a punctured neighborhood of $0$ and that if $L\in \sC^k$, then 
\begin{itemize}
\item $L$ intersects at most $9$ other elements $J\in \sC^k$;
\item If $L\cap J \neq \emptyset$, then $\ell (J)/2 \leq \ell (L) \leq 2\ell (L)$ and $L\cap J$ is either a vertex or a side of the smallest among the two.
\end{itemize}
More in general if the intersection of two distinct elements $L$ and $J$ in $\sC=\bigcup_k \sC^k$ has nonempty interior, then one is contained in the other: if $L\subset J$ we then say that $L$ is a descendant of $J$ and $J$ an ancestor of $L$. If in addition $\ell (L) = \ell (J)/2$, then we say that $L$ is a son of $J$ and $J$ is the father of $L$. When $L$ and $J$ intersect only at their boundaries, we then say that $L$ and $J$ are adjacent.

Next, for each dyadic square $L$ we set $r_L := \sqrt{2} M_0 \ell (L)$. Note that, by our choice of $N_0$, we have that:
\begin{equation}\label{e:belli_lontani}
\mbox{if $L\in \sC^{k, j}$ and $k\geq N_0$, then $\bC_{64 r_L} (z_L) \subset \bC_{2^{1-j}}\setminus \bC_{2^{-2-j}}$.}
\end{equation}
In particular $\bV_{u,a}\cap \bC_{64 r_L} (z_L)$ consists of $\bar Q$ connected components and we can select the one containing $(z_L, u (z_L, w_L))$, which we will denote by $\bV_L$. We will then denote by $T_L$ the current $T\res \bV_L$. According to Lemma \ref{l:density},
$\bV_L \cap \{z_L\}\times \R^n$ contains at least one point of $\supp (T)$: we select any such point and denote it by $p_L=(z_L,y_L)$. Correspondingly we will denote by $\bB_L$ the ball $\bB_{64 r_L} (p_L)$. 

\begin{definition}\label{e:height}
The height of a current $S$ in a set $E$ with respect to a plane $\pi$ is given by
\begin{equation}
\bh (S, E, \pi) := \sup \{|\p_{\pi}^\perp (p-q)|:p,q\in \supp (S) \cap E\}\, .
\end{equation}
If $E = \bC_r (p, \pi)$ we will then set $\bh (S, \bC_r (p, \pi)) := \bh (S, \bC_r (p, \pi), \pi)$. If $E = \bB_r (p)$, $T$ is as in Assumption \ref{ipotesi_base} and $p\in \Sigma$ (in the cases (a) and (c) of Definition \ref{d:semicalibrated}), then
$\bh (T, \bB_r (p)) := \bh (T, \bB_r (p), \pi)$ where $\pi$ gives the minimal height among all $\pi$ for which $\bE (T, \bB_r (p), \pi) = \bE (T, \bB_r (p))$ (and such that $\pi\subset T_p \Sigma$ in case (a) and (c) of Definition \ref{d:semicalibrated}). Moreover, for such $\pi$ we say that it optimizes the excess and the height in $\bB_r (p)$.
\end{definition}

We are now ready to define the dyadic decomposition of $\gira\setminus 0$. 

\begin{definition}[Refining procedure]\label{d:refining} We build inductively the families of squares $\sS, \sW = \sW_e \cup \sW_h \cup \sW_n$ and their subfamilies $\sS^k = \sS\cap \sC^k$, $\sS^{k,j} = \sS\cap \sC^{k,j}$ and so on. First of all, we set $\sS^k = \sW^k = \emptyset$ for $k<N_0$.
For $k\geq N_0$ we use a double induction. Having defined $\sS^{k'}, \sW^{k'}$ for all $k' < k$ and $\sS^{k, j'}, \sW^{k, j'}$ for all $j' < j$, we pick all squares $L$ of $\sC^{k,j}$ which do not have any ancestor already assigned to $\sW$ and we proceed as follows.
\begin{itemize}
\item[(EX)] We assign $L$ to $\sW^{k,j}_e$ if
\begin{equation}\label{e:EX}
  \bE (T_{L},\bB_{L})>C_e\bmax_0 \d (L)^{2\gamma_0-2+2\delta_1}\ell (L)^{2-2\delta_1};
 \end{equation}
\item[(HT)] We assign $L$ to $\sW^{k,j}_h$ if we have not assigned it to $\sW_e$ and
 \begin{equation}\label{e:HT}
 \bh (T_L,\bB_L)> C_h \bmax_0^{\sfrac14}\d (L)^{\sfrac{\gamma_0}{2} - \beta_2}\ell(L)^{1+\beta_2};
 \end{equation}
\item[(NN)] We assign $L$ to $\sW^{k,j}_n$ if we have not assigned it to $\sW_e\cup \sW_h$ and it intersects a square $J$ already assigned to $\sW$ with $\ell (J) = 2 \ell (L)$.
\item[(S)] We assign $L$ to $\sS^{k,j}$ if none of the above occurs.
\end{itemize}
We finally set
\begin{equation}\label{e:bGamma}
\bGam:= ([-1,1]^2 \times \R^2) \cap \gira \setminus \bigcup_{L\in \sW} L = \{0\} \cup \bigcap_{k\geq N_0} \bigcup_{L\in \sS^k} L.
\end{equation}
\end{definition}

\begin{proposition}[Whitney decomposition]\label{p:whitney}
Let $T$, $\gamma_0$, $\beta_2$ and $\delta_1$ be as in the Assumptions \ref{induttiva}, \ref{piu_fogli} and \ref{esponenti}. If $M_0\geq C$, $N_0 \geq C(M_0)$, $C_e, C_h \geq C (M_0, N_0)$ (for suitably large constants) and $\eps_2$ is sufficiently small then:
\begin{itemize}
\item[(i)] $\ell (L) \leq 2^{-N_0 +1} |z_L|$ $\forall L\in \sS\cup \sW$;
\item[(ii)] $\sW^k = \emptyset$ for all $k\leq N_0+6$;
\item[(iii)] $\bGam$ is a closed set and $\sep (\bGam, L) := \inf \{|x-x'|: x\in \bGam, x'\in L\} \geq  2\ell (L)$ $\forall L\in \sW$.
\end{itemize}
Moreover, the following 
estimates hold with $C = C(M_0, N_0, C_e, C_h)$:
\begin{align}
&\bE (T_J, \bB_J) \leq C_e\bmax_0 \d (J)^{2\gamma_0-2+2\delta_1}\ell (J)^{2-2\delta_1}\qquad \forall J\in \sS\, ,\label{e:ex_ancestors}\\
&\bh (T_J, \bB_J) \leq C_h \bmax_0^{\sfrac14}\d (J)^{\sfrac{\gamma_0}{2} - \beta_2}\ell(J)^{1+\beta_2}
\qquad\;\;\; \forall J\in \sS\,, \label{e:ht_ancestors}\\ 
&\bE (T_H, \bB_H) \leq C\, \bmax_0 \d (H)^{2\gamma_0-2+2\delta_1}\ell (H)^{2-2\delta_1}\qquad \forall H\in \sW\, ,\label{e:ex_whitney}\\
&\bh (T_H, \bB_H) \leq C\, \bmax_0^{\sfrac14}\d (H)^{\sfrac{\gamma_0}{2} - \beta_2}\ell(H)^{1+\beta_2}
\qquad\;\;\;\, \forall H\in \sW\, . \label{e:ht_whitney}
\end{align}
\end{proposition}

\subsection{Approximating functions and construction algorithm} We will see below that in (a suitable portion of) each $\bB_L$ the current $T_L$ can be approximated efficiently with a graph of a Lipschitz multiple-valued map. The average of the sheets of this approximating map will then be used as a local model for the center manifold.

\begin{definition}[$\pi$-approximations]\label{d:pi-approximations}
Let $L\in \sS\cup \sW$ and $\pi$ be a $2$-dimensional plane. If $T_L\res\bC_{32 r_L} (p_L, \pi)$ fulfills 
the assumptions of \cite[Theorem 1.5]{DSS2} in the cylinder $\bC_{32 r_L} (p_L, \pi)$, then the resulting map $f: B_{8r_L} (p_L, \pi)  \to \Iq (\pi^\perp)$ given by \cite[Theorem 1.5]{DSS2} is a {\em $\pi$-approximation of $T_L$ in $\bC_{8 r_L} (p_L, \pi)$}. 
%The map $\hat{h}:B_{7r_L} (p_L, \pi) \to \pi^\perp$ given
%\footnote{{\color{red} Q: Maybe $\hat{h}$ should be $\hat{h}_L$?} {\color{blue} A: No, because $\hat{h}_L$ will be used when $\pi = \pi_L$, whereas the reference plane $\pi$ above might be different.}} 
%by $\hat{h}:= (\etaa\circ f)* \varrho_{\ell (L)}$ will be called the {\em smoothed average of the $\pi$-approximation}, where we recall the notation $\etaa \circ f(x) := Q^{-1} \sum_{i=1}^Q f_i(x)$ for any $Q$-valued map $f = \sum_i \a{f_i}$. 
\end{definition}

As in \cite{DS4}, we wish to find a suitable smoothing of the average of the $\pi$-approximation $\etaa\circ f$. However the smoothing procedure is more complicated in the case (b) of Definition \ref{d:semicalibrated}: rather than smoothing by convolution, we need to solve a suitable
elliptic system of partial differential equations. This approach can in fact be used in cases (a) and (c) as well. In several instances regarding case (a) and (c) we will have to manipulate maps defined on some affine space $q+\pi$ and taking value on $\pi^\perp$, where $q\in \Sigma$ and $\pi\subset T_q \Sigma$. In such cases it is convenient to introduce the following conventions: the maps will be regarded as maps defined on $\pi$ (requiring a simple translation by $q$), the space $\pi^\perp$ will be decomposed into $\varkappa := \pi^\perp \cap T_q \Sigma$ and its orthogonal complement $T_q \Sigma^\perp$ and we will regard $\Psi_q$ as a map defined on $\pi\times \varkappa$ and taking values in $T_q \Sigma^\perp$. Similarly, elements of $\pi^\perp$ will be decomposed as $(\xi, \eta)\in \varkappa \times T_q \Sigma^\perp$.

\begin{lemma}\label{l:tecnico2}
Let the assumptions of Proposition \ref{p:whitney} hold and assume $C_e \geq C^\star$ and $C_h \geq C^\star C_e$ for a suitably large $C^\star (M_0, N_0)$. For each $L\in \sW\cup \sS$ we choose a plane $\pi_L$ in $\bB_L$ which optimizes the excess and which, among all the ones optimizing the excess, also optimizes the height in $\bB_L$. 
For any choice of the other parameters,
if $\eps_2$ is sufficiently small, then $T_L \res \bC_{32 r_L} (p_L, \pi_L)$ satisfies the assumptions of
\cite[Theorem 1.5]{DSS2} for any $L\in \sW\cup \sS$. 
\end{lemma}

Indeed Lemma \ref{l:tecnico2} is a cororollary of the much more general Proposition \ref{p:app} and we will not give a separate proof. 

\begin{definition}[Smoothing]\label{d:smoothing}
Let $L$ and $\pi_L$ be as in Lemma \ref{l:tecnico2} and denote by $f_L$ the corresponding $\pi_L$-approximation. In case of Definition \ref{d:semicalibrated} (a)\&(c) we let
$\bar{f} (x) := \sum_i \a{\p_{T_{p_L} \Sigma} (f_i)}$ be the projection of $f_L$ on the tangent $T_{p_L} \Sigma$, whereas in the other case (Definition \ref{d:semicalibrated}(b)) we set $\bar{f} = f$. We let $\bar{h}_L$ be a solution (provided it exists) of 
\begin{equation}\label{e:ellittico-10}
\left\{
\begin{array}{ll}
\mathscr{L}_{L} \bar{h}_L = \mathscr{F}_{L}\\ \\
\left.\bar{h}_L\right|_{\partial B_{8r_L} (p_L, \pi_L)} = \etaa\circ \bar{f}_L\, ,
\end{array}\right.
\end{equation}
where $\mathscr{L}_L$ is a suitable second order linear elliptic operator with constant coefficients and $\mathscr{F}_L$ a suitable affine map: the precise expressions for $\mathscr{L}_L$ and $\mathscr{F}_L$ depend on a careful Taylor expansion of the first variations formulae and are given in Proposition \ref{p:pde}. We then set $h_L (x):= (\bar{h}_L (x), \Psi_{p_L} (x, \bar{h}_L (x))$ in case (a) and (c) and $h_L (x) = \bar{h}_L (x)$ in case (b). The map $h_L$ is the {\em tilted interpolating function} relative to $L$. 
\end{definition} 

In what follows we will deal with graphs of multivalued functions $f$ in several system of coordinates. These objects can
be naturally seen as currents $\bG_f$ (see \cite{DS2}) and in this respect we will use extensively the notation and results of \cite{DS2} (therefore $\gr (f)$ will denote the ``set-theoretic'' graph).

\begin{lemma}\label{l:tecnico3}
Let the assumptions of Proposition \ref{p:whitney} hold and assume $C_e \geq C^\star$ and $C_h \geq C^\star C_e$ (where $C^\star$ is the
constant of Lemma \ref{l:tecnico2}). For any choice of the other parameters,
if $\eps_2$ is sufficiently small the following holds. For any $L\in \sW\cup \sS$, there is a unique solution $\bar{h}_L:  B_{8r_L} (p_L, \pi_L) \to \varkappa_L = \pi_L^\perp \cap T_{p_L} \Sigma$ of \eqref{e:ellittico-10} and there is a smooth $g_L: B_{4r_L} (z_L, \pi_0)\to \pi_0^\perp$ such that 
$\bG_{g_L} = \bG_{h_L}\res \bC_{4r_L} (p_L, \pi_0)$, where $h_L$ is the tilted interpolating function of Definition \ref{d:smoothing}.
Using the charts introduced in Definition \ref{d:Riemann_surface}, the map $g_L$ will be considered as defined on the ball
$B_{4r_L} (z_L, w_L)\subset \gira$.
\end{lemma}

The center manifold is defined by gluing together the maps $g_L$. 

\begin{definition}[Interpolating functions]\label{d:glued} The map $g_L$ in Lemma \ref{l:tecnico2} will be called the {\em $L$-interpolating function}.
Fix next a $\vartheta\in C^\infty_c \big([-\frac{17}{16}, \frac{17}{16}]^m, [0,1]\big)$ which is identically $1$ on $[-1,1]^m$.
For each $k$ let $\sP^k := \sS^k \cup \bigcup_{i=N_0}^k \sW^i$ and
for $L\in \sP^k$ define $\vartheta_L ((z,w)):= \vartheta (\frac{z-z_L}{\ell (L)})$. Set
\begin{equation}
\hat\varphi_j := \frac{\sum_{L\in \sP^j} \vartheta_L g_L}{\sum_{L\in \sP^j} \vartheta_L} \qquad \mbox{on $\{(z,w)\in \gira: z\in [-1,1]^2\setminus \{0\}\}$}\, 
\end{equation}
and extend the map to $0$ defining $\hat\varphi_j (0)=0$. 
In case (b) of Definition \ref{d:semicalibrated} we set $\varphi_j:= \hat\varphi_j$. In cases (a) and (c) we let
$\bar{\varphi}_j (z,w)$ be the first $\bar{n}$ components of $\hat{\varphi}_j (z,w)$ and define
$\varphi_j (z,w) = \big(\bar{\varphi}_j (z,w), \Psi (z, \bar{\varphi}_j (z,w))\big)$.
$\varphi_j$ will be called the {\em glued interpolation} at step $j$.
\end{definition}

We now come to the first main theorem, which yields the surface which we call ``branched center manifold'' (again notice that
for $\bar Q=1$ there is certainly no branching, since the surface is a classical $C^{1,\alpha}$ graph, but we keep nonetheless the same terminology). In the statement we will need
to ``enlarge'' slightly dyadic squares: given $L\in \sC$ let $H$ be the dyadic square of $\R^2=\C$ so that $L$ is a connected component
of $\gira \cap (H\times \C)$. Given $\sqrt{2} \sigma < |z_L|= |z_H|$, we let $H'$ be the closed euclidean square of $\R^2$ which has the same
center as $H$ and sides of length $2\sigma$, parallel to the coordinate axes. The square $L'$ concentric to $L$ and with sidelength
$2\ell (L')= 2\sigma$ is then defined to be that connected component of $\gira\cap (H'\times \C)$ which contains $L$.

\begin{theorem}\label{t:cm}
Under the same assumptions of Lemma \ref{l:tecnico2}, the following holds provided $\eps_2$ is sufficiently small. 
\begin{itemize}
\item[(i)] For $\kappa := \beta_2/4$ and $C=C(M_0,N_0, C_e, C_h)$ we have (for all $j$) 
\begin{align}
&|\varphi_j (z,w)| \leq C \bmo^{\sfrac14} \vert z\vert^{1+\sfrac{\gamma_0}{2}} \qquad\qquad\qquad \mbox{for all $(z,w)$\label{e:stima_C0}}\\
&|D^l\varphi_j(z,w)| \leq C \bmo^{\sfrac12} \vert z\vert^{1+\gamma_0-l} \qquad \mbox{for $l=1,\dots,3$ and $(z,w)\neq 0$\label{e:stima_C3}}\\
&[D^3 \varphi_j]_{\mathcal{A}_j, \kappa} \leq C \bmo^{\sfrac12} 2^{2j}\label{e:stima_Hoelder}\, .
\end{align}
 \item[(ii)] The sequence $\varphi_j$ stabilizes on every square $L\in \sW$: more precisely, if $L\in \sW^i$ and 
$H$ is the square concentric to $L$ with $\ell (H) = \frac{9}{8} \ell (L)$, then $\varphi_k = \varphi_j$ on $H$ for every $j,k\geq i+2$.
Moreover there is an admissible smooth branching $\phii:\gira\cap ([-1,1]^2\times \C)\to \R^n$
such that $\varphi_k\to\phii$ uniformly on $\gira\cap ([-1,1]^2\times \C)$ and in $C^3 (\mathcal{A}_j)$ for every $j\geq 0$. 
Note in particular that $\phii$ coincides with $g_L$ on a nonempty open subset of each $L\in \sW$.  
\item[(iii)]  For some constant $C= C(M_0, N_0, C_e, C_h)$ and for $a':= b + \gamma_0> b$ we have
\begin{equation}\label{e:cornuto!}
|u (z,w) - \phii (z,w)|\leq C \bmo^{\sfrac{1}{2}} |z|^{a'}\, .
\end{equation}
\end{itemize}
\end{theorem}

\begin{definition}[Center manifold, Whitney regions]\label{d:cm}
The manifold $\cM:={\rm Gr}(\phii)$, where $\phii$ is as in Theorem \ref{t:cm}, is called
{\em a branched center manifold for $T$ relative to $\bG_u$}. It is convenient to introduce the map
$\Phii: \gira\cap ([-1,1]^2\times \C) \to \R^{2+n}$ given by $\Phii (z,w) = (z, \phii (z,w))$. If we neglect the origin,
$\Phii$ is then a classical ($C^3$) parametrization of $\cM$. $\Phii (\bGam)$ will be called the contact set.
Moreover, to each $L\in \sW$ we associate a {\em Whitney region} $\cL$ on $\cM$ as follows:
\begin{itemize}
\item[(WR)] $\cL := \Phii (H\cap ([-1,1]^2\times \C))$, where $H$ is the square concentric to $L$ with $\ell (H) = \frac{17}{16} \ell (L)$.
\end{itemize}
\end{definition}

\section{The normal approximation}

In what follows we assume that the conclusions of Theorem \ref{t:cm} apply and denote by $\cM$ the corresponding
center manifold. For any Borel set $\cV\subset \cM$ we will denote 
by $|\cV|$ its $\cH^2$-measure and will write $\int_\cV f$ for the integral of $f$
with respect to $\cH^2$. 
$\cB_r (q)$ denotes the geodesic balls in $\cM$. Moreover, we refer to \cite{DS2}
for all the relevant notation pertaining to the differentiation of (multiple valued)
maps defined on $\cM$, induced currents, differential geometric tensors and so on.

We next define the  open set
\begin{itemize}
\item[(V)] $\bV:= \{(x,y)\in \R^2\times \R^n : x\in [-1,1]^2 \mbox{ and } |\phii (x,w) - y|< c_s |x|^b/2\}$.
\end{itemize}
$\bV$ is clearly a horned neighborhood of the graph of $\phii$. By \eqref{e:separation}, Assumption \ref{induttiva} and
Theorem \ref{t:cm} it is clear that the following corollary holds

\begin{corollary}\label{c:cover} Under the hypotheses of Theorem \ref{t:cm}, there is $r>0$ such that
\begin{itemize}
\item[(i)] For every $x\in \R^2$ with $0<|x| = 2\rho < 2r$, the set $\bC_{\rho} (x)\cap \bV$ consists of $\bar Q$ distinct connected components and $\supp (T)\cap \bC_{3r} \subset \bV$.
\item[(ii)] There is a well-defined nearest point projection $\p: \bV\cap \bC_{4r} \to \gr (\phii)$, which is a $C^{2, \kappa}$ map.
\item[(iii)] For every $L\in \sW$ with $\d (L) \leq 2r$ and every $q\in L$ we have 
\[
\supp (\langle T, \p, \Phii (q)\rangle) \subset 
\big\{y\in \R^{2+n}\, : |\Phii (q)-y|\leq C \bmo^{\sfrac14} \d (L)^{\sfrac{\gamma_0}{2}-\beta_2} 
\ell (L)^{1+\beta_2}\big\}\, .
\]
\item[(iv)] $\langle T, \p, p\rangle = Q \a{p}$ for every $p\in \Phii (\bGam)\cap \bC_{2r} \setminus \{0\}$.
\end{itemize}
\end{corollary}

The main goal of this paper is to couple the branched center manifold of Theorem \ref{t:cm} with a good map defined on $\cM$ and taking values in its normal bundle, which approximates accurately $T$ in a neighborhood of the origin. 

\begin{definition}[$\cM$-normal approximation]\label{d:app}
Let $r$ be as in Corollary \ref{c:cover} and define 
\begin{itemize}
\item[(U)] $\bU:= \p^{-1} (\bC_{2r} \cap \mathcal{M})$.
\end{itemize}
An {\em $\cM$-normal approximation} of $T$ is given by a pair $(\cK, F)$ such that
\begin{itemize}
\item[(A1)] $F: \bC_{2r} \cap \cM \to \Iq (\bU)$ is Lipschitz and takes the form 
$F (x) = \sum_i \a{x+N_i (x)}$, with $N_i (x)\perp T_x \cM$ and $x+N_i (x) \in \Sigma$
for every $x$ and $i$.
\item[(A2)] $\cK\subset \cM$ is closed, contains $\Phii \big(\bGam\cap \bC_{2r})$ and $\bT_F \res \p^{-1} (\cK) = T \res \p^{-1} (\cK)$.
\end{itemize}
The map $N = \sum_i \a{N_i}:\cM\cap \bC_{2r} \to \Iq (\R^{2+n})$ is {\em the normal part} of $F$.
\end{definition}

In the definition above it is not required that the map $F$ approximates efficiently the current
outside the set $\Phii (\bGam)$. However, all the maps constructed
in this paper and used in the subsequent note \cite{DS5} will approximate $T$ with a high degree of accuracy
in each Whitney region: such estimates are detailed
in the next theorem. In order to simplify the notation, we will use $\|N|_{\cV}\|_{C^0}$ (or $\|N|_{\cV}\|_0$) to denote the number
$\sup_{x\in \cV} \cG (N (x), Q\a{0}) = \sup_{x\in \cV} |N(x)|$.

\begin{theorem}[Local estimates for the $\cM$-normal approximation]\label{t:approx}
Let $r$ be as in Corollary \ref{c:cover} and $\bU$ as in Definition \ref{d:app}. Then 
there is an $\cM$-normal approximation $(\cK, F)$ such that
the following estimates hold on every Whitney region $\cL$ associated to $L\in \sW$ with $\d (L)\leq r$:
\begin{align}
&\Lip (N|
_{\cL}) \leq C \bmo^{\beta_0} \d (L)^{\beta_0\,\gamma_0}\,\ell(L)^{\beta_0\gamma_0} \quad\mbox{and}\quad  \|N|
_{\cL}\|_{C^0}\leq C \bmo^{\sfrac14} \d (L)^{\sfrac{\gamma_0}{2}-\beta_2} \ell (L)^{1+\beta_2},\label{e:Lip_regional}\\
&|\cL\setminus \cK| + \|\bT_F - T\| (\p^{-1} (\cL)) \leq C \bmo^{1+\beta_0} \,\d (L)^{(1+\beta_0)(2\gamma_0-2+2\delta_1)}\,\ell (L)^{2+(1+\beta_0)(2-2\delta_1)}   ,\label{e:err_regional}\\
&\int_{\cL} |DN|^2 \leq C \bmo \,\d (L)^{2\gamma_0-2+2\delta_1}\,\ell (L)^{4-2\delta_1}\, .\label{e:Dir_regional}
\end{align}
Moreover, for every Borel $\cV\subset \cL$, we have
\begin{align}\label{e:av_region}
\int_\cV |\etaa\circ N| &\leq 
C \bmo \d (L)^{2(1+\beta_0)\gamma_0-2-\beta_2}\,\ell (L)^{5+\beta_2/4} \notag\\
&+ C \bmo^{\sfrac12+\beta_0}\,\d (L)^{2\beta_0 \gamma_0+\gamma_0-1-\beta_2}\, \ell (L)^{1+\beta_2}\,\int_\cV \cG \big(N, Q \a{\etaa\circ N}\big) \,.
\end{align} 
The constant $C=C(M_0,N_0,C_e,C_h)$ does not depend on $\eps_2$.
\end{theorem}

\subsection{Separation and splitting} 
We conclude this section with two theorems which allow us to estimate the sidelengths of the squares of type $\sW_h$ and $\sW_e$. The squares in $\sW_n$ do not enjoy similar bounds, but they can be partitioned in families, each of which consists of squares sufficiently close to an element of $\sW_e$.

\begin{proposition}[Separation]\label{p:separ}
There is a dimensional constant $C^\sharp > 0$ with the following property.
Assume the hypotheses of Theorem \ref{t:approx}, and in addition
$C_h^{4} \geq C^\sharp C_e$. 
If $\eps_2$ is sufficiently small, then the following conclusions hold for every $L\in \sW_h$ with $\d (L)\leq r$
:
\begin{itemize}
\item[(S1)] $\Theta (T_L, p) \leq Q - 1$ for every $p\in \bB_{16 r_L} (p_L)$.
\item[(S2)] $L\cap H= \emptyset$ for every $H\in \sW_n$
with $\ell (H) \leq \frac{1}{2} \ell (L)$.
\item[(S3)] $\cG \big(N (x), Q \a{\etaa \circ N (x)}\big) \geq \frac{1}{4} C_h \bmo^{\sfrac{1}{4}}\, \d (L)^{\sfrac{\gamma_0}{2}-\beta_2}
\ell (L)^{1+\beta_2}$  $\forall x\in \Phii (B_{8 \ell (L)} (z_L, w_L))$.
\end{itemize}
\end{proposition}

A simple corollary of the previous proposition is the following.

\begin{corollary}[Domains of influence]\label{c:domains}
For any $H\in \sW_n$ there is a chain $L=L_0,\dots,L_r=H$ such that
\begin{itemize}
\item[(a)] $L_0\in \sW_e$ and $L_k\in \sW_n$ for all $k>0$;
\item[(b)] $L_k\cap L_{k-1}\neq \emptyset$ and $\ell (L_k)=\frac{\ell (L_{k-1})}{2}$ for all $k>0$.
\end{itemize}
In particular $H\subset B_{3\sqrt{2}\ell(L)}(z_L, w_L)$.
\end{corollary}

We use this last corollary to partition $\sW_n$.

\begin{definition}[Domains of influence]\label{d:domains}
We first fix an ordering of the squares in $\sW_e$ as $\{J_i\}_{i\in \mathbb N}$ so that their sidelengths do not increase. Then $H\in \sW_n$
belongs to $\sW_n (J_0)$ (the domain of influence of $J_0$) if there is a chain as in Corollary \ref{c:domains} with $L_0 = J_0$.
Inductively, $\sW_n (J_r)$ is the set of squares $H\in \sW_n \setminus \cup_{i<r} \sW_n (J_i)$ for which there is
a chain as in Corollary \ref{c:domains} with $L_0 = J_r$.
\end{definition}

\begin{proposition}[Splitting]\label{p:splitting}
There are positive constants $C_1$,$C_2(M_0)$, $\bar r (M_0, N_0, C_e)$ such that, if $M_0\geq C_1$, $C_e\geq C_2 (M_0)$, if the hypotheses of Theorem \ref{t:approx} hold and $\eps_2 $ is chosen sufficiently small,
then the following holds. If $L\in \sW_e$ with $\d (L) \leq \bar r$, $q\in \gira$ with $\dist (L, q) \leq 4\sqrt{2} \,\ell (L)$ and $\Omega := \Phii (B_{\ell (L)/8} (q))$, then:
\begin{align}
&C_e \bmo\, \d (L)^{2\gamma_0-2+2\delta_1} \ell(L)^{4-2\delta_1} \leq \ell (L)^2 \bE (T_L, \bB_L) \leq C \int_\Omega |DN|^2\, ,\label{e:split_1}\\
&\int_{\cL} |DN|^2 \leq C \ell (L)^2 \bE (T_L, \bB_L) \leq C \ell (L)^{-2} \int_\Omega |N|^2\, ,\label{e:split_2}
\end{align}
where $C=C(M_0,N_0,C_e,C_h)$.
\end{proposition}

\section{Center manifold construction}

\subsection{Technical preliminaries} In this section we prove the two technical Lemmas \ref{l:density} and \ref{l:piccolezza}.

\begin{proof}[Proof of Lemma \ref{l:density}] Consider $x_0\in \pi_0$ with $2\rho = |x_0|$, a smooth $C^2$ function $\phi: B_{\rho} (x_0) \to \R^n$ and the open set $\bV_{\varrho} := \{(x,y): x\in B_{\rho/2} (x_0), |y-\phi (x)|< \varrho\}$. Recall that there is a geometric constant $C$ such that, if $\varrho \leq C /\|D^2 \phi\|_{B_\rho (x_0)}$, then for each $p\in \bV_{\varrho}$ there is a unique nearest point $\P (p)\in \gr (\phi)$ (which defines a $C^1$ map $\P: \bV_{\varrho}\to \gr (\phi)$). In particular, if $\|D^2 \phi\|_{B_\rho (x_0)} \leq C \rho^{\alpha-1}$, the existence of such point is guaranteed under the assumption that $\varrho \leq c \rho^{1-\alpha}$ (where $c$ is a, possibly small but positive, constant). Consider now an admissible smooth branching $u:\gira_{\bar Q} \to \R^n$. If $\bar{Q}=1$, the above discussion shows easily the existence of a well defined $C^1$ map $\P: \bV_{u,a}\cap \bC_{2r} \to \gr (u)$, provided $r$ is sufficiently small. If $\bar{Q}>
 1$, the same conclusion holds under the assumption that $u$ is $b$-separated and $a>b>1$. Indeed consider $p= (z,y)\in  \bV_{u,a}$ and $(z,w_i)\in \gira_Q$ such that $|y- u (z,w_i)|\leq c_s |z|^a$. The assumptions of being well-separated implies easily that $|p- u(\zeta, \omega)|\geq c_s |z|^b$ whenever $\zeta\not\in B_{|z|/2} (z,w_i)$ and thus we can argue locally on the sheet $\gr (u|_{B_{|z|/2} (z,w_i)})$. 

Next, up to rescaling we can assume that $\P$ is well-defined on $\bV_{u,a}\cap \bC_2$. The discussion before Lemma \ref{l:density} applies now verbatim and we conclude that the unique tangent cone at $0$ is $Q\bar Q \a{\pi_0}$.

To reach the other two conclusions of the Lemma we argue by contradiction: if they were  wrong, then we would find a sequence of points $\{x_k\}\subset B_2 (0)\setminus \{0\}$ converging to $0$ for which one of the following two conditions hold:
\begin{itemize}
\item either $\{x_k\}\times \R^n$ contains a point $p_k\in \supp (T)$ with $\Theta (p_k, T) \geq Q + 1$ (recall that the density of $T$ is an integer at every point, cf. \cite{DSS1});
\item or one connected component $\Omega_k$ of $(\{x_k\}\times \R^n)\cap \bV_{u,a}$ does not intersect $\supp (T)$.
\end{itemize}
Set $2r_k := |x_k|$ and consider the connected component $\bV_k$ of $\bV_{u,a} \cap \bC_{r_k} (x_k)$ which contains $p_k$ (in the first case) or $\Omega_k$ (in the second). Let $S_k := T_k \res \bV_k$ and let $q_k = (x_k, u (x_k, w_k))$ be such that $q_k \in \bV_k$. Finally set $Z_k := (S_k)_{q_k, r_k}$. Observe that $\supp (Z_k)$ is contained in a neighborhood of height $C r_k^{a-1}$ of $\pi_0$ and we therefore conclude that $Z_k$ converges to a current $Z$ which is an integer multiple of $\a{B_1 (0)}$. On the other hand, since 
\[
(\P_\sharp S_k)\res \bC_{r_k/2} (x_k)=Q \bG_u \res \bC_{r_k/2} (x_k)\cap \bV_k
\]
for $k$ large enough, we conclude that $Z= Q \a{B_1 (0)}$. Now, either $\supp (Z_k)\cap (\{0\}\times \R^n)$ contains a point $\bar q_k$ of multiplicity $Q+1$ or it is empty. By the constancy theorem $(\p_{\pi_0})_\sharp Z_k = Q_k \a{B_1 (0)}$ for some integer $Q_k$ and,
since  $(\p_{\pi_0})_\sharp Z_k\to (\p_{\pi_0})_\sharp Z$, for $k$ large enough we would have $(\p_{\pi_0})_\sharp Z_k = Q \a{B_1 (0)}$.
This is then incompatible with the emptiness of $\supp (Z_k)\cap (\{0\}\times \R^n) = \emptyset$ because $Q\geq 1$. As for the other alternative, we must have, by the almost minimality of $Z_k$ (see \cite{DSS1})
\[
\limsup_{k\to \infty} \|Z_k\| (\bB_{1/2 - |\bar q_k|} (\bar q_k)) \leq \lim_{k\to\infty} \|Z_k\| (\bB_{1/2} (0)) = \textstyle{\frac{Q}{4}} \omega_2\, .
\]
Since $\bar q_k \to 0$, the almost monotonicity formula (see \cite{DSS1}) would imply $\Theta (\bar q_k, Z_k) \leq Q + o (1)$.
\end{proof}

\begin{proof}[Proof of Lemma \ref{l:piccolezza}] Since $Q\bar{Q}\a{\pi_0}$ is tangent to $T$ at $0$, we obviously must have $T_0 \Sigma\supset \pi_0$ and thus  $T_0 \Sigma = \R^{2+\bar{n}}\times \{0\}$ can be achieved suitably rotating the coordinates. To achieve the other two conclusions we scale $\Sigma$ and intersect it with $\bC_4 (0 ,T_0\Sigma)$ to reach that $\Sigma\cap \bC_4 (0, T_0 \Sigma)$ is the graph of some $\Psi$ with very small $C^{3,\eps_0}$ norm. We can then extend $\Psi$ outside $B_4 (0, T_0\Sigma)$ without increasing the $C^{3,\eps_0}$ norm by more than a factor: this gives (i) and (ii) and also shows that $\mathbf{c}$ can be assumed smaller than $\eps_2$ in case (a) and (c) of Definition \ref{d:semicalibrated}. For the details we refer the reader to the proof of \cite[Lemma 1.5]{DS4}. The rest of the Lemma is a simple scaling argument.
\end{proof}

\subsection{Proof of Proposition \ref{p:whitney}} In this section we prove several estimates on the excess, height and tilting of planes $\pi_L$ in the cubes $L\in \sW\cup\sS$. Proposition \ref{p:whitney} will then be a simple corollary of these more general statements.

\begin{proposition}[Tilting of optimal planes]\label{p:tilting opt}
Let $T$ be as in Assumptions \ref{induttiva} and \ref{piu_fogli} and assume the various parameters satisfy Assumption \ref{esponenti}. If $C_e, C_h \geq C (M_0, N_0)$ and $\eps_2$ is sufficiently small then:
\begin{itemize}
\item[(i)] The conclusions (i), (ii) and (iii) of Proposition \ref{p:whitney} hold.
\item[(ii)] $\bB_H\subset\bB_L \subset \bB_ {\d (L)/10} (p_L)$ and $T_H = T_L\res \bV_H$ for all $H, L\in \sW\cup \sS$ with $H\subset L$;
\end{itemize}
Moreover, if $H, L \in \sW\cup\sS$ and either $H\subset L$ or $H\cap L \neq \emptyset$ with $\frac{\ell (L)}{2} \leq \ell (H) \leq \ell (L)$, then the following holds, for $\bar{C} = \bar{C} (M_0, N_0, C_e)$ and $C = C(M_0, N_0, C_e, C_h)$:
\begin{itemize}
\item[(iii)] $\d (L) /2 \leq \d (H) \leq 2 \d (L)$ (and $\d (L) = \d (H)$ when $H\subset L$);
\item[(iv)] $|\pi_H-\pi_L| \leq \bar{C} \bmo^{\sfrac{1}{2}} \d (L)^{\gamma_0-1+\delta_1} \ell (L)^{1-\delta_1}$;
\item[(v)] $|\pi_H-\pi_0| \leq  \bar C \bmo^{\sfrac{1}{2}} \d(H)^{\gamma_0}$;
\item[(vi)] $\bh (T_H, \bC_{36 r_H} (p_H, \pi_0)) \leq C \bmo^{\sfrac{1}{4}} \d (H)^{\sfrac{\gamma_0}{2}} \ell (H)$ and $\supp (T_H) \cap \bC_{36 r_H} (p_H, \pi_0) \subset \bB_H$; 
\item[(vii)]  $\bh (T_L, \bC_{36r_L} (p_L, \pi_H))\leq C \bmo^{\sfrac{1}{4}} \d (L)^{\sfrac{\gamma_0}{2} -\beta_2} \ell (L)^{1+\beta_2}$
and $\supp (T_L)\cap \bC_{36 r_L} (p_L, \pi_H)\subset \bB_L$.
\end{itemize}
In particular, the estimates  \eqref{e:ex_whitney} and \eqref{e:ht_whitney} hold.
\end{proposition}

The proof of the proposition will use repeatedly a few elementary observations concerning the excess and the height, which we collect in the following
lemma.

\begin{lemma}\label{l:tecnico}
If $T$ is as in Proposition \ref{p:tilting opt} there is a geometric constant $C_0$ with the following properties. Assume the points $p,q$ belong to $\supp (T)\cap \bC_{\sqrt{2}}$, $\bB_r (p) \subset \bB_\rho (q) \subset \bC_2$ and $r\geq \rho/4$. Then, if $\eps_2 \leq C_0^{-1}$
\begin{itemize}
\item[(i)] $\bE (T, \bB_\rho (q))\leq C_0 \min_{\tau} \bE (T, \bB_\rho (q), \tau) + C_0 \bmo \rho^2$;
\item[(ii)] $\bE (T, \bB_r (p)) \leq C_0 \bE (T, \bB_\rho (q)) + C_0 \bmo r^2$;
\item[(iii)] $|\pi - \tau|^2 \leq C_0[\bE (T, \bB_r (p), \pi) + \bE (T, \bB_\rho (q), \tau)]$;
\item[(iv)] $\bh (T, F, \pi) \leq \bh (T, F, \tau) + C_0 |\pi-\tau| \diam (\supp (T) \cap F)$ for any set $F$;
\item[(v)] $\bh (T, \bC_r (0, \pi)) \leq C_0 \bmo^{\sfrac{1}{2}} r^{1+\gamma_0} + C_0 |\pi-\pi_0| r$ whenever $|\pi - \pi_0|\leq C_0^{-1}$ and $r<7/4$.
\end{itemize}
\end{lemma}
\begin{proof} Recall that, by Lemma \ref{l:density2} and Allard's monotonicity formula (which can be applied by \cite[Proposition 1.2]{DSS1}), we have
\begin{equation}\label{e:density_10}
\frac{3\omega_2}{4} \rho^2 \leq \|T\| (\bB_\rho (p))\leq C_0 \rho^2\, .
\end{equation}
(i) is trivial in (b) of Definition \ref{d:semicalibrated}, since $\bE (T, \bB_\rho (q))= \min_{\tau} \bE (T, \bB_\rho (q), \tau)$. In the cases (a) and (c) recall that 
\[
\bE (T, \bB_\rho (q))= \min_{\tau\subset T_q \Sigma} \bE (T, \bB_\rho (q), \tau)\, .
\]
Let now $\pi$ be such that $\bE (T, \bB_\rho (q), \pi)=\min_{\tau} \bE (T, \bB_\rho (q), \tau) =: E$. Then, by the Chebyshev inequality there is a point $q'\in \bB_\rho (q)\cap \supp (T)$ such that
\[
|\vec{T} (q') - \vec \pi|^2 \leq \frac{\omega_2 \rho^2}{\|T\| (\bB_\rho (q))} E \leq C_0 E\, .
\]
Observe that $\vec{T} (q')$ is the orienting $2$-vector of some space $\xi\subset T_{q'} \Sigma$ and that 
\[
|T_{q'} \Sigma - T_q \Sigma|^2 \leq C_0 \|A_\Sigma\|_{C^0}^2 \rho^2 \leq C_0 \bmo \rho^2\, .
\]
Thus there is a $2$-plane $\tau \subset T_q \Sigma$ such that $|\tau - \pi|^2 \leq C_0 E + C_0 \bmo \rho^2$. Hence
\begin{align*}
\bE (T, \bB_\rho (p)) \leq \bE (T, \bB_\rho (q), \tau) \leq C_0 (E+ C_0 \bmo \rho^2)\|T\| (\bB_\rho (q)) / (\omega_2 \rho^2) \leq C_0 E + C_0 \bmo \rho^2\, .
\end{align*}

Keeping the notation of the argument above, in the case (b) of Definition \ref{d:semicalibrated} statement (ii) follows from  the simple observation
\[
\bE (T, \bB_r (p)) \leq \bE (T, \bB_r (p), \pi) \leq 4^2 \bE (T, \bB_\rho (q), \pi) = 16 \bE (T, \bB_\rho (q))\, .
\]
In the cases (a) and (c) of Definition \ref{d:semicalibrated} we combine the same idea with (i).

(iii) is a simple consequence of 
\begin{align}
|\pi-\tau|^2 &\leq \frac{2}{\|T\| (\bB_\rho (q))} \int_{\bB_\rho (q)} (|\vec{T} - \vec{\pi}|^2 + |\vec{\tau} - \vec{T}|^2) d\|T\|\nonumber\\
&\stackrel{\eqref{e:density_10}}{\leq} C_0 \big(\bE (T, \bB_\rho (q), \pi) + \bE (T, \bB_\rho (q), \tau)\big)\, ,
\end{align}
and $\bE (T, \bB_\rho (q), \pi) \leq 16 \bE (T, \bB_r (p), \pi)$.
Next, for $p,q\in \supp (T) \cap F$ we compute 
\[
|\p_\pi^\perp (p-q)| \leq |\p_\tau^\perp (p-q)| + |(\p_\tau^\perp - \p_\pi^\perp) (p-q)|\leq \bh (T, F, \tau) + C |\pi-\tau| |p-q|\, .
\]
Taking the supremum over $p,q\in F\cap \supp (T)$ we reach (iv).

We finally argue for (v).
Fix $r< 7/8$, $\pi$ with $|\pi-\pi_0|\leq C_0^{-1}$ and the cylinder $\bC :=  \bC_r (0, \pi)$.
Observe that, by Assumption \ref{induttiva}, for every $p = (x,y) \in \supp (T) \cap (\R^2\times \R^n)$ we have
$|y|\leq \eps_2^{\sfrac{1}{2}} |x|^{1+\alpha} \leq \eps_2^{\sfrac{1}{2}} |x|^{1+\gamma_0}$. It follows easily that, for $C_0$ sufficiently large and $\eps_2$ sufficiently small, this implies that $\supp (T)\cap \bC \subset \bC_{8r /7} (0, \pi_0)$. Hence, $\bh (T, \bC, \pi_0) \leq \bh (T, \bC_{8r/7} (0, \pi_0)) \leq C_0 \bmo^{\sfrac{1}{2}} r^{1+\gamma_0}$. As a consequence $\diam (T\cap \bC) \leq C_0 r$ and (v) follows from (iv).
\end{proof} 

\begin{proof}[Proof of Proposition \ref{p:tilting opt}]  In this proof we will use the following convention: geometric constants will be denoted by $C_0$ or $c_0$, constants depending upon
$M_0, N_0, C_e$ will be denoted by $\bar{C}$ or $\bar{c}$ and constants depending upon $M_0, N_0, C_e$ and $C_h$ will be denoted by $C$ or $c$.
Next observe that the second inclusion in (ii) is in fact correct for any cube $L\in \sC^j$ with $j\geq N_0$, provided $N_0$ is chosen sufficiently large compared to $M_0$. (iii) is instead an obvious consequence of the construction.

\medskip

{\bf Proof of (i), (ii) and (iii) in Proposition \ref{p:whitney}.}
The conclusion (i) is obvious since indeed it also holds for every $L\in \sC^{N_0}$. (iii) is a simple consequence of the fact that, because of (NN) in the refining procedure, given any pair $H, L \in \sW$ with nonempty intersection, $\frac{1}{2} \ell (H) \leq \ell (L) \leq 2 \ell (H)$. Consider now any $L\in \sC^j$ with $N_0 \leq j \leq N_0+6$. Observe first that $2^{-N_0-10} \d (L) \leq \ell (L) \leq 2^{-N_0} \d (L)$. Consider now the point $p\in \gr (u)$ which has the same projection onto $\pi_0$ as $p_L$ (namely $z_L$) and which is closest to $p_L$. Recall that $\pi (p)$ is the tangent to $\gr (u)$ in $p$. We thus can use \eqref{e:effetto_energia} to estimate
\begin{equation}\label{e:delicato}
\bE (T_L, \bB_L, \pi (p)) \leq \bE (T, \bB_{|z_L|/4} (p))) \leq C (M_0, N_0) \bmo \d (L)^{2\gamma_0 -2 + 2\delta_1} \ell (L)^{2-2\delta_1}\, .
\end{equation}
In order to derive the latter inequality we need that $\bB_L \subset \bB_{|z_L|/4} (p)$. Since $|z_L|$ is compable, up to a geometric constant, to $\d (L)$, whereas the radius $r_L$ equals $64 \sqrt{2} M_0 \ell (L)$ and is comparable, up to a geometric constant, to $M_0 2^{-N_0} \d (L)$, we just need to choose $N_0$ sufficiently large compared to $M_0$. 

By Lemma \ref{l:tecnico}(i), from \eqref{e:delicato} we conclude 
\[
\bE (T_L, \bB_L) \leq C (M_0, N_0) \bmo \d (L)^{2\gamma_0 -2 + 2\delta_1} \ell (L)^{2-2\delta_1} + C (M_0) \bmo \ell (L)^2\, .
\]
Hence, for $C_e$ sufficiently large, condition (EX) of Definition \ref{d:refining} cannot be a reason to stop the refining procedure of any cube $L\in \sC^j$ when $N_0 \leq j \leq N_0 +6$.

Recall next the chosen plane $\pi_L$ such that $\bE (T_L, \bB_L, \pi_L) = \bE (T_L, \bB_L)$ and 
$\bh (T_L, \bB_L) = \bh (T_L, \bB_L, \pi_L)$. By Lemma \ref{l:tecnico}(iii) we easily conclude that 
\[
|\pi_L - \pi (p)|\leq C (M_0, N_0) C_e^{\sfrac{1}{2}}
\bmo^{\sfrac{1}{2}} \d (L)^{\gamma_0}\, .
\]
On the other hand $|\pi (p) - \pi_0|\leq C_0 [Du]_{0, \alpha, B_{C_0 \d (L)}} \d (L)^\alpha \leq C_0 \bmo^{\sfrac{1}{2}} \d (L)^{\gamma_0}$ and thus 
\begin{equation}\label{e:mostrato}
|\pi_L - \pi_0| \leq C (M_0, N_0) \,C_e^{\sfrac{1}{2}} \bmo^{\sfrac{1}{2}}\d (L)^{\gamma_0}\qquad \forall L\in \sC^{N_0}\, .
\end{equation}
Setting $\rho:= (2\sqrt{2}+\sfrac{1}{10}) \d (L)$ we have
\[
\bB_L \subset \bC_{\d (L)/10} (p_L, \pi_0) \subset \bC_\rho (0, \pi_0)\, .
\] 
Note that $\rho \leq (2\sqrt{2}+\frac{1}{10}) \frac{1}{2} \leq \frac{3}{2}$ and we can apply Lemma \ref{l:tecnico}(v) to conclude
\begin{equation}\label{e:mostrato2}
\bh (T, \bB_L, \pi_0) \leq \bh (T, \bC_\rho (0, \pi_0), \pi_0) \leq C_0 \bmo^{\sfrac{1}{2}} \d (L)^{1+\gamma_0}\, .
\end{equation}
In particular ${\rm diam}\, \supp (T)\cap \bB_L)\leq \d (L)$, provided $\eps_2$ is small enough. We can therefore 
apply Lemma \ref{l:tecnico}(iv) and use \eqref{e:mostrato} and \eqref{e:mostrato2} to infer
\begin{align*}
& \bh (T_L, \bB_L)\leq \bar C \bmo^{\sfrac{1}{2}} \d(L)^{1+\gamma_0}\leq \bar C  \bmo^{\sfrac{1}{4}} \d (L)^{\sfrac{\gamma_0}{2} -\beta_2} \ell (L)^{1+\beta_2}\, .
\end{align*}
Thus, choosing $C_h$ large depending upon $M_0, N_0$ and $C_e$, we conclude that condition (HT) in Definition \ref{d:refining} cannot be a reason to stop the refining procedure of a cube $L\in \sC^j$ when $N_0\leq j\leq N_0+6$. 

This means that: for $k = N_0$ and $j=0$ all cubes of $\sC^{N_0,0}$ are assigned to $\sS$ and refined at the subsequent step (the condition (NN) is empty here).
But then the same happens for $k=N_0$ and $j=1$, since $\sW^{N_0, 0}$ is empty. Proceeding inductively we conclude this for every $j$ and thus obtain that $\sW^{N_0}$ is empty. We now repeat the argument with $\sW^{N_0+1, j}$ to conclude that $\sW^{N_0+1}$ is also empty. Proceeding for other 5 steps we conclude then that (ii) holds.

\medskip

{\bf Proof of (ii)-(iv)-(v)-(vi)-(vii) when $H\subset L$.} The proof is by induction over $i$, where $H\in \sC^i$. We thus prove first the claims when $i = N_0$. 
Under this assumption $H=L$ and hence (iv) is trivial. The second inclusion in (ii) has already been proved above and the remaining assertions of (ii) are obvious because $H=L$. (v) has been shown above, cf. \eqref{e:mostrato}. The first conclusion in (vi) follows easily, since $\bh (T_H, \bC_{36 r_H} (p_H, \pi_0)) \leq C_0 \bmo^{\sfrac{1}{2}} \d (H)^{1+\gamma_0}$ by Lemma \ref{l:tecnico}(v) and $\ell (H) \geq \d (H) / C(N_0)$. The inclusion in (vi) follows then trivially from this bound when $\bmo\leq\eps_2$ is small enough, because $p_H \in \supp (T_H)$. As for (vii), recall that $L=H$ in our case. First observe that $|\pi_H-\pi_0|\leq C_0 C_e \d (L)^{\gamma_0}$, simply by \eqref{e:mostrato} (assuming $C_e \geq C (M_0, N_0)$). Thus we can apply Lemma \ref{l:tecnico}(iv)\&(v): since $\d (L)$ and $\ell (L)$ are comparable up to a constant $C(N_0)$, we conclude that $\bh (T_L, \bC_{36r_L} (p_L, \pi_H))\leq C \bmo^{\sfrac{1}{4}} \d (L)^{\sfrac{\gamma_0}{2} -\beta_2} \ell (L)^{1+\beta_2}$. As we already argued for (vi), the inclusion is a consequence of the bound.

We now pass to the inductive step. Thus fix some $H_{i+1} \in \sS^{i+1}\cup \sW^{i+1}$ and consider a 
chain $H_{i+1}\subset H_i \subset \ldots \subset H_{N_0}$ with $H_l\in \sS^l$ for $l\leq i$.
We wish to prove all the conclusions (ii)-(iv)-(v)-(vi)-(vii) when $H = H_{i+1}$ and $L = H_j$ for some $j \leq i+1$, recalling that, by inductive assumption, all the statements hold when $H =H_k$ and $L= H_l$ for $l\leq k \leq i$. 
Note also that $\d (H_k) = \d (H_{i+1})$ for all $k$.

With regard to (ii), it is enough to prove that $\bB_{H_{i+1}} \subset \bB_{H_{i}}$ and $\bV_{H_{i+1}} \subset \bV_{H_i}$. Note that $|z_{H_i}-z_{H_{i+1}}| \leq 2\sqrt{2}\,\ell(H_i)$ (recall the notation $p_H = (z_H, y_H)$).   In particular notice that $\bC_{r_{H_{i+1}}} (p_{H_{i+1}}, \pi_0) \subset \bC_{r_{H_i}} (p_{H_i}, \pi_0)$.  Recall the open sets $\bV_{H_i}$ and $\bV_{H_{i+1}}$ defined in Section \ref{ss:whitney}. Since $H_i$ and $H_{i+1}$ are nearby cubes in $\gira$, it is clear that 
the points $(z_{H_{i+1}}, u (z_{H_{i+1}}, w_{H_{i+1}}))$ and $(z_{H_i}, u (z_{H_i}, w_{H_i}))$ must be in the same connected component of $\bV_{u,a} \cap \bC_{r_{H_{i}}} (p_{H_{i}}, \pi_0)$. It then follows that $\bV_{H_{i+1}}\subset \bV_{H_{i}}$.
In particular $p_{H_{i+1}}\in \supp (T_{H_i})$
and (vi) applied to $H= H_i$ implies then that $|p_{H_{i+1}}-p_{H_i}|\leq 2 (\sqrt{2} + C \bmo^{\sfrac{1}{4}}) \ell (H_{i})$. In particular, assuming
that $\eps_2 \leq c$ for some positive constant $c=c (M_0, N_0, C_e, C_h)$, we conclude $|p_{H_{i+1}} - p_{H_i}|\leq 3 \sqrt{2} \ell (H_i)$ and $\bB_{H_{i+1}}\subset \bB_{H_i}$ follows from the fact that $M_0$ is assumed larger than a suitable geometric constant.

We now come to (iv). Recall that $H_i$ belongs to $\sS$ (indeed $H_{i+1}$ is a son of $H_i$). Hence, from the inclusion $\bB_{H_{i+1}}\subset \bB_{H_i}$, from the identity $T_{H_{i+1}} = T_{H_i} \res \bB_{H_{i+1}}$ and from Lemma \ref{l:tecnico}(ii) we easily infer that 
\[
\bE (T_{H_{i+1}}, \bB_{H_{i+1}}) \leq C_0 \bE (T_{H_i}, \bB_{H_i}) + C_0 \bmo \ell (H_{i+1})^2
\leq \bar{C} \bmo \d (H_{i+1})^{2\gamma_0 -2 + 2 \delta_1} \ell (H_{i+1})^{2-2\delta_1}\, .
\]
We thus have, from Lemma \ref{l:tecnico}(iii), 
\[
|\pi_{H_i} - \pi_{H_{i+1}}| \leq \bar{C} \bmo^{\sfrac{1}{2}} 
\d (H_{i+1})^{\gamma_0 -1 + \delta_1} \ell (H_{i+1})^{1-\delta_1}\, .
\]
On the other hand, since $\d (H_l)= \d (H_j)$ for every $l\geq j$, by the same argument with $l$ in place of $i$ we also get
\[
|\pi_{H_l} - \pi_{H_{l+1}}| \leq \bar{C} \bmo^{\sfrac{1}{2}} 
\d (H_{i+1})^{\gamma_0 -1 + \delta_1} \ell (H_{l+1})^{1-\delta_1}\, .
\]
Summing the latter estimates for $l$ between $i$ and $j$, we easily reach (iv) for $H= H_{i+1}$ and $L = H_j$. 

As for (v), note that it holds for $H_{N_0}$ and moreover we just proved (iv) for $H= H_{i+1}$ and $L= H_{N_0}$, and thus, by triangular inequality, we get (v) (with a constant independent of the index $i$!). 

As for (vi), note first that $\supp (T_{H_{i+1}}) \cap \bC_{36 r_{H_{i+1}}} (p_{H_{i+1}}, \pi_0) \subset \supp (T_{H_i}) \cap \bC_{36 r_{H_i}} (p_{H_i}, \pi_0) \subset \bB_{H_i}$ (the latter because (vi) holds for $H = H_i$ by inductive hypothesis). Thus we can apply Lemma \ref{l:tecnico}(iv) to conclude
\[
\bh (T_{H_{i+1}}, \bC_{36 r_{H_{i+1}}} (p_{H_{i+1}}, \pi_0) )\leq C \bh (T_{H_i}, \bB_{H_i}) + C_0 |\pi_{H_i} - \pi_0|\, \diam (\supp (T_{H_i})\cap \bB_{H_i})\, .
\]
On the other hand we already noticed that $H_i \in \sS$. Taking into account (v) we then conclude the inequality of (vi) for $H={H_{i+1}}$ and, as already noticed in other cases, the inclusion follows from the estimate and $p_{H_{i+1}}\in \bB_{H_{i+1}}\cap \supp (T_{H_{i+1}})$. 

We finally come to (vii). Fix $H= H_{i+1}$. First we prove it for $L = H_{N_0}$. Observe that by the bound from (iv) on $|\pi_H - \pi_0|$, we can bound 
$\bh (T_{H_{N_0}}, \bC_{36r_{H_{N_0}}} (p_{H_{N_0}}, \pi_H))$ with the same argument used for $\bh (T_{H_{N_0}}, \bC_{36r_{H_{N_0}}} (p_{H_{N_0}}, \pi_{H_{N_0}}))$. As already argued several times, we then conclude the inclusion $\supp (T_{H_{N_0}}) \cap \bC_{36r_{H_{N_0}}} (p_{H_{N_0}}, \pi_H)\subset \bB_{H_{N_0}}$. We now argue inductively on $j$: assuming that we know (vii) for $H$ and $L= H_j$, we now wish to conclude it for $L= H_{j+1}$.
Notice that $\bC_{36 r_{H_{j+1}}} (p_{H_{j+1}}, \pi_H) \subset \bC_{36 r_{H_{j}}} (p_{H_{j}}, \pi_H)$. Then the inductive assumption gives
$\supp (T_{H_j}) \cap \bC_{36 r_{H_{j+1}}} (p_{H_{j+1}}, \pi_H)\subset \bB_{H_j}$ and recalling that $T_{H_{j+1}} = T_{H_j} \res \bB_{H_{j+1}}$ and that
$H_j\in \sS$, we can use Lemma \ref{l:tecnico}(iv) to bound
\[
\bh (T_{H_{j+1}}, \bC_{36 r_{H_{j+1}}} (p_{H_{j+1}}, \pi_H)) \leq \bh (T_{H_j}, \bB_{H_j}) + C_0 |\pi_H - \pi_{H_{j}}| \diam (\supp (T_{H_j}) \cap \bB_{H_j})\, .
\]
However, having already shown (iv), this easily shows the bound in (vii). The inclusion then follows with the usual argument used above.

\medskip

{\bf Proof of \eqref{e:ex_whitney} and \eqref{e:ht_whitney}.} Fix $H\in \sW$ and let $L$ be its father (recall Proposition \ref{p:whitney}(ii)). Having shown (ii), we know that $\bB_H\subset \bB_L$. We then use $\d (L) = \d (H)$, $\ell(L)= 2\ell(H)$, and $L\in \mathscr{S}$ to estimate 
\[
\bE (T_L, \bB_L)\leq C_e \bmax_0 \d (H)^{2\gamma_0-2+2\delta_1}\ell (H)^{2-2\delta_1}
\] 
and Lemma \ref{l:tecnico}(i) to conclude \eqref{e:ex_whitney} as follows
\begin{align*}
\bE(T_H,\bB_H)\leq & C_0 \bE(T_H,\bB_H,\pi_L) +\bar C \bmo r_L^2 \leq \bar C \bE (T_L, \bB_L)+\bar C \bmo \ell(H)^2\\
\leq & \bar C\bmax_0 \d (H)^{2\gamma_0-2+2\delta_1}\ell (L)^{2-2\delta_1}
\end{align*}
Next, we use Lemma \ref{l:tecnico}, (iv) and
\[
\bh (T_L, \bB_L) \leq C_h \bmax_0^{\sfrac14}\d (L)^{\sfrac{\gamma_0}{2} - \beta_2}\ell(L)^{1+\beta_2}
\] 
to conclude \eqref{e:ht_whitney}.

\medskip

{\bf Proof of (iv) and (vii) when $H$ and $L$ are neighbors.} Without loss of generality assume $\ell (L)\geq \ell (H)$. If $L\not\in \sC^{N_0}$, then let $J$ be the father of $L$ (and observe that it belongs to $\mathscr{S}$). Observe that $|z_H - z_J|, |z_L - z_J| \leq 2\sqrt{2} \ell (J)$. On the other hand, observe that $p_H, p_L$ are both elements of $\bC_{36 r_J} (p_J, \pi_0)$ (provided $M_0$ is larger than a geometric constant). Thus, by (vi) (applied to $J$), for $\eps_2$ sufficiently small we easily conclude
$|p_H - p_J|, |p_L-p_J| \leq 3 \sqrt{2} \ell (J)$. Since $\ell (L), \ell (H) \leq \ell (J)/2$, again assuming that $M_0$ is larger than a geometric constant we have the inclusion $\bB_H\cup \bB_L \subset \bB_J$. It is also easy to see that $\bV_H\cup \bV_L \subset \bV_J$. Now, we can use \eqref{e:ex_ancestors}, \eqref{e:ex_whitney} and (iii) to achieve
\[
|\pi_H - \pi_J|, |\pi_L-\pi_J| \leq  \bar{C} \bmo^{\sfrac{1}{2}} \d (J)^{\gamma_0-1+\delta_1} \ell (J)^{1-\delta_1}\, .
\]
Next we use again (iii), the triangle inequality and $\ell (H) \leq \ell (L) \leq \ell (J) \leq 4 \ell (H)$ to show (iv). The case $L\in \sC^{N_0}$ can be handled similarly, just using a ball concentric to $\bB_L$ and slightly larger so to include $\bB_H$: the excess and the height in this ball is then estimated with the same argument used for estimating them in $\bB_L$.

As for (vii) we fix a chain of ancestors $L = L_j, L_{j-1}, \ldots, L_i, \ldots \ldots , L_{N_0}$ and, as in the proof of (vii) for the case $H\subset L$, we argue inductively over $i$. The argument is precisely the same and can be applied because, using (iv) for $H$ and $L$ and for $L_i$ and $L_{i+1}$, we can sum the corresponding estimate to show that
\[
|\pi_H - \pi_{L_i}| \leq  \bar{C} \bmo^{\sfrac{1}{2}} \d (L_i)^{\gamma_0-1+\delta_1} \ell (L_i)^{1-\delta_1}\, . \qedhere
\]
\end{proof}

\section{$\pi$-approximations and elliptic regularizations}

In this section we introduce the $\pi$-approximations and define the corresponding elliptic regularizations of their averages, which in turn will be the building blocks of the center manifold. We begin with the following:

\begin{proposition}\label{p:app}
Assume the hypotheses and the conclusions of Proposition \ref{p:tilting opt} apply and let $\eps_2$ be sufficiently small.
If $H, L \in \sW\cup\sS$ and either $H\subset L$ or $H\cap L \neq \emptyset$ with $\frac{\ell (L)}{2} \leq \ell (H) \leq \ell (L)$, then 
\begin{align}
& (\p_{\pi_H})_\sharp (T_L \res \bC_{32 r_L} (p_L, \pi_H)) = Q \a{B_{32r_L} (\p_{\pi_H} (p_L), \pi_H)}\, ,\label{e:app2}\\
&\partial T_L \res \bC_{32 r_L} (p_L, \pi_H) = 0\, .\label{e:app3}
\end{align}
Moreover 
\begin{equation}\label{e:app4}
\bE (T_L, \bC_{32 r_L} (p_L, \pi_H)) \leq \bar C \bmo \d (L)^{2\gamma_0 -2 + 2 \delta_1} \ell (L)^{2-2\delta_1}
\end{equation}
and hence \cite[Theorem~1.5]{DSS2} applies to the current $T_L\res \bC_{32r_L} (p_L, \pi_H)$ in $\bC_{32 r_L} (p_L, \pi_H)$.
\end{proposition}
\begin{proof}
\eqref{e:app3} is rather straightforward: by the height estimate in Proposition \ref{p:tilting opt}(vii) we conclude easily 
$\supp (T_L) \cap \bC_{32 r_L} (p_L, \pi_H)\subset \bC_{36r_L} (p_L, \pi_0)$. On the other hand by definition of $T_L = T \res \bV_L$
and by Assumption \ref{induttiva}(Hor), we have $\supp (\partial T_L) \subset \partial \bC_{64 r_L} (p_L, \pi_0)$, implying 
$\supp (\partial T_L) \cap \bC_{36 r_L} (p_L, \pi_0) = \emptyset$ and thus also  $\supp (\partial T_L) \cap \bC_{32 r_L} (p_L, \pi_H) = \emptyset$. 

In order to prove \eqref{e:app2} we argue as follows. First consider the chain of ancestors of $L$:= $L= L_j \subset L_{j-1} \subset
\ldots \subset L_{N_0} =: J$, where $J\in \sS^{N_0}$. We first show that $(\p_{\pi_0})_\sharp (T_J \res \bC_{36 r_J} (p_J, \pi_0)) =
Q \a{B_{36 r_J} (z_J, \pi_0)}$. This is done in the following way: consider that $\gr (u)\cap \bC_{64r_J} (p_J, \pi_0)$ is the graph of a $C^{1,\alpha}$ function $v$ with $\|v\|_{C^{1,\alpha}}\leq C_0 \bmo^{\sfrac{1}{2}}$. 
Define the function $v_t (x) := t v (x)$ and let $\p_t$ be the orthogonal projection onto $\gr (v_t)$, which is well-defined on $\bV_J$ provided $\bmo$ is sufficiently small (the smallness being independent of $J$). The currents 
$S_t := (\p_t)_\sharp (T_J \res  \bC_{64r_J} (p_J, \pi_0))$
are easily seen to coincide with $Q_t \bG_v \res \bC_{36 r_J} (z_J, \pi_0)$ for some integers $Q_t$ in the cylinder $\bC_{36 r_J} (p_J, \pi_0)$ by the constancy theorem. On the other hand such currents vary continuously and thus the integer $Q_t$ must be constant. This implies that $Q_0 = Q_1 = Q$. On the other hand $\p_0 = \p_{\pi_0}$ and we have thus proved our claim.

Observe that $(\p_{\pi_0})_\sharp (T_L \res \bC_{36 r_L} (p_L, \pi_0)) = Q \a{B_{36r_L} (z_L, \pi_0)}$ because  $T_L \res \bC_{36r_L} (p_L, \pi_0) = T_J \res \bC_{36 r_L}  (p_L, \pi_0)$.
Choose next a continuous path of planes $\pi_t$ which connects $\pi_0$ and $\pi_H$ and satisfies the bound $|\pi_t - \pi_0|\leq C_0 |\pi_H-\pi_0|$ for some geometric constant $C_0$. We then look at $Z_t = (\p_{\pi_t})_\sharp (T_L \res \bC_{36r_L} (p_L, \pi_0))$ and conclude, similarly to the previous paragraph, that
$((\p_{\pi_H})_\sharp (T_L \res \bC_{36r_L} (p_L, \pi_0)))\res \bC_{32 r_L} (p_L, \pi_H) = Q \a{B_{32r_L} (\p_{\pi_H} (p_L), \pi_H)}$. On the other hand since $(T_L \res \bC_{36r_L} (p_L, \pi_0)))\res \bC_{32 r_L} (p_L, \pi_H) = T_L \res \bC_{32 r_L} (p_L, \pi_H)$, this concludes the proof of \eqref{e:app2}.

Now, by Proposition \ref{p:tilting opt}(iv)\&(vii) 
\[
\bE (T_L, \bC_{32 r_L} (p_L, \pi_H)) \leq \bar C \bE (T_L, \bB_L) + \bar C |\pi_L-\pi_H|^2
\leq \bar C \bmo \d (L)^{2\gamma_0 -2 + 2 \delta_1} \ell (L)^{2-2\delta_1} \leq \bar C \bmo
\]
and in order to apply \cite[Theorem~1.5]{DSS2} we just need to choose $\eps_2$ sufficiently small.
\end{proof}

We next generalize slightly the terminology of Section \ref{ss:whitney}. 

\begin{definition}\label{d:f_HL}
Let $H$ and $L$ be as in Proposition \ref{p:app}. After applying \cite[Theorem~1.5]{DSS2} to $T_L\res \bC_{32r_L} (p_L, \pi_H)$ in the cylinder $\bC_{32 r_L} (p_L, \pi_H)$ we denote by $f_{HL}$ the corresponding $\pi_H$-approximation. However, rather then defining $f_{HL}$ on the disk $B_{8 r_L}  (p_L, \pi_H)$, by applying a translation we assume that the domain of $f_{HL}$ is the disk $B_{8r_L} (p_{HL}, \pi_H)$ where $p_{HL} = p_H + \p_{\pi_H} (p_L - p_H)$. Note in particular that $\bC_r (p_{HL}, \pi_H)$ equals $\bC_r (p_L, \pi_H)$, whereas $B_{8r_L} (p_{HL}, \pi_H) \subset p_H + \pi_H$ and $p_H\in \bB_{8r_L}(p_{HL},\pi_H)$. 
\end{definition}

Observe that $f_{LL} = f_L$.

\subsection{First variations}\label{ss:ellittico} The next proposition is the core in the construction of the center manifold and it is the main reason behind the $C^{3,\gamma_0}$ estimate for the glued interpolation. It is also the place where our proof differs most from that of \cite{DS4}. 

\begin{definition}\label{d:bar}
Let $H$ and $L$ be as in Proposition \ref{p:app}. In the cases (a) and (c) of Definition \ref{d:semicalibrated} we denote by $\varkappa_H$ the orthogonal complement in $T_{p_H} \Sigma$ of $\pi_H$ and we denote by $\bar{f}_{HL}$ the map $\p_{\varkappa_H} \circ f_{HL}$.
\end{definition}

In what follows we will consider elliptic systems of the following form. Given a vector valued map $v: p_H+ \pi_H\supset \Omega \to \varkappa_H$ and after introducing an orthonormal system of coordinates $x^1, x^2$ on $\pi_H$ and $y^1, \ldots , y^{\bar n}$ on $\varkappa_H$,  the system is given by the $\bar{n}$ equations
\begin{equation}\label{e:elliptic}
\Delta v^k + \underbrace{({\bf L}_1)^k_{ij} \partial_j v^i + ({\bf L}_2)^k_i v^i}_{=:\mathscr{E}^k (v)} = \underbrace{({\bf L}_3)^k_i (x-x_H)^i +
({\bf L}_4)^k}_{=:\mathscr{F}^k} \, ,
\end{equation}
where we follow Einstein's summation convention and the tensors $\bL_i$ have constant coefficients. After introducing the operator $\mathscr{L} (v) = \Delta v + \mathscr{E} (v)$ we summarize the corresponding elliptic system \eqref{e:elliptic} as
\begin{equation}\label{e:ellittico}
\mathscr{L} (v) = \mathscr{F}\, .
\end{equation}
We then have a corresponding  weak formulation for $W^{1,2}$ solutions of \eqref{e:ellittico}, namely $v$ is a weak solution in a domain $D$ if the integral
\begin{equation}\label{e:weak_form}
\mathscr{I} (v, \zeta) := \int (Dv : D\zeta + (\mathscr{F}-  \mathscr{E} (v))\cdot \zeta)\, 
\end{equation}
vanishes for smooth test functions $\zeta$ with compact support in $D$.

\begin{proposition}\label{p:pde}
Let $H$ and $L$ be as in Proposition \ref{p:app} (including the possibility that $H=L$) and
let $f_{HL}$, $\bar{f}_{HL}$ and $\varkappa_{H}$ be as in Definition \ref{d:f_HL} and Definition \ref{d:bar}.
Then, there exist tensors with constant coefficients 
$\bL_1, \ldots, \bL_4$ and a constant $C=C(M_0,N_0,C_e,C_h)$,  with the following properties:
\begin{itemize}
\item[(i)] The tensors depend upon $H$ and $\Sigma$ (in the cases (a) and (c) of Definition \ref{d:semicalibrated}) or $\omega$ (in case (b) of Definition \ref{d:semicalibrated}) and $|{\bf L}_1| + |{\bf L}_2| + |{\bf L}_3| + |{\bf L_4}|\leq C \bmo^{\sfrac{1}{2}}$.
\item[(ii)] If $\mathscr{L}_H$, $\mathscr{I}_H$, and $\mathscr{F}_H$ are defined through \eqref{e:elliptic}, \eqref{e:ellittico} and \eqref{e:weak_form}, then
\begin{align}\label{e:pde2}
\mathscr{I}_{H} (\etaa \circ \bar{f}_{HL}, \zeta) &\leq C \bmo \,\d (L)^{2(1+\beta_0)\gamma_0-2-\beta_2}\,
r_L^{4+\beta_2}\,\|D\zeta\|_{0}
\end{align}
for all $\zeta\in C^\infty_c (B_{8r_L}(p_{HL}, \pi_H), \varkappa_H)$.
\end{itemize}
\end{proposition}

\begin{proof}
Set for simplicity $\pi= \pi_{H}$, $\varkappa:=\varkappa_{H}$
$r= r_L$, $p= p_{HL}$, $f=f_{HL}$, $B= B_{8r} (p, \pi)$ and $T=T_{L}$. 

\smallskip

\noindent\textit{Cases (a) and (b) of Definition \ref{d:semicalibrated}}. The proof is very similar to the one
of \cite[Proposition~5.2]{DS4}. Nevertheless, for the sake of completeness,
we give here all the details.
We fix  a system of coordinates 
$(x,y,w)\in \pi\times \varkappa \times (T_{p_{H}} \Sigma)^\perp$
so that $p_{H} = (0,0,0)$.
We drop the subscript $p_{H}$ for the map $\Psi_{p_{H}}$.
Recall that, by Lemma \ref{l:piccolezza}, 
\[
\Psi (0,0)=0, \quad D\Psi (0,0)=0\quad\text{and}\quad
\|D \Psi\|_{C^{2,\eps_0}}\leq C \bmo^{\sfrac{1}{2}}. 
\]
Let $\zeta \in C_c(B_{8r}(p, \pi), \varkappa)$ be a test function.
We consider the vector field $\chi:\Sigma \to \R^{2+ n}$ given by
$\chi (q) = (0, \zeta (x), D_y \Psi (x,y) \cdot  \zeta(x))$
for every $q = (x,y,\Psi (x,y))\in \Sigma$.
Note that $\chi$ is tangent to $\Sigma$. 
Therefore we infer that $\delta T (\chi) =0$ and 
\begin{align}
|\delta \bG_f (\chi)| = |\delta \bG_f (\chi) - \delta T (\chi)| 
\leq C \int_{\bC_{8r} (p, \pi)} |D\chi|\, d \|\bG_f - T\|\, .\label{e:var_prima}
\end{align}
Observe also that $|\chi| \leq C |\zeta|$ and $|D\chi| \leq C |\zeta| + C |D\zeta|
\leq C |D\zeta| $. Set $E:= \bE \big(T, \bC_{32 r} (p, \pi)\big)$. 
Thus, by \eqref{e:app4} and Proposition~\ref{p:tilting opt}(vii) we have
\begin{align}
E \leq & C \bmo \d (L)^{2\gamma_0-2+2\delta_1} \ell (L)^{2-2\delta_1}\, ,
\bh (T, \bC_{32r} (p, \pi)) \leq & C \bmax_0^{\sfrac14}\d (L)^{\sfrac{\gamma_0}{2} - \beta_2}\ell(L)^{1+\beta_2}\, .
\end{align}
Recall that, by \cite[Theorem 1.5]{DSS2} we have
\begin{gather}
|Df| \leq C E^{\beta_0} + C \bmo^{\sfrac{1}{2}} r
\leq C \bmo^{\beta_0}\,\d (L)^{(2\gamma_0-2+2\delta_1)\beta_0}
r^{\beta_0(2-2\delta_1)}\label{e:lip f}\\
|f| \leq C \bh (T, \bC_{32 r} (p, \pi)) + (E^{\sfrac{1}{2}} + r\,\bmo^{\sfrac{1}{2}}) r
 \leq C \bmo^{\sfrac{1}{4}}\,\d (L)^{\sfrac{\gamma_0}{2}-\beta_2} r^{1+\beta_2},\label{e:Linf}\\
\int_B |Df|^2 \leq C\, r^2 E \leq C \bmo \,\d (L)^{2\gamma_0-2+2\delta_1}\, r^{4-2\delta_1},\label{e:Dir f}
\end{gather}
(although \eqref{e:Dir f} is not explicitely stated in \cite[Theorem 1.5]{DSS2} it follows from the estimates therein as in
\cite[Remark 4.5]{DS3}) and 
\begin{align}
&|B\setminus K| \leq C \bmo^{1+\beta_0}\,\d(L)^{(1+\beta_0)(2\gamma_0-2+2\delta_1)} r^{2 +(1+\beta_0) (2 - 2\delta_1)}\, ,\label{e:aggiuntiva_1}\\
&\left| \|T\| (\bC_{8r} (p, \pi)) - |B| - \frac{1}{2} \int_B |Df|^2\right| \leq C \bmo^{1+\beta_0}\,\d(L)^{(1+\beta_0)(2\gamma_0-2+2\delta_1)} r^{2 +(1+\beta_0) (2 - 2\delta_1)}&\, , \label{e:aggiuntiva_2}
\end{align}
where $K\subset B$ is the set
\begin{equation}\label{e:recall_K}
B\setminus K = \p_\pi \left(\left(\supp (T)\Delta \supp(\bG_f)\right)\cap \bC_{8r_L} (p_L, \pi)\right)\, .
\end{equation}
Writing  $f= \sum_i \a{f_i}$ and $\bar{f} = \sum_i \a{\bar{f}_i}$,
since $\gr (f) \subset \Sigma$, we have
$f = \sum_i \a{(\bar{f}_i, \Psi (x, \bar{f}_i))}$.
From \cite[Theorem~4.1]{DS2} we can infer that
\begin{align}
\delta \bG_f (\chi) &= 
\int_B \sum_i \Big( \underbrace{D_{xy}\Psi (x, \bar{f}_i)\cdot \zeta}_{(A)} + 
\underbrace{(D_{yy} \Psi (x, \bar{f}_i)\cdot D\bar{f}_i) \cdot \zeta}_{(B)} + \underbrace{D_y \Psi (x, \bar{f}_i) \cdot D_x \zeta}_{(C)}\Big)\nonumber\\
&\qquad\quad
: \Big(\underbrace{D_x \Psi (x, \bar{f}_i)}_{(D)} + \underbrace{D_y \Psi (x, \bar{f}_i)\cdot D \bar{f}_i}_{(E)}\Big)
+ \underbrace{\int_B \sum_i D\zeta : D\bar f_i}_{(F)} + {\rm Err}\, ,\label{e:tayloraccio}
\end{align}
where the error term ${\rm Err}$ in 
\eqref{e:tayloraccio} satisfies the inequality
\begin{align}\label{e:tayloraccio2}
|{\rm Err}|&\leq C \int |D\chi| |Df|^3 \leq \|D\zeta\|_{L^\infty} \||Df|\|_{L^\infty} \int |Df|^2 \notag \\
&\leq C
\|D\zeta\|_{0} \bmo^{1+\beta_0}\,
\d (L)^{(1+\beta_0)(2\gamma_0-2+2\delta_1)} r^{4-2\delta_1+\beta_0(2-2\delta_1)}\, .
\end{align}
The integral (F) in \eqref{e:tayloraccio} is 
\begin{equation}\label{e:whatisF}
(F) = Q \int_B D\zeta : D (\etaa \circ \bar f)\, .
\end{equation}
We therefore expand the product in the other integral and estimate
all terms separately, using the Taylor expansion
\[
D \Psi (x,y) = D_x D \Psi (0,0) \cdot x + D_y D\Psi (0,0) \cdot y + O \big(\bmo^{\sfrac{1}{2}} (|x|^2 + |y|^2)\big)
\]
so that
\begin{gather}
|D \Psi (x, \bar{f}_i)| \leq C \bmo^{\sfrac{1}{2}} r\notag\\
D\Psi (x, \bar{f}_i) = D_x D\Psi (0,0)\cdot x + O \big(\bmo^{\sfrac{1}{2}+\sfrac{1}{4}} 
\d (L)^{\sfrac{\gamma_0}{2}-\beta_2} r^{1+\beta_2}\big),\notag\\\
|D^2 \Psi (x, \bar{f}_i)| \leq C \bmo^{\sfrac{1}{2}}\quad\mbox{and}\quad D^2 \Psi (x, \bar{f}_i) = D^2 \Psi (0,0) + O \big(\bmo^{\sfrac{1}{2}} r\big)\, .\nonumber
\end{gather}
We compute as follows:
\begin{align}
\int \sum_i (A):(D) &= \int \sum_i (D_{xy} \Psi (0,0) \cdot \zeta) : D_x \Psi (x, \bar{f}_i) + O\Big( \bmo\, r^2 \int |\zeta|\Big) \nonumber\\
&= \int Q (D_{xy} \Psi (0,0)\cdot \zeta) : (D_{xx} \Psi (0,0)\cdot x) \label{e:(AD)}\\
& \quad +  O \Big( \bmo\,\d (L)^{\sfrac{\gamma_0}{2}-\beta_2}\, r^{1+\beta_2} \int |\zeta|\Big).\notag
\end{align}
The  integral in \eqref{e:(AD)} has the form $\int \bL_{AD}\,x \cdot \zeta$.
Next, we estimate
\begin{align}
\int \sum_i \big((A):(E) + &(B):(D)+ (B) : (E) \big) \notag\\
&  = O \Big(\bmo^{1+\beta_0} \d (L)^{\beta_0(2\gamma_0-2+2\delta_1)}r^{1+\beta_0(2-2\delta_1)} \int |\zeta|\Big)\label{e:(AE)}
\end{align}
and
\begin{gather}
% \int \sum_i (B):\left((D)+(E)\right) = O\Big(\bmo^{1+\beta_0} \d (L, 0)^{\beta_0(2\min\{\gamma_0-1,0\}+2\delta_1)}r^{1+\beta_0(2-2\delta_1)}\int |\zeta|\Big),\label{e:(B(D+E))}\\
\int \sum_i (C):(E) = O \Big(\bmo^{1+\beta_0} \d (L)^{\beta_0(2\gamma_0-2+2\delta_1)}r^{2+\beta_0(2-2\delta_1)}\int |D\zeta|\Big).\label{e:(CE)}
\end{gather}
Finally we compute 
\begin{align*}
\int \sum_i (C):(D) &= \int \sum_i ((D_{xy}\Psi (0,0) \cdot x) \cdot D_x \zeta) : D_x \Psi (x, \bar{f}_i) \\
& \quad + O \Big(\bmo\,\d (L)^{\sfrac{\gamma_0}{2}-\beta_2}r^{2+\beta_2} \int |D\zeta|\Big)\\
&= Q \int (D_{xy}\Psi (0,0) \cdot x) \cdot D_x \zeta) : (D_{xx} \Psi (0,0) \cdot x) \\
& \quad + O \Big(\bmo\,\d (L)^{\sfrac{\gamma_0}{2}-\beta_2}r^{2+\beta_2} \int |D\zeta|\Big).
\end{align*}
Integrating by parts in the last integral we reach
\begin{equation}\label{e:(CD)}
\int \sum_i (C):(D) = \int \bL_{CD}\,x \cdot \zeta + O \Big(\bmo\,\d (L)^{\sfrac{\gamma_0}{2}-\beta_2} r^{2+\beta_2} \int |D\zeta|\Big)\, .
\end{equation}
Set next $\bL_3 :=\bL_{AD} + \bL_{CD}$. Clearly $\bL_3$ is a quadratic
function of $D^2 \Psi (0,0)$, i.e.~a
quadratic function of the tensor $A_\Sigma$ at the point $p_H$.
From \eqref{e:var_prima}, \eqref{e:tayloraccio2}, \eqref{e:(AD)} -- \eqref{e:(CD)},
we infer \eqref{e:pde2} and (i). Indeed we have to compare the following three types of errors
\begin{gather}
\mathcal{E}_1:= \bmo^{1+\beta_0} \d (L)^{(1+\beta_0)(2\gamma_0-2+2\delta_1)}r^{4-2\delta_1+\beta_0(2-2\delta_1)}\\
\mathcal{E}_2:= \bmo^{1+\beta_0} \d (L)^{\beta_0(2\gamma_0-2+2\delta_1)}r^{4+\beta_0(2-2\delta_1)}\label{e:calE2}\\
\mathcal{E}_3:= \bmo \d (L)^{\sfrac{\gamma_0}{2}-\beta_2}r^{4+\beta_2}
\end{gather}
with $\bmo \d (L)^{2(1+\beta_0)\gamma_0-2 - \beta_2}
r^{4+\beta_2}$.
It is easy to see that if
\begin{equation}\label{e:condizione beta2}
-2\delta_1+\beta_0(2-2\delta_1) - \beta_2 >0
\end{equation}
then
\begin{align}\label{e:finale}
\mathcal{E}_2 \leq \mathcal{E}_1 & \leq \bmo^{1+\beta_0} \d (L)^{(1+\beta_0)(2\gamma_0-2+2\delta_1)-2\delta_1+\beta_0(2-2\delta_1) - \beta_2}r^{4+\beta_2}\notag\\
& \leq \bmo^{1+\beta_0} \d (L)^{2(1+\beta_0)\gamma_0-2 - \beta_2} r^{4+\beta_2}
\end{align}
Therefore
\begin{equation}\label{e:calE123}
 \mathcal{E}_1,\mathcal{E}_2,\mathcal{E}_3\leq \bmo \d (L)^{2(1+\beta_0)\gamma_0-2 - \beta_2}
r^{4+\beta_2}\,.
\end{equation}
To conclude the proof we observe that, by the bound \eqref{e:Dir f} and \cite[Theorem 5.2]{DSS2}
\begin{align*}
\int_{\bC_{8r} (p, \pi)} |D\chi|\, d \|\bG_f - T\| &\leq C \|D\zeta\|_0 \mass (T\res\bC_{8r} (p,\pi) -\bG_f)\\
& \leq C_0 \|D\zeta\|_0 r^2 E^{\beta_0} (E+ \bmo r^2) \leq C \mathcal{E}_2\,. 
\end{align*}

\medskip

\noindent\textit{Case (b) of Definition \ref{d:semicalibrated}}.
Fix coordinates $(x,y) \in \R^2\times\R^n$ such that $p_{H} = (0,0)$.
Consider the vector field $\chi(x,y) := (0, \zeta(x))$
for some $\zeta$ as in the statement. Recalling \cite[Proposition 1.2]{DSS1} we infer
\[
 \delta \bG_f (\chi) = \delta T (\chi) + \textup{Err}_0=T(d\omega\ser \chi)  +\textup{Err}_0= \bG_f(d\omega\ser \chi)+\textup{Err}_0 +\textup{Err}_1
\]
with
\begin{align}
|\textup{Err}_0| + |\textup{Err}_1| &= |\delta T(\chi) - \delta \bG_f(\chi)| 
+ \big\vert T(d\omega \ser \chi) - \bG_f(d\omega \ser \chi)\big\vert\notag\\
&\leq C\, \big(\|D\zeta\|_{0} + \|d\omega\ser \chi\|_{0} \big)
\| T - \bG_f\|(\bC_{8r}(p,\pi))\notag\\
& \leq C\,\big(\|D\zeta\|_{0} + \|\zeta\|_{0} \big)
\,E^{\beta_0}\, (E + r^2\,\bmo)\, r^2 \notag\\
&\leq C\,\|D\zeta\|_{0}\,\bmo^{1+\beta_0}
\,\d (H)^{(2\gamma_0-2+2\delta_1)(1+\beta_0)}
\,r^{2 + (2 - 2\delta_1) (1+ \beta_0)}.
\end{align}
From \cite[Theorem~4.1]{DS2}
\[
\delta \bG_f(\chi) = Q \int D(\etaa\circ f) \colon D\zeta + \textup{Err}_2
\]
with
\begin{align*}
|\textup{Err}_2| &\leq C\,\int |D\zeta|\, |Df|^3\leq C\,\|D\zeta\|_{0} E^{1+\beta_0}\,r^2\\
& \stackrel{\eqref{e:Dir f}}{\leq} C\, \|D\zeta\|_{0}\,\bmo^{1+\beta_0}\,\d (H)^{(2\gamma_0-2+2\delta_1)(1+\beta_0)}\,
r^{2 + (2-2\delta_1)(1+\beta_0)}.
\end{align*}
Next we proceed to expand $\bG_f(dw\ser \chi)$. To this aim we write
\begin{align}
d\omega(x,y) & = \sum_{l=1}^n a_l (x,y)\, dy^l \wedge dx^1\wedge dx^2  +
\sum_{j=1,2}\sum_{l<k} b_{lk,j}(x,y)\, dy^l \wedge dy^k \wedge dx^j\notag\\
&
%+ \sum_{l<k} \omega_{lk,2}(x,y)^{(2)}\,dx_2\wedge dy^l \wedge dy^k
+\sum_{l<k<j} c_{lkj}(x,y)\,dy^l \wedge dy^k \wedge dy^j
\end{align}
and get
\begin{align}
d\omega \ser \chi & = \underbrace{\sum_{l=1}^n a_l\,\zeta^l\,
dx^1\wedge dx^2}_{\omega^{(1)}}
+ \underbrace{\sum_{j=1,2}\sum_{l<k} b_{lk,j}\, \zeta^l dy^k \wedge dx^j}_{\omega^{(2)}}
+\underbrace{\sum_{l<k<j} c_{lkj}\,\zeta^l \, dy^k \wedge dy^j}_{\omega^{(3)}}.
\end{align}
We consider separately $\bG_f(\omega^{(1)}), \bG_f(\omega^{(2)}), \bG_f(\omega^{(3)})$.
We start with the latter
\begin{align}
\bG_f(\omega^{(3)}) & \leq C\,\|d\omega\|_0\,\|\zeta\|_0\int_{B}|Df|^2 \stackrel{\eqref{e:Dir f}}{\leq}
C\,\bmo^2\, \d (H)^{2\gamma_0-2+2\delta_1}\, r^{5-2\delta_1} \|D\zeta\|_0.
\end{align}
Next
\begin{align}
\bG_f(\omega^{(2)}) & = \sum_{l<k} \sum_{i=1}^{Q}
\int \zeta^l(x) \,\left(
b_{lk,2}(x,f_i(x))\, \frac{\de f_i^k}{\de x^1} 
-b_{lk,1}(x,f_i(x))\, 
\frac{\de f_i^k}{\de x^2} \right) dx\notag\\
& = Q\sum_{l<k} 
\int \zeta^l(x) \,\left(
b_{lk,2}(0,0)\, \frac{\de (\etaa\circ f)^k}{\de x^1} 
-b_{lk,1}(0,0)\, \frac{\de (\etaa\circ f)^k}{\de x^2} \right)dx
+ \textup{Err}_3,\notag\\
& = \int \bL_1\,D(\etaa \circ f) \cdot \zeta + \textup{Err}_3
\end{align}
with
\begin{align}
|\textup{Err}_3| & \leq C \|\zeta\|_0\,\|D(d\omega)\|_0
\int_B \left(r\,|Df| + |f|\,|Df| \right)\, dx\notag \\
%&\leq C \|D\zeta\|_0\,\bmax_0\, 
%\big(r + \osc(f) + \bh(T, \bC_{8r}(0,\pi))  \big)\, r^3\, E^{\beta_0}\notag \\
&\stackrel{\eqref{e:lip f}\&\eqref{e:Linf}}\leq C \|D\zeta\|_0\,\bmax_0^{1+\beta_0}\, r^{4+(2-2\delta_1)\beta_0}\, \d (H)^{(2\gamma_0-2+2\delta_1)\beta_0}
\end{align}
and $\bL_1 : \R^{n\times 2} \to \R^n$ given by
\[
\bL_1 A\cdot e_l := Q\sum_{k=1}^n \big(b_{lk,2}(0,0)\, A_{k1} 
-b_{lk,1}(0,0)\, A_{k2}\big) \quad \forall\;A=(A_{kj})_{k=1,\ldots,n}^{j=1,2}\in \R^{n\times 2}.
\]
Finally
\begin{align}
\bG_f(\omega^{(1)}) & = \sum_{l} \sum_{i=1}^{Q}
\int \zeta^l(x) \,
a_{l}(x,f_i(x))\,dx\notag\\
& = Q\sum_{l} \int \zeta^l(x) \,
\left(a_{l}(0,0) + D_x a_l(0,0) \cdot x +
\,D_y a_l(0,0) \cdot (\etaa\circ f) \right) dx + \textup{Err}_4,\notag\\
& = \int \big(\bL_2\,(\etaa\circ f)  + \bL_3\, x + \bL_4 \big) \cdot \zeta
+ \textup{Err}_4
\end{align}
where $\bL_2:\R^n \to \R^n$, $\bL_3:\R^2 \to \R^n$
 $\bL_4\in \R^n$ are given by
\begin{gather}
\bL_2 \,v \cdot e_l := \sum_{k=1}^n \frac{\de a_l}{\de y^k}(0,0) \, v^k
\quad \forall\;v\in\R^n,\;\forall\;l=1,\ldots,n\\
\bL_3\, w \cdot e_l := \sum_{j=1}^2 \frac{\de a_l}{\de x^j}(0,0) \, w^j
\quad \forall\;w\in\R^n,\;\forall\;l=1,\ldots,n\\
\bL_4 \cdot e_l := a_l(0,0)\quad\forall\;l=1,\ldots,n
\end{gather}
and arguing as above
\begin{align}
|\textup{Err}_4| \leq C \|\zeta\|_0\,[D (d\omega)]_{\eps_0}
\int_B \left(r^{1+\eps_0} + |f|^{1+\eps_0} \right) \,dx
%\leq C \|D\zeta\|_0\,\bmo\, \big(r^2 + \osc(f)^2\,  \big)\, r^3
\leq C \|D\zeta\|_0\,\bmo\, r^{4+\eps_0}.
\end{align}
In order to deduce \eqref{e:pde2} we need to compare, using \eqref{e:calE2} and \eqref{e:calE123}, 
\begin{gather*}
|\textup{Err}_0+\textup{Err}_1+\textup{Err}_2|\leq\|D\zeta\|_0 \mathcal{E}_{1} \leq \|D\zeta\|_0 \bmo \d (L)^{2(1+\beta_0)\gamma_0-2 - \beta_2}
r^{4+\beta_2}
\\
\mathcal{E}_2 = C\,\bmax_0^{1+\beta_0}\,\d (H)^{(2\gamma_0-2+2\delta_1)\beta_0}\,r^{4+(2-2\delta_1)\beta_0}
\\
\mathcal{E}_{2.5}= C\,\bmo^2\, \d (H)^{2\gamma_0-2+2\delta_1}\, r^{5-2\delta_1}
\\
\mathcal{E}_4 = C\,\bmo\, r^{4+\eps_0}
\end{gather*}
with $\bmo \d (L)^{2(1+\beta_0)\gamma_0-2 - \beta_2}
r^{4+\beta_2}$.
As before, if \eqref{e:condizione beta2} holds, then $\mathcal{E}_{2}\leq \mathcal{E}_{1}$.
Moreover, since $\mathcal{E}_4 \leq r^{4+\beta_2}$, to conclude \eqref{e:pde2} it is enough to observe that if
\begin{equation}\label{condizione_beta0}
1 \geq \beta_0(2-2\delta_1)
\end{equation} 
then $0> 2\gamma_0-2+2\delta_1 > (2\gamma_0-2+2\delta_1)(1+\beta_0) $ and 
$5-2\delta_1 > 2 + (2 - 2\delta_1) (1+ \beta_0)$, so that
\begin{align*}
\d (H)^{2\gamma_0-2+2\delta_1}\, r^{5-2\delta_1} \leq
\d (H)^{(2\gamma_0-2+2\delta_1)(1+\beta_0)}\,r^{2 + (2 - 2\delta_1) (1+ \beta_0)}\,,
\end{align*}
that is $\mathcal{E}_{2.5\&4}\leq \mathcal{E}_1$.
\end{proof}

\subsection{Tilted interpolating functions, $L^1$ and $L^\infty$ estimates} In this subsection we generalize the definition of the tilted interpolating functions $h_L$. More precisely we consider

\begin{definition}\label{d:tilted_HL}
Let $H$ and $L$ be as in Proposition \ref{p:app}, assume that the conclusions of Proposition \ref{p:pde} applies and let $\mathscr{L}_H$ and $\mathscr{F}_H$ be the corresponding operator and map as given by Proposition \ref{p:pde} in combination with \eqref{e:elliptic}, \eqref{e:ellittico} and \eqref{e:weak_form}. Let $f_{HL}$ be as in Definition \ref{d:f_HL}, $\varkappa_H$ and $\bar{f}_{HL}$ be as in Definition \ref{d:bar} and fix coordinates $(x,y,z)\in \pi_H \times \varkappa_H \times T_{p_H} \Sigma^\perp$ as in the proof of Proposition \ref{p:pde}. We then let $\bar{h}_{HL}$ be the solution of
\begin{equation}\label{e:ell_def}
\left\{
\begin{array}{l}
\mathscr{L}_H \bar{h}_{HL} = \mathscr{F}_H\\ \\
\left.\bar{h}_{HL}\right|_{\partial B_{8r_L} (p_{HL}, \pi_H)} = \etaa \circ \bar{f}_{HL}\, .
\end{array}\right.
\end{equation}
In case (b) of Definition \ref{d:semicalibrated} we then define $h_{HL} = \bar{h}_{HL}$, whereas in the other cases
we define $h_{HL} (x) = (\bar{h}_{HL} (x), \Psi_{p_H} (x, \bar{h}_{HL} (x)))$.
\end{definition}

In order to show that the maps $\bar{h}_{HL}$ are well defined, we need to show that there is a solution of the system \eqref{e:ell_def}.

\begin{lemma}\label{l:exist}
Under the assumptions of Definition \ref{d:tilted_HL}, if $\eps_2$ is sufficiently small, then the elliptic system
\begin{equation}\label{e:esiste}
\left\{
\begin{array}{l}
\mathscr{L}_H v = F\ \\ \\
\left.v\right|_{\partial B_{8r_L} (p_{HL}, \pi_H)} = g\, .
\end{array}\right.
\end{equation}
has a unique solution for every $F\in W^{-1,2}$ (the latter being the usual space of distributions which can be represented as first partial derivatives of $L^2$ functions) and every $g\in W^{1,2} (B_{8 r_L} (p_{HL}, \pi_H))$. Observe moreover that $\|Dv\|_{L^2} \leq C_0 r_L (\|F\|_{L^2} + \bmo^{\sfrac{1}{2}} \|g\|_{L^2}) + C_0 \|Dg\|_{L^2}$ whenever $F\in L^2$. 
\end{lemma}
\begin{proof} As for the first assertion,
it suffices to show the Lemma for $g=0$, since we can define $w= v-g$ and solve $\mathscr{L}_H (w) = F + \mathscr{L}_H (g)$ . Setting $B = B_{8r_L} (p_{HL}, \pi_H)$, the existence and uniqueness for the latter case reduces, by Lax-Milgram, to the coercivity of the suitable quadratic form $\mathscr{Q} (v,v)$ on $W^{1,2}_0 (B)$. 
The latter follows easily from 
\begin{align*}
\mathscr{Q} (w,w) &:= \int (|Dw|^2 - \bL_1\,Dw\cdot w - \bL_2\, w\cdot w)\\
&\geq
\|Dw\|_{L^2(B)}^2 - \frac{|\bL_1|}2\,\|Dw\|_{L^2(B)}^2 - \left(\frac{|\bL_1|}2 +|\bL_2| \right)\|w\|_{L^2(B)}^2\, .
\end{align*}
Since $r_L \leq 1$, by the Poincar\'e inequality $\|w\|_{L^2}^2 \leq C_0 \|Dw\|_{L^2}^2$ for every $w\in W^{1,2}_0 (B)$.
The coercivity follows then from $|{\bf L}_1| + |{\bf L}_2| \leq C \bmo^{\sfrac{1}{2}} \leq C \eps_2$, where the constant $C$ depends only upon $M_0, N_0, C_e$ and $C_h$. In particular we can assume the coercivity factor to be $\frac{1}{2}$.

On the other hand, multiplying the equation by $w$ and integrating by parts we easily see (using the coercivity) that
\begin{align*}
\frac{1}{2} \int |Dw|^2 \leq &\int (|Dw||Dg| + |F||w|) + C \bmo^{\sfrac{1}{2}} \int (|g||w| + |w| |Dg|)\\
\leq &\frac{1}{4}\int |Dw|^2 + \frac{r_L^2}{\gamma} \int |F|^2 + \frac{2\gamma}{r_L^2} \int |w|^2
+ C \int (|Dg|^2 +\textstyle{\frac{\bmo}{\gamma}} r^2_L |g|^2)\, ,
\end{align*}
where $\gamma$ is any fixed positive number and $C$ does not depend upon it.

We choose $\gamma$ smaller than a geometric constant, so that we can use the Poincar\'e inequality to absorb the terms $\int |w|^2$ on the right hand side. We then conclude the desired estimate $\|Dw\|_{L^2} \leq C (\|Dg\|_{L^2} + \bmo^{\sfrac{1}{2}} r_L \|g\|_{L^2} + r_L \|F_L\|_{L^2})$. Since $v = w + g$, we then conclude $\|Dv\|_{L^2} \leq C (\|Dg\|_{L^2} + \bmo^{\sfrac{1}{2}} r_L \|g\|_{L^2} + C r_L \|F_L\|_{L^2})$ .
\end{proof}

Observe that $h_{HH} = h_H$. We next record three fundamental estimates, which regard, respectively, the $L^\infty$ norms of derivatives of solutions of $\mathscr{L}_H (v) = F$, the $L^\infty$ norm of $\bar{h}_{HL}- \etaa\circ \bar{f}_{HL}$ and the $L^1$ norm of $\bar{h}_{HL}-\etaa \circ \bar f_{HL}$. 

\begin{proposition}\label{p:stime}
Let $H$ and $L$ be as in Proposition \ref{p:pde} and assume the conclusions in there apply. Then the following estimates
hold for a constant $C = C(m_0, N_0, C_e, C_h)$ for $\hat{B} := B_{8r_L} (p_{HL}, \pi_H)$ and $\tilde{B}:=B_{6r_L} (p_{HL}, \pi_H)$:
\begin{align}
&\|\bar{h}_{HL} - \etaa\circ \bar{f}_{HL}\|_{L^1 (\hat{B})} \leq C \bmo \d (L)^{2(1+\beta_0) \gamma_0 -2 -\beta_2} \ell (L)^{5 + \beta_2}\label{e:stima-L1}\\
&\|\bar{h}_{HL} - \etaa \circ \bar{f}_{HL}\|_{L^\infty (\tilde{B})}\leq C \bmo \d (L)^{2(1+\beta_0) \gamma_0 -2 -\beta_2} \ell (L)^{3 + \beta_2}
+ C \bmo^{\sfrac{1}{2}} \ell (L)^2\, .\label{e:stima-Linfty}
\end{align}
Moreover, if $\mathscr{L}_H$ is the operator of Proposition \ref{p:pde}, $r$ a positive number no larger than $1$ and $v$ a solution of $\mathscr{L}_H (v) =F $ in $B_{8r} (q, \pi_H)$ for some smooth $F$, then
\begin{equation}\label{e:L1-Linfty}
\|v\|_{L^\infty (B_{6r} (q, \pi_H))} \leq \frac{C_0}{r^2} \|v\|_{L^1 (B_{8r} (q, \pi_H))} + C r^2 \|F\|_{L^\infty (B_{8r} (q, \pi_H))}\, 
\end{equation}
and, for $l\in \mathbb N$
\begin{equation}\label{e:higher}
\|D^l v\|_{L^\infty (B_{6r} (q, \pi_H))} \leq \frac{C_0}{r^{2+l}} \|v\|_{L^1 (B_{8r} (q, \pi_H))} + C r^2 \sum_{j=0}^l r^{j-l} \|D^j F\|_{L^\infty (B_{8r} (q, \pi_H))},
\end{equation}
where the latter constants depend also upon $l$. 
\end{proposition}

\begin{proof}
{\bf Proof of \eqref{e:L1-Linfty}}. The estimate will be proved for a linear constant coefficient operator of the form 
$\mathscr{L} = \Delta + {\bf L}_1 \cdot D + {\bf L}_2$ when ${\bf L}_1$ and ${\bf L}_2$  are sufficiently small. We can then assume $\pi_H= \R^2$ and $q=0$. Besides, if we define $u (x) := v (rx)$ we see that $u$ just satisfies $\Delta u + r \mathbf{L}_1 \cdot D u + r^2 \mathbf{L}_2 \cdot u = 0$ and thus, without loss of generality, we can assume $r=1$. We thus set $B= B_8 (0)\subset \R^2$.

We recall the following interpolation estimate on the ball of radius $1$, see \cite[Theorem 1]{Nir}. For $0\leq j\leq m$ and $\frac{j}{m}\leq a\leq 1$ we have, for a constant $C_0=C_0 (m,j,q,s)$,
 \begin{equation}\label{e:Nir}
 \|D^ju\|_{L^p(B_1)}\leq C \|D^mu\|^a_{L^s(B_1)}\;\|u\|^{1-a}_{L^q(B_1)}+C\;\|u\|_{L^q(B_1)}\, ,
 \end{equation}
 where  
 \[
  \textstyle \frac{1}{p}=\frac{j}{2}+a\big(\frac{1}{s}-\frac{m}{2}\big)+(1-a)\frac{1}{q}\, .
 \]
We apply the estimate \eqref{e:Nir} for $j=1$, $m=2$, $q=1$ and $p=s=2$, $a=\sfrac23$ and use Young's inequality and a simple scaling argument to achieve the inequality
\begin{equation}\label{e:una}
\|Du\|_{L^2 (B_\rho (x))} \leq C_0 \rho \|D^2 u\|_{L^2 (B_\rho (x))} + C_0 \rho^{-2} \|u\|_{L^1 (B_\rho (x))}\, .
\end{equation}
Moreover, by the Poincar\'e inequality:
\begin{equation}\label{e:due}
\|u\|_{L^2 (B_\rho (x))} \leq C_0 \rho \|Du\|_{L^2 (B_\rho (x))} + C_0 \rho^{-1} \|u\|_{L^1 (B_\rho (x))}\, .
\end{equation}
Next, recall the standard $L^2$ estimates for second order derivatives of solutions of the Laplace equations: if $B_{2\rho} (x)\subset B$, then
\begin{equation}\label{e:tre}
\|D^2 u\|_{L^2 (B_\rho (x))} \leq C_0 \|\Delta u\|_{L^2 (B_{2\rho} (x))} + C_0 \rho^{-3} \|u\|_{L^1 (B_{2\rho} (x))}\, .
\end{equation}
Now, recall that $\Delta u = - {\bf L}_1 \cdot Du - {\bf L}_2 \cdot u + F$. Using the fact that $|{\bf L}_1| + |{\bf L}_2| \leq C_0 \bmo^{\sfrac{1}{2}}$, we can combine all the inequalities above to conclude 
\begin{equation}\label{e:quattro}
\rho^6 \|D^2 u\|^2_{L^2 (B_\rho (x))} \leq C_0 \rho^6 \bmo \|D^2 u\|^2_{L^2 (B_{2\rho} (x))} + C_0\|u\|^2_{L^1 (B_8)} + C_0 \|F\|^2_{L^\infty (B_8)}\, .
\end{equation}
Define next
\begin{equation}
S := \sup \{ \rho^3 \|D^2 u \|_{L^2 (B_\rho (x))} : B_{2\rho (x)} \subset B_8\}
\end{equation}
and let $\varrho$ and $\xi$ be such that $B_{2 \varrho} (\xi) \subset B_8$ and 
\begin{equation}\label{e:ottimalita'}
\varrho^3 \|D^2 u\|_{L^2 (B_\varrho (\xi))} \geq \frac{S}{2}\, .
\end{equation}
We can cover $B_{\varrho} (\xi)$ with $\bar N_0$ balls $B_{\varrho/2} (x_i)$ with $x_i \in B_{\varrho} (\xi)$, where $\bar N_0$ is only a geometric constant. We then can apply \eqref{e:quattro} to conclude that
\[
\frac{S}{2} \leq C_0 \bar  N_0 \bmo^{\sfrac{1}{2}} S + C_0 \bar N_0 \|u\|_{L^1 (B_8)} + C_0 \bar N_0 \|F\|_{L^\infty (B_8)}\, .
\]
Therefore, when $\bmo^{\sfrac{1}{2}}$ is smaller than a geometric constant we conclude $S\leq C_0 \|u\|_{L^1 (B_8)}+C_0 \|F\|_{L^\infty (B_8)}$. By definition of $S$, we have reached the estimate
\[
\rho^3 \|D^2 u\|_{L^2 (B_\rho (x))} \leq C_0 \|u\|_{L^1 (B_8)} + C_0 \|F\|_{L^\infty (B_8)} \qquad \mbox{whenever $B_{2\rho} (x) \subset B_8$.}
\]
Of course, with a simple covering argument, this implies
\begin{equation}\label{e:cinque}
\|D^2 u\|_{L^2 (B_6)}\leq C_0 \|u\|_{L^1 (B_8)}  + C_0 \|F\|_{L^\infty (B_8)}\, .
\end{equation}
Next, again using the interpolation inequality \eqref{e:una} we get 
\[
\|Du\|_{L^2 (B_6)} \leq C_0 \|u\|_{L^1 (B_8)} + C_0 \|F\|_{L^\infty (B_8)}\, .
\] 
So, by Sobolev embedding 
\[
\|Du\|_{L^4 (B_6)} \leq C_0\|D u\|_{W^{1,2} (B_6)} \leq C_0\|u\|_{L^1 (B_8)} + C_0 \|F\|_{L^\infty (B_8)}\, .
\]
Again using interpolation and Sobolev we finally achieve
\[
\|u\|_{L^\infty (B_6)}\leq C_0 \|u\|_{W^{1,4} (B_6)} \leq C_0 \|u\|_{L^1 (B_8)} + C_0 \|F\|_{L^\infty (B_8)}\, .
\]

\medskip

{\bf Proof of \eqref{e:higher}.} As in the previous step, we can, without loss of generality, assume $r=1$. Note that a byproduct of the argument given above is also the estimate
\[
\|Du\|_{L^1 (B_6)} \leq C_0 \|u\|_{L^1 (B_{8})} + C_0 \|F\|_{L^\infty (B_8)}\, .
\]
In fact, by a simple covering and scaling argument one can easily see that 
\[
\|Du\|_{L^1 (B_\tau)}\leq C_0 (\tau) \|u\|_{L^1 (B_8)} + C_0 (\tau) \|F\|_{L^\infty (B_8)} \qquad \mbox{for every $\tau<8$.}
\] 
We can then differentiate the equation and use the proof of the previous paragraph to show 
\[
\|Du\|_{L^\infty (B_\sigma)} \leq C_0 (\sigma, \tau) \|Du\|_{L^1 (B_\tau)} + C_0 (\sigma, \tau) \|DF\|_{L^\infty (B_\tau)}\, .
\] 
Again, arguing as above, a byproduct of the proof is also the estimate
\[
\|D^2 u\|_{L^1 (B_\sigma)} \leq C_0 (\sigma, \tau) \|Du\|_{L^1 (B_\tau)} + C_0 (\sigma, \tau) \|DF\|_{L^\infty (B_\tau)}\, .
\] 
This can be applied inductively to get estimates for all higher derivatives.

\medskip

{\bf Proof of \eqref{e:stima-L1}.} Let $B := B_{8 r_L} (p_{HL}, \pi_H)$. We use the coordinates introduced in the proof of Proposition \ref{p:pde}. We set
$w:= \bar{h}_{HL} - \etaa\circ \bar f_{HL}$ and observe that
\[
\left\{
\begin{array}{l}
\mathscr{L} w = \mathscr{F}_H - \mathscr{L}_H (\etaa \circ \bar{f}_{HL})\\ \\
w|_{\partial B} =0
\end{array}
\right.
\]
Next, for $1<p<\infty$, we define the continuous (by Calderon-Zygmund theory) linear
operator $T: L^p(B) \to W^{1,p}_0(B) \cap W^{2,p} (B)$ by $T(g) = \psi$ where
\[
\begin{cases}
-\Delta \psi = g & \text{in } B\\ \\
\psi = 0 & \text{on } \partial B.
\end{cases}
\]
Applying the Sobolev embedding $W^{1,3}(B)\hookrightarrow C^0(B)$ to the derivative of $\zeta \in W^{2,3}\cap W^{1,3}_0$ we
conclude that
\[
\|D \zeta - {\textstyle{\mint}} D \zeta\|_0 \leq C_0 r_L^{1-\frac{2}{3}} \|D^2\zeta\|_{L^3}\, .
\]
On the other hand, by interpolation and Poincar\'e we conclude
\[
\|D\zeta\|_{L^3} \leq \|\zeta\|_{L^3}^{\sfrac{1}{2}}\|D^2 \zeta\|_{L^3}^{\sfrac{1}{2}} 
\leq \frac{\varepsilon}{2r_L} \|\zeta\|_{L^3} + \frac{r_L}{2\varepsilon}\|D^2 \zeta\|_{L^3}
\leq C_0 \varepsilon \|D\zeta\|_{L^3} +  \frac{r_L}{2\varepsilon}\|D^2 \zeta\|_{L^3}\, ,
\]
for every positive $\varepsilon$. Choosing the latter accordingly we achieve $\|D\zeta\|_{L^3} \leq C_0 r_L \|D^2 \zeta\|_{L^3}$ and
thus $\|D\zeta\|_0 \leq C_0 r_L \|D^2 \zeta\|_{L^3}$.

We now use these bounds in \eqref{e:pde2} to get 
\begin{align*}
\left|\int_B (Dw : D\zeta - \bL_1\,Dw\cdot \zeta - \bL_2\, w\cdot \zeta)\right|
\leq 
C \bmo \,\d (L)^{2(1+\beta_0)\gamma_0-2-\beta_2}\,
r_L^{4+\beta_2}\,r_L^{1-\frac{2}{3}}\,\|D^2\zeta\|_{L^3}.
\end{align*}
Then, we can estimate the $L^{\sfrac32}$-norm of $w$ as follows:
\begin{align*}
\|w\|_{L^{\sfrac32}(B)} & = \sup_{\|h\|_{L^3(B)}=1} \int_B w \,h
= - \sup_{\|h\|_{L^3(B)}=1} \int_B w \,\Delta\,T(h)\\
&\leq \sup_{\|h\|_{L^3(B)}=1} \int_B  D \,w \cdot D T(h)\\
&\leq C \bmo \,\d (L)^{2(1+\beta_0)\gamma_0-2-\beta_2}\,
r_L^{5+\beta_2-\sfrac23}\,\sup_{\|h\|_{L^3(B)}=1} \|D^2T(h)\|_{L^3}\\
&+ \sup_{\|h\|_{L^3(B)}=1} \int_B 
(- \bL_1\,Dw\cdot T(h) - \bL_2\, w\cdot T(h))\, .
\end{align*}
Recalling the Calderon-Zygmund estimates we have 
\begin{align*}
&\|D^2 T (h)\|_{L^3} \leq C_0 \|h\|_{L^3}\\ 
&\|D T (h)\|_{L^3} \leq C_0 r_L \|h\|_{L^3}\\
&\|T (h)\|_{L^3} \leq C_0 r_L^2 \|h\|_{L^3}\, . 
\end{align*}
Integrating by parts we then achieve
\begin{align*}
\|w\|_{L^{\sfrac32}(B)} &\leq C \bmo \,\d (L)^{2(1+\beta_0)\gamma_0-2-\beta_2}\,
r_L^{5+\beta_2-\sfrac23}
+\sup_{\|h\|_{L^3(B)}=1} \int_B w \cdot (\bL_1 DT(h) - \bL_2 T(h))\\
&\leq C \bmo \,\d (L)^{2(1+\beta_0)\gamma_0-2-\beta_2}\,
r_L^{5+\beta_2-\sfrac23} + C \bmo^{\sfrac{1}{2}} \|w\|_{L^{3/2} (B)}\, .
\end{align*}
Therefore, if $\bmo^{\sfrac{1}{2}}$ is sufficiently small, that is $\eps_2$ is sufficiently small, we deduce that
\begin{gather*}
\|w\|_{L^1}\leq C\,r_L^{\sfrac23}
\|w\|_{L^{\sfrac32}(B)} \leq C\,\bmo\,\dist(H)^{2(1+\beta_0)\gamma_0-2-\beta_2}\,r_L^{5+\beta_2}.\label{e:C-Z}
\end{gather*}

\medskip

{\bf Proof of \eqref{e:stima-Linfty}.} The estimate follows easily from \eqref{e:stima-L1} and \eqref{e:L1-Linfty}, recalling
that $\|\mathscr{F}_H\|_0 \leq C \bmo^{\sfrac{1}{2}}$. 
\end{proof}

\section{Main estimates on the interpolating functions}

In this section we adopt the terminology of the previous subsection and we show that

\begin{proposition}\label{p:main_est}
Assume the conclusions of Proposition \ref{p:app} applies,
let $\kappa := \sfrac{\beta_2}{4}$ and assume $\eps_2$ is sufficiently small, depending upon the other parameters.
Then there exists a constant $C=C(M_0, N_0, C_e, C_h)$ such that for any cube 
$H\in \sW\cup \sS$, the following conclusions hold.
\begin{itemize}
\item[(i)] Lemma \ref{l:tecnico3} applies and thus $g_H$ is well-defined.
\item[(ii)] The following estimates hold:
\begin{align}
&\|h_H- \p_{\pi_H}^\perp (p_H)\|_{C^0 (B_{6r_H} (p_H, \pi_H))} \leq C \bmo^{\sfrac{1}{4}} \d (H)^{\sfrac{\gamma_0}{2} - \beta_2} \ell (H)^{1+\beta_2}\label{e:sfava_tanto}\\
&\|g_H\|_{C^0}\leq C \bmo^{\sfrac14} \d(H)^{1+\sfrac{\gamma_0}{2}} \label{e:interpo_0}\\
&\|Dg_H\|_{C^0} + \d (H) \|D^2 g_H\|_{C^0} + \d (H)^2 \|D^3 g_H\|_{C^{\kappa}}\leq C \bmo^{\sfrac{1}{2}} \d (H)^{\gamma_0} \label{e:interpo_1} \\
&\|g_{H} - u (z_H, w_H)\|_{C^0}\leq C \,\bmo^{\sfrac14}\, \d (H)^{\sfrac{\gamma_0}{2}} \ell(H) + c_s \d (H)^a \label{e:interpo_3}\\
&|\pi_{H}-T_{(x, g_{H}(x))}\bG_{g_{H}}|\leq C \bmo^{\sfrac12} \d (H)^{\gamma_0-1 + \delta_1}\ell (H)^{1-\delta_1}  \qquad\qquad\forall x\in B_{4r_H} (z_H, w_H)\, .\label{e:interpo_4} 
\end{align}
\item[(iii)] If $L\in \sW\cup \sS$, $L\cap H \neq \emptyset $ and $\ell (H) \leq \ell (L) \leq 2 \ell (H)$, then
\begin{align}
&\|D^lg_{L} - D^l g_{H}\|_{C^0(B_{r_L}(z_L, w_L))} \leq C\, \bmo^{\sfrac12}\,
\d (H)^{2(1+\beta_0)\gamma_0-\beta_2 -2}\,\ell(H)^{3+\kappa-l } \quad \forall\; l=0, \ldots, 3\, .
\label{e:interpo_5}
\end{align}
\item[(iv)] If $L\in \sS\cup \sW$ and $\d (H)\leq \d (L) \leq 2 \d (H)$, then
\begin{gather}
|D^3g_{H}(z_H, w_H)- D^3g_{L}(z_L, w_L)| \leq C \,\bmo^{\sfrac12} \d (H)^{2(1+\beta_0)\gamma_0-\beta_2 -2} \, d((z_H,w_H),(z_L,w_L))^\kappa\, ,\label{e:interpolazione5}
\end{gather}
where $d ( \cdot, \cdot)$ denotes the distance in $\gira$.
\end{itemize}
\end{proposition}

\subsection{Proof of (i) and (ii) in Proposition \ref{p:main_est}} iF $H\in \mathscr{C}^{N_0}$, then the estimates follow from the standard elliptic theory for the system of equations defining the interpolating functions. We start by fixing $H, L, J$ so that $H\in \sS\cup \sW$, $L$ is an ancestor of $H$ (possibly $H$ itself) and $J$ is  the father of $L$. We denote by $B'$ the ball $B_{8r_J} (p_{HJ}, \pi_H)$, by $B$ the ball $B_{8r_L} (p_{HL}, \pi_H)$, by $\bC'$ the cylinder $\bC_{8r_J} (p_J, \pi_H)$ and by $\bC$ the cylinder $\bC_{8r_L} (p_L, \pi_H)$. Observe that $B\subset B'$ (this just requires $M_0$ sufficiently large, given the estimate $|p_J-p_L|\leq 2\sqrt{2}\ell (J)$) and thus $\bC\subset \bC'$. Next, set $E:= \bE (T_L, \bC_{32 r_L} (p_L, \pi_H))$ and $E':= \bE (T_J, \bC_{32r_J} (p_J, \pi_H))$ and recalling Proposition \ref{p:tilting opt}(vii) and \eqref{e:app4} we record
\begin{align}
E \leq &C \bmo \d (L)^{2\gamma_0-2+2\delta_1} \ell (L)^{2-2\delta_1} \leq C \bmo \d (H)^{2\gamma_0-2+2\delta_1} \ell (J)^{2-2\delta_1}\\
E' \leq &C \bmo \d (J)^{2\gamma_0-2+2\delta_1} \ell (J)^{2-2\delta_1}\leq C \bmo \d (H)^{2\gamma_0-2+2\delta_1} \ell (J)^{2-2\delta_1}\\
\bh (T_L, \bC) \leq &C \bmax_0^{\sfrac14}\d (L)^{\sfrac{\gamma_0}{2} - \beta_2}\ell(L)^{1+\beta_2}\leq 
C \bmax_0^{\sfrac14} \d (H)^{\sfrac{\gamma_0}{2} - \beta_2}\ell(J)^{1+\beta_2}\\
\bh (T_J, \bC') \leq &C \bmax_0^{\sfrac14}\d (J)^{\sfrac{\gamma_0}{2} - \beta_2}\ell(J)^{1+\beta_2}
\leq C \bmax_0^{\sfrac14} \d (H)^{\sfrac{\gamma_0}{2} - \beta_2}\ell(J)^{1+\beta_2}\, .
\end{align}
Next let $\bar K$ be the projection of $\gr (f_{HL})\cap \gr (f_{HJ})$ onto $p_{HL} + \pi_H$. Since $T_J \res \bB_L = T_L$, we can estimate
\[
|B\setminus \bar K|\leq \cH^2 (\gr (f_{HL})\setminus \supp (T_L)) + \cH^2 (\gr (f_{HJ})\setminus \supp (T_J))
\]
and recalling the estimates of \cite[Theorem 1.5]{DSS2} we achieve
\[
|B\setminus \bar K| \leq C_0 r_J^2 (E^{\beta_0} (E + C_0 \bmo r_J^2) + E'^{\beta_0} (E' + C_0 \bmo r_J^2))
\leq C \bmo ^{1+\beta_0} \d (H)^{2 (1+\beta_0) \gamma_0 -2} \ell (J)^4\, .
\]
In particular $\bar K$ is certainly nonempty, provided $\eps_2$ is small enough. Using the estimates of \cite[Theorem 1.5]{DSS2} on the oscillation of $f_{HL}$ and $f_{HJ}$ we conclude that
\[
\|\etaa\circ f_{HL} - \etaa\circ f_{HJ}\|_{L^\infty (B)} \leq C \bmo^{\sfrac{1}{4}} \d (H)^{\sfrac{\gamma_0}{2} - \beta_2} \ell (J)^{1+\beta_2}\, .
\]
Set therefore $\zeta := \etaa \circ \bar f_{HL} - \etaa\circ \bar f_{HJ}$ and conclude that
\[
\|\zeta\|_{L^1 (B)}
 \leq \|\etaa\circ f_{HL} - \etaa\circ f_{HJ}\|_{L^\infty (B)} \, |B\setminus \bar K|
\leq C \bmo^{1+\beta_0+\sfrac14} \d (H)^{\sfrac{\gamma_0}{2} -\beta_2+2(1+\beta_0)\gamma_0-2} \ell(J)^{5+\beta_2}\, .
\]
If we define $\xi:= \bar h_{HL}- \bar h_{HJ}$ we can use \eqref{e:stima-L1} of Proposition \ref{p:stime} and the triangular inequality to infer
\[
\|\xi\|_{L^1 (B)} \leq C \bmo \d (H)^{2(1+\beta_0)\gamma_0-2-\beta_2} \ell(J)^{5+\beta_2}\, .
\]
In turn, again by Proposition \ref{p:stime}, this time using the fact that $\mathscr{L}_H \xi =0$ and \eqref{e:higher}, we infer
\begin{align}\label{e:1234h}
\|D^l (\bar{h}_{HL} - \bar{h}_{HJ})\|_{C^0 (\hat B)} &\leq C\bmo \d (H)^{2(1+\beta_0)\gamma_0 -2-\beta_2} \ell (J)^{3+\beta_2 - l} \notag\\
&\leq C\bmo \d (H)^{2(1+\beta_0)\gamma_0 -2-\beta_2} \ell (J)^{3+2\kappa-l} \qquad \mbox{for $l=0,1,2,3,4$,}
\end{align}
where $\hat B = B_{6r_L} (p_{HL}, \pi_H)$.
Interpolating we also get easily
\begin{equation}\label{e:interpolated}
[D^3 (\bar{h}_{HL} - \bar{h}_{HJ})]_{0,\kappa, \hat B} \leq C\bmo \d (H)^{2(1+\beta_0)\gamma_0 -2-\beta_2} \ell (J)^{\kappa}\, .
\end{equation}

In case (b) of Definition \ref{d:semicalibrated} we have $h_{HL}= \bar{h}_{HL}$ and $h_{HJ} = \bar{h}_{HJ}$. In case (a) and (c), using the system of coordinates introduced in the proof of Proposition \ref{p:stime} we have 
\begin{align*}
h_{HL} (x)  &= (\bar h_{HL} (x), \Psi_{p_H} (x, \bar h_{HL} (x)))\\
h_{HJ} (x) &= (\bar h_{HJ} (x), \Psi_{p_H} (x, \bar h_{HJ} (x)))
\end{align*}
and we use the chain rule and the regularity of $\Psi_{p_H}$ to achieve the corresponding estimates
\begin{align}
&\|D^l (h_{HL} - h_{HJ})\|_{C^0 (\hat B)} \leq C\,\bmo \d (H)^{2(1+\beta_0)\gamma_0 -2 - \beta_2} \ell (J)^{3 + 2 \kappa-l} \qquad \mbox{for $l=0,1,2,3$.}\label{e:induttive_1}\\
&[D^3 (h_{HL}- h_{HJ})]_{0,\kappa, \hat B} \leq C \,\bmo \d (H)^{2(1+\beta_0)\gamma_0 -2-\beta_2} \ell (J)^{\kappa}\label{e:induttive_2}\, .
\end{align}
Fix now a chain of cubes $H = H_j \subset H_{j-1} \subset \ldots \subset H_{N} =: L$, where each $H_{i-1}$ is the father of $H_i$.
Summing the estimates above and using the fact that $\ell (H_j) = 2^{-j} \d (H)$ and $\ell(H) \leq \d (H)=\d (H_{N_0}) $, we infer
\begin{align}
&\|D^l (h_{HL} - h_H)\|_{C^0 (\tilde{B})} \leq C \,\d (H)^{2(1+\beta_0)\gamma_0+1 -l}  \qquad \mbox{for $l=0,1,2,3$}\\
&[D^3 (h_{HL} - h_H)]_{0, \kappa, \tilde{B}} \leq C\, \d (H)^{2(1+\beta_0)\gamma_0-\beta_2 + \kappa -2} \, ,
\end{align}
where $\tilde B = B_{6r_H} (p_H, \pi_H)$. 
Observe that, assuming that we have fixed coordinates so that $p_H = (0,0,0)$ we also know, arguing as in the proof of Proposition \ref{p:pde}, that, if we set $\bar{B} := B_{8r_L} (p_{HL}, \pi_H)$, then
\[
\|\etaa\circ \bar{f}_{HL}\|_{L^\infty (\bar{B})} \leq C \bmo^{\sfrac{1}{4}} \d (H)^{1+\sfrac{\gamma_0}{2}}\, .
\]
In particular, applying \eqref{e:stima-Linfty} of Proposition \ref{p:stime}, we conclude
\[
\|\bar h_{HL}\|_{C^0(\hat B)} \leq  C \bmo^{\sfrac{1}{4}} \d (H)^{1+\sfrac{\gamma_0}{2}}\,.
\]
Recalling the height estimate Proposition \ref{p:tilting opt}(vii) (taking into account that $p_H = 0 \in \supp (T)$) we have
\[
\|\etaa \circ \bar f_H\|_{C^0} \leq C \bmo^{\sfrac{1}{4}} \d (H)^{\sfrac{\gamma_0}{2} - \beta_2} \ell (H)^{1+\beta_2}\,
\] 
and thus
\[
\|\etaa\circ \bar f_H\|_{L^1 (B_{8r_H} (p_H, \pi_H))} \leq C \bmo^{\sfrac{1}{4}} \d (H)^{\sfrac{\gamma_0}{2} - \beta_2} \ell (H)^{3+\beta_2}
\]
Using \eqref{e:stima-L1} we conclude 
\[
\|\bar h_H\|_{L^1 (B_{8r_H} (p_H, \pi_H))} \leq C \bmo^{\sfrac{1}{4}} \d (H)^{\sfrac{\gamma_0}{2} - \beta_2} \ell (H)^{3+\beta_2}\, 
\]
and with the help of \eqref{e:L1-Linfty} we achieve
\begin{equation}\label{e:h-Linfty_1}
\|\bar h_H\|_{L^\infty (B_{6r_H} (p_H, \pi_H))} \leq C  \bmo^{\sfrac{1}{4}} \d (H)^{\sfrac{\gamma_0}{2} - \beta_2} \ell (H)^{1+\beta_2}\, .
\end{equation}
Using the estimates upon $\Psi_{p_H}$ and the fact that $\Psi_{p_H} (0) =0$, $D\Psi_{p_H} (0)=0$ we easily conclude 
\begin{equation}\label{e:h-Linfty_2}
\|h_H\|_{L^\infty (B_{6r_H} (p_H, \pi_H))} \leq C  \bmo^{\sfrac{1}{4}} \d (H)^{\sfrac{\gamma_0}{2} - \beta_2} \ell (H)^{1+\beta_2}\, ,
\end{equation}
which in fact is \eqref{e:sfava_tanto}.

We next estimate the derivatives of $h_{HL}$. Let $E:= \bE (T_L, \bC_{32 r_L} (p_L, \pi_H))$
and recall the discussion above and the estimates of \cite[Theorem 1.5]{DSS2} to conclude that
\begin{equation}
\int_{\bar{B}} |Df_{HL}|^2 \leq C_0 r_L^2 E \leq C \bmo \d (H)^{2\gamma_0 -2 + 2 \delta_1} \ell (L)^{4-2\delta_1}\,. 
%\leq \bmo \d (H)^{2 + 2\gamma_0}\, .
\end{equation}
We thus conclude that $\|D\etaa\circ \bar{f}_{HL}\|_{L^2 (\bar{B})} \leq C \bmo^{\sfrac{1}{2}} 
\d (H)^{\gamma_0-1+\delta_1}\,\ell(H)^{2-\delta_1}$. An analogous estimate can be derived for $\|\etaa\circ \bar f_{HL}\|_{L^2 (\bar B)}$ using the $L^\infty$ bound already derived.  
We can now use Lemma \ref{l:exist} to estimate $\|D \bar{h}_{HL}\|_{L^2}\leq C \bmo^{\sfrac{1}{2}} 
\d (H)^{\gamma_0-1+\delta_1}\,\ell(H)^{2-\delta_1}$ (recall the estimate on $\mathcal{F}_H$ derived in the previous section) and thus
$\|D \bar{h}_{HL}\|_{L^1 (\bar B)} \leq C \bmo^{\sfrac{1}{2}} \d (H)^{\gamma_0-1+\delta_1}\,\ell(H)^{3-\delta_1}$.  If we differentiate the equation defining $\bar{h}_{HL}$ we then find
\[
\mathscr{L}_H \partial_j \bar{h}_{HL}^i = ({\mathbf L}_3)^i_{j}
\]
and we can thus apply \eqref{e:L1-Linfty} of Proposition \ref{p:stime}, with $v=D\bar h_{HL}$, to conclude that 
\begin{equation}\label{e:capotribu}
\|D^l \bar{h}_{HL}\|_{L^\infty (B_{6r_L})} \leq C \bmo^{\sfrac{1}{2}} \d (H)^{\gamma_0-1+\delta_1} \ell(L)^{2-\delta_1-l}\leq C \bmo^{\sfrac{1}{2}} \d (H)^{\gamma_0+1-l} \qquad \mbox{for $l =1,2,3,4$,}
\end{equation}
where we used that, for the starting cubes $L=H_{N_0}$, $\d(H)=\d(L)\leq C (M_0) \ell(L)$.

Arguing as above we achieve a similar estimate for $h_{HL}$. We observe however that the condition $D\Psi_{p_H} (0,0) =0$ plays an important role (assuming to have moved the origin so that it coincides with $p_H$). For instance we have
\[
D h_{HL} = (D \bar{h}_{HL}, D_x \Psi_{p_H} (x, \bar{h}_{HL} (x)) + D_y \Psi_{p_H} (x, \bar{h}_{HL} (x)) D \bar{h}_{HL} (x))\, .
\]
Thus we can easily estimate
\begin{equation}\label{e:capotribu2}
|Dh_{HL} (x)|\leq C \bmo^{\sfrac{1}{2}} \d (H)^{\gamma_0} + |D\Psi_{p_H} (x, \bar{h}_{HL} (x))|\, .
\end{equation}
Now, the second summand in \eqref{e:capotribu2} is estimated with $\|D^2 \Psi_{p_H}\| \ell (H) \leq C \bmo^{\sfrac{1}{2}} \d (H)$, precisely because
$D \Psi_{p_H} (0,0)=0$.

\medskip

It follows by \eqref{e:induttive_1}, \eqref{e:induttive_2}, \eqref{e:capotribu} and the triangular inequality that we have the uniform estimates 
\begin{align}\label{e:7.20}
&\|Dh_H\|_{C^0 (B)} + \d (H) \|D^2 h_H\|_{C^0 (B)} + \d (H)^2 \|D^3 h_H\|_{C^\kappa (B)} \leq C \bmo^{\sfrac{1}{2}} \d (H)^{\gamma_0}
\end{align}

Recall now that, by Proposition \ref{p:tilting opt} we have $|\pi_H - \pi_0| \leq C \bmo^{\sfrac{1}{2}} \d (H)^{\gamma_0}$. 
We can therefore apply \cite[Lemma B.1]{DS4} to the rescaling $k_H (x) := \d (H)^{-1} h_H ( \d (H) x)$ and conclude the existence of the interpolating functions $g_H$ and that the estimates \eqref{e:interpo_1}  hold. 
Moreover, combining \eqref{e:sfava_tanto} with \cite[Lemma B.1]{DS4} we also get
\begin{equation}\label{e:C0-100}
\|g_H - \p_{\pi_0}^\perp (p_H)\|_{C^0} \leq C \bmo^{\sfrac{1}{4}} \d (H)^{\sfrac{\gamma_0}{2}} \ell (H)\, .
\end{equation}
On the other hand $\p_{\pi_0} (p_H) = z_H$ and since $p_H\in \supp (T_H)\cup \bV_{u,a}$, we conclude immediately $|\p_{\pi_0}^\perp (p_H) - u (z_H, w_H)|\leq c_s \d (H)^a$. Combining this last estimate with \eqref{e:C0-100} we conclude \eqref{e:interpo_3}.

Finally, recall that, if $E:= \bE (T, \bC_{32 r_H} (p_H, \pi_H))$, then
\[
\int_{B_{8r_H} (p_H, \pi_H)} |Df_H|^2 \leq C \bmo \d (H)^{2\gamma_0-2+2\delta_2} \ell (H)^{4-2\delta_1}\, ,
\] 
from which clearly we get
\[
\int_{B_{8r_H} (p_H, \pi_H)} |D \etaa \circ f_H|^2 \leq C \bmo \d (H)^{2\gamma_0-2+2\delta_1} \ell (H)^{4-2\delta_1}\, .
\]
By the last estimate in Lemma \ref{l:exist} (recalling again the bounds on $\mathcal{F}_H$ and $\etaa\circ f_H$), we deduce
\[
\int_{B_{8r_H} (p_H, \pi_H)} |D \bar{h}_H|^2 \leq C \bmo \d (H)^{2\gamma_0-2+2\delta_1} \ell (H)^{4-2\delta_1}\, .
\]
Thus we conclude the existence of a point $p\in B_{8r_H} (p_H, \pi_H)$ such that 
\begin{equation}\label{e:bound_tangente}
|D\bar h_H (p)|\leq C \bmo^{\sfrac{1}{2}} \d (H)^{\gamma_0-1+\delta_1} \ell (H)^{1-\delta_1}\,.
\end{equation}
Assume now to be in the case (a) or (c) of Definition \ref{d:semicalibrated} and  shift the origin so that it coincides with $p_H$. Given the bound \eqref{e:7.20} on $D^2\bar{h}_H$ we then conclude
\[
|D\bar h_H (0)|\leq C \bmo^{\sfrac{1}{2}} \d (H)^{\gamma_0-1+\delta_1} \ell (H)^{1-\delta_1}
\] 
and, since $D \Psi_{p_H} (0)=0$, we also have $|D h_H (0)|\leq C \bmo^{\sfrac{1}{2}} \d (H)^{\gamma_0-1+\delta_1} \ell (H)^{1-\delta_1}$. Hence using the bound on $\|D^2 h_H\|_0$, we finally conclude $|D\bar{h}_H (q)| \leq C \bmo^{\sfrac{1}{2}} \d (H)^{\gamma_0-1+\delta_1} \ell (H)^{1-\delta_1}$ for all $q$'s in the domain of $\bar{h}_H$.  This implies the estimate
\[
|T_p \bG_{h_H} - \pi_H| \leq C \bmo^{\sfrac{1}{2}} \d (H)^{\gamma_0 -1 + \delta_1} \ell (H)^{1-\delta_1}\, 
\qquad \forall p\in \gr (h_H)\cap \bC_{6r_H} (p_H, \pi_H)\, .
\]

Since however $\gr (g_H) \subset  \gr (h_H)\cap \bC_{6r_H} (p_H, \pi_H)$, we then conclude \eqref{e:interpo_4}. The same conclusion for case (b) in Definition \ref{d:semicalibrated} follows directly from \eqref{e:bound_tangente}. 

\subsection{Proof of (iii) and (iv)} We observe first that (iv) is a rather simple consequence of (iii). Indeed fix $H$ and $L$
as in the statements and consider $H = H_i \subset H_{i-1} \subset \ldots \subset H_{N_0}\in \mathscr{C}^{N_0}$ and $L = L_j \subset L_{j-1}\subset \ldots \subset L_{N_0}\in \mathscr{C}^{N_0}$ so that
$H_l$ is the father of $H_{l+1}$ and $L_l$ is the father of $L_{l+1}$. We distinguish two cases:
\begin{itemize}
\item[(A)] If $H_{N_0}\cap L_{N_0}\neq \emptyset$, we let $i_0$ be the smallest index so that $H_{i_0}\cap L_{i_0}\neq \emptyset$;
\item[(B)] $H_{N_0} \cap L_{N_0} =\emptyset$.
\end{itemize}
In case (A) observe that $\max \{\ell (H_{i_0}), \ell (L_{i_0})\}\leq d ((z_H, w_H), (z_L, w_L)) := d$. On the other hand, recalling that $\d (H_l)= \d (H)$, $\d (L_i)= \d (L)$ and  $\d (L) \leq 2 \d (H)$, by (iii) with $l=3$ we have
\begin{align*}
&|D^3 g_H (z_H, w_H) - D^3 g_{H_{i_0}} (z_{H_{i_0}}, w_{H_{i_0}})|\leq \sum_{l=i_0}^{i-1} |D^3 g_{H_l} (z_{H_l}, w_{H_l}) - D^3 g_{H_{l+1}} (z_{H_{l+1}}, w_{H_{l+1}})|\\
\leq & C \bmo^{\sfrac{1}{2}} \d (H)^{2(1+\beta_0)\gamma_0-\beta_2-2} \ell (H_{i_0})^\kappa \sum_{l=i_0}^{i-1} 2^{(i_0-l)\kappa} \leq C \bmo^{\sfrac{1}{2}} \d (H)^{2(1+\beta_0)\gamma_0-\beta_2 -2} d^\kappa\\
&|D^3 g_L (z_L, w_L) - D^3 g_{L_{i_0}} (z_{L_{i_0}}, w_{L_{i_0}})|\leq \sum_{l=i_0}^{j-1} |D^3 g_{L_l} (z_{L_l}, w_{L_l}) - D^3 g_{L_{l+1}} (z_{L_{l+1}}, w_{L_{l+1}})|\\
\leq &C \bmo^{\sfrac{1}{2}} \d (L)^{2(1+\beta_0)\gamma_0-\beta_2 -2} \ell (L_{i_0})^\kappa \sum_{l=i_0}^{j-1} 2^{(i_0-l)\kappa} \leq C \bmo^{\sfrac{1}{2}} \d (H)^{2(1+\beta_0)\gamma_0-\beta_2 -2} d^\kappa 
\end{align*}
\begin{align*}
&|D^3 g_{L_{i_0}} (z_{L_{i_0}}, w_{L_{i_0}}) - D^3 g_{H_{i_0}} (z_{H_{i_0}}, w_{H_{i_0}})| \leq C \bmo^{\sfrac{1}{2}} \d (H_{i_0})^{2(1+\beta_0)\gamma_0-\beta_2  -2} \ell (H_{i_0})^{\kappa}\\
\leq &\bmo^{\sfrac{1}{2}} \d (H)^{2(1+\beta_0)\gamma_0-\beta_2  -2} d^\kappa\, .
\end{align*}
The triangle inequality implies then the desired estimate.

In case (B) we first notice that by the very same argument we have the estimates
\begin{align*}
|D^3 g_H (z_H, w_H) - D^3 g_{H_{N_0}} (z_{H_{N_0}}, w_{H_{N_0}})|&\leq C \bmo^{\sfrac{1}{2}} \d (H)^{2(1+\beta_0)\gamma_0-\beta_2 -2} d^\kappa\\
|D^3 g_L (z_L, w_L) - D^3 g_{L_{N_0}} (z_{L_{N_0}}, w_{L_{N_0}})|&\leq  C \bmo^{\sfrac{1}{2}} \d (H)^{2(1+\beta_0)\gamma_0-\beta_2 -2} d^\kappa\, .
\end{align*}
Next we find a chain of cubes $H_{N_0} =J_0, J_1, \ldots, J_N = L_{N_0}$, all distinct and belonging to $\sS^{N_0}$, so that 
\begin{itemize}
\item $\d (H) \leq \d (J_l) \leq \d (L) \leq 2 \d (H)$;
\item $J_l\cap J_{l+1}\neq \emptyset$ and thus $\ell (H_{N_0}) \leq \ell (J_l)\leq \ell (L_{N_0})$;
\item $N$ is smaller than a constant $C (N_0, \bar{Q})$. 
\end{itemize}
Using again (iii) and arguing as above we conclude
\begin{align*}
&| D^3 g_{H_{N_0}} (z_{H_{N_0}}, w_{H_{N_0}}) - D^3 g_{L_{N_0}} (z_{L_{N_0}}, w_{L_{N_0}})|\\
\leq &\sum_{l=1}^{N} |D^3 g_{J_l} (z_{J_l}, w_{J_l}) - D^3 g_{J_{l-1}} (z_{J_{l-1}}, w_{J_{l-1}})|
\leq C N \bmo^{\sfrac{1}{2}} \d (H)^{2(1+\beta_0)\gamma_0-\beta_2 -2} d^\kappa\, .
\end{align*}
Again, using the triangular inequality we conclude (iv).

\medskip

We now come to (iii). Fix therefore two cubes $H$ and $L$ as in the statement and set $r:= r_L$. Observe that, by (ii) and \cite[Lemma C.2]{DS4}, it suffices to show that $\|g_H- g_L\|_{L^1 (B)} \leq C \bmo^{\sfrac{1}{2}} \d (H)^{\sfrac{\gamma_0}{2} -2} \ell (H)^{5+\kappa}$. where $B=B_r (z_L, \pi_0)$. Consider now the two corresponding tilted interpolating functions, namely
$h_L$ and $h_H$. Given the estimate on the Lipschitz constant and the $C^0$ norm upon $h_L$ proved in the previous paragraph, we can find a function $\hat{h}_L: B_{7 r} (p_{HL}, \pi_H) \to \pi_H^\perp$ such that
$\bG_{\hat{h}_L} = \bG_{h_L} \res \bC_{6r} (p_L, \pi_H)$ applying \cite[Lemma B.1]{DS4} (in this paragraph $\hat\cdot$ will always denote the reparametrization on $\pi_H$). By (i) clearly $\bG_{\hat{h}_L} \res \bC_{r} (z_L, \pi_0)= \bG_{g_L}$. We can therefore apply \cite[Lemma B.1]{DS4} to conclude that
\[
\|g_H - g_L\|_{L^1 (B)} \leq C \|h_H - \hat{h}_L\|_{L^1 (B_{5r} (p_L, \pi_H))}\, .
\]
Consider next the tilted interpolating function $h_{HL}$ and observe that, by \eqref{e:1234h} and the usual estimates on $\Psi$, we know 
\[
\|h_H - h_{HL}\|_{L^1 (B_{5r} (p_H, \pi_H))} \leq C \bmo^{\sfrac{1}{2}} \d (H)^{2(1+\beta_0)\gamma_0-\beta_2 -2} \ell (H)^{5+\beta_2}\, .
\]
Indeed, by \eqref{e:stima-L1} and the usual estimates on $\Psi$,
it is enough to see that
\begin{align*}
\|\etaa\circ f_{HH} - \etaa\circ f_{HL}\|_{L^1 (B_{5r} (p_H, \pi_H))} 
& \leq \|\etaa\circ f_{HH} - \etaa\circ f_{HL}\|_{L^\infty (B_{5r} (p_H, \pi_H))} \, |B_{5r} (p_H, \pi_H)\setminus \bar K|\\
&\leq 
C \bmo^{1+\beta_0+\sfrac14} \d (H)^{\sfrac{\gamma_0}{2} -\beta_2+2(1+\beta_0)\gamma_0-2} \ell(H)^{5+\beta_2}\, ,
\end{align*}
where, as above, we used the estimates of \cite[Theorem 1.5]{DSS2} on the oscillation of $f_{HH}$ and $f_{HL}$ (with $\bar K$ the projection of $\gr (f_{HH})\cap \gr (f_{HL})$ onto $p_{HL} + \pi_H$) and $B_{5r} (p_H, \pi_H) \subset  B_{8r} (p_{HL}, \pi_H)$.
Hence, since $\beta_2 \geq \kappa$, we are reduced to show
\begin{equation}\label{e:goal}
\|h_{HL} - \hat{h}_L\|_{L^1 (B_{5r} (p_H, \pi_H))}\leq C \bmo^{\sfrac{1}{2}} \d (H)^{2(1+\beta_0)\gamma_0-\beta_2 -2} \ell (H)^{5+\kappa}\, .
\end{equation}
In turn, consider the $\pi_H$-approximating function $f_{HL}$ and the $\pi_L$-approximating function $f_{LL} = f_L$. 
In the $\pi_H \times \varkappa_H\times T_{p_H} \Sigma^\perp$ coordinates we set
\[
\bef_{HL} (x) = (\p_{\varkappa_H} (\etaa\circ f_{HL} (x)), \Psi_{p_H} (x,\p_{\varkappa_H} (\etaa\circ f_{HL} (x))))
\]
and recall that, by Proposition \ref{p:stime}, we have
\begin{equation}\label{e:pezzo_1}
\|h_{HL}- \bef_{HL}\|_{L^1 (B_{8r_L} (p_{HL}, \pi_H))} \leq C  \bmo^{\sfrac{1}{2}} \d (L)^{2(1+\beta_0)\gamma_0-\beta_2 -2} \ell (L)^{5+\beta_2}\, .
\end{equation}
Similarly, in the $\pi_L \times \varkappa_L\times T_{p_L} \Sigma^\perp$ coordinates we set
\[
\bef_{L} (x) = (\p_{\varkappa_L} (\etaa\circ f_{L} (x)), \Psi_{p_L} (x,\p_{\varkappa_L} (\etaa\circ f_{L} (x))))
\]
and get
\[
\|h_{L}- \bef_{L}\|_{L^1 (B_{8r_L} (p_L, \pi_L))} \leq C  \bmo^{\sfrac{1}{2}} \d (L)^{2(1+\beta_0)\gamma_0-\beta_2 -2} \ell (L)^{5+\beta_2}\, .
\]
Next we denote by $\hat{\bef}_L$ the map $\hat{\bef}_L : B_{6r_L} (p_{HL}, \pi_H)\to \pi_H^\perp$ such that $\bG_{\hat{\bef}_L} = \bG_{\bef_L} \res \bC_{6r_L} (p_L, \pi_H)$
and we use again \cite[Lemma B.1]{DS4} to infer
\begin{equation}\label{e:pezzo_2}
\|\hat h_{L}- \hat \bef_{L}\|_{L^1 (B_{6r_L} (p_{HL}, \pi_H))} \leq C \|h_{L}- \bef_{L}\|_{L^1 (B_{8r_L} (p_L, \pi_L)}
\leq C  \bmo^{\sfrac{1}{2}} \d (L)^{2(1+\beta_0)\gamma_0-\beta_2 -2} \ell (L)^{5+\beta_2}\, .
\end{equation}
In view of \eqref{e:pezzo_1} and \eqref{e:pezzo_2}, \eqref{e:goal} is then reduced to
\begin{equation}\label{e:goal_2}
\|\bef_{HL} - \hat{\bef}_L\|_{L^1 (B_{5r_L} (p_{HL}, \pi_H))}\leq C \bmo^{\sfrac{1}{2}} \d (H)^{2(1+\beta_0)\gamma_0-\beta_2 -2} \ell (H)^{5+\kappa}\, .
\end{equation}
Consider now the map $\hat{f}_L: \bB_{6r_L} (p_{HL}, \pi_H)\to \Iq (\pi_H^\perp)$ such that $\bG_{\hat{f}_L} = \bG_{f_L}\res \bC_{6r_L} (p_L, \pi_H)$. Let $A$ and $\hat{A}$ be the projections on $p_H+\pi_H = p_{HL} +\pi_H$ of the Borel sets $\gr (f_{HL}))\setminus \supp (T)$ and $\gr (\hat f_{L})\setminus \supp (T) \subset \gr (f_L)\setminus \supp (T)$. We know that
\begin{align*}
|A\cup A'| \leq &\|\bG_{f_{HL}} - T\| (\bC_{8r_L} (p_L, \pi_H)) + \|\bG_{f_L} - T \| (\bC_{8r_L} (p_L, \pi_L))\\
 \leq &C \bmo^{1+\beta_0} \d (H)^{2(1+\beta_0) \gamma_0-2} \ell (H)^{4}\, .
\end{align*}
On the other hand, thanks to the height bound Proposition \ref{p:tilting opt}(vii) and the estimates of 
\cite[Theorem 1.5]{DSS2}, that $\|\etaa\circ f_{HL} - \etaa \circ {\hat{f}_L}\|_{L^\infty (B_{7r_L} (p_{HL}, \pi_H))} \leq C \bmo^{\sfrac{1}{4}} \d (H)^{\sfrac{\gamma_0}{2} - \beta_2} \ell (H)^{1+\beta_2}$. We thus conclude that
\[
\|\etaa\circ f_{HL} - \etaa \circ{\hat{f}_L}\|_{L^1 (B_{6r_L} (p_{HL}, \pi_H))} \leq C \bmo^{\sfrac{1}{2}} \d (H)^{2(1+\beta_0)\gamma_0-\beta_2 -2} \ell (H)^{5+\beta_2}\, .
\]
Define in the $\pi_H\times \varkappa_H \times T_{p_H} \Sigma^\perp$ coordinates the function
\[
\beg (x) := (\p_{\varkappa_H} (\etaa \circ \hat{f}_L (x)), \Psi_{p_H} (x, \p_{\varkappa_H} (\etaa \circ \hat{f}_L (x))))\, .
\]
We can thus conclude that
\begin{equation}\label{e:pezzo_3}
\|\bef_{HL}- \beg\|_{L^1 (B_{6r_L} (p_{HL}, \pi_L))}
\leq C  \bmo^{\sfrac{1}{2}} \d (L)^{2(1+\beta_0)\gamma_0-\beta_2 -2} \ell (L)^{5+\beta_2}\, .
\end{equation}
Thus, \eqref{e:goal_2} is now reduced to
\begin{equation}\label{e:goal_3}
\|\beg - \hat{\bef}_L\|_{L^1 (B_{5r_L} (p_{HL}, \pi_H))}\leq C \bmo^{\sfrac{1}{2}} \d (H)^{2(1+\beta_0)\gamma_0-\beta_2  -2} \ell (H)^{5+\kappa}\, .
\end{equation}
Denoting by ${\rm An}$ the distance $|\pi_H-\pi_L|$, by $\hat{B}$ the ball $B_{6r_L} (p_{HL}, \pi_H)$ and by $\tilde{B}$ the ball $B_{8r_L} (p_L, \pi_L)$, we then have, by \cite[Lemma 5.6]{DS4}
\[
\|\beg - \hat{\bef}_L\|_{L^1 (\hat B)}\leq C_0 ({\rm osc}\, (f_L) + r_L {\rm An}) \left(\int |Df_L|^2 + r_L^2 (\|D\Psi_{p_L}\|^2_{C^0 (\tilde{B})} + {\rm An}^2)\right)\, .
\]
Recall that $D\Psi_{p_L} (p_L) =0$ and thus $\|D\Psi_{p_L}\|^2_{C^0 (\tilde{B})}\leq C_0 \bmo r_L^2$. Recalling the estimate of Proposition \ref{p:tilting opt} on $|\pi_H-\pi_L|$ and upon the Dirichlet energy of $f_L$ (namely \eqref{e:Dir f}), we then conclude
\[
\int |Df_L|^2 + r_L^2 (\|D\Psi_{p_L}\|^2_{C^0 (\tilde{B})} + {\rm An}^2)\leq C \bmo \d (L)^{2\gamma_0 -2 + 2\delta_1} \ell (H)^{4-2\delta_1}\, .
\]
On the other hand
\[
{\rm osc}\, (f_L) + r_L {\rm An}\leq C \bmo^{\sfrac{1}{4}} \d(H)^{\sfrac{\gamma_0}{2} - \beta_2} \ell (H)^{1+\beta_2}\, .
\]
Thus \eqref{e:goal_3} follows by our choice of the various parameters, in particular $\beta_2-2\delta_1\geq \sfrac{\beta_2}{4}=\kappa$.

\section{Proof of Theorem \ref{t:cm}}

\subsection{Proof of (i)} As in all the proofs so far, we will use $C_0$ for geometric constants and $C$ for constants which depend upon $M_0, N_0, C_e$ and $C_h$. Define $\chi_H := \vartheta_H/ (\sum_{L\in \sP^j} \vartheta_L)$ for each $H\in \sP^j$(cf. Definition \ref{d:glued})
and observe that 
\begin{align}
\sum_{H\in\sP^j} \chi_H = 1 \;\; \mbox{on $\mathcal{A}_k$ $\forall k\in\N$}\qquad \text{and}\qquad
\|\chi_H\|_{C^i} &\leq C_0 \,\ell (H)^{-i} \quad \forall i\in \{0,1,2,3,4\}\label{e:p_unita}\, .
\end{align}
Fix any $H\in \sP^j$ and let $k$ be such that $H\subset \mathcal{A}_k$.
Set $\sP^j (H):=\{L\in \sP^j : L\cap H\neq\emptyset\}\setminus \{H\}$ for each $H\in \sP^j$. 
By construction $\frac{1}{2} \ell (L) \leq \ell (H) \leq 2\, \ell (L)$ and $2^{-k-1} \leq \d (L) \leq 2^{-k+1}$ for every $L\in \sP^j (H)$. Moreover the cardinality of $\sP^j(H)$
is at most $12$. 
Fix a point $p= (z,w)\in H$ and observe that $C_0^{-1} 2^{-k} \leq |z| \leq C_0 2^{-k}$. From \eqref{e:interpo_0} of Proposition \ref{p:main_est} we then conclude
\[
|\hat\varphi_j (z,w)|\leq   C \bmo^{\sfrac{1}{4}} \d (H)^{1+\sfrac{\gamma_0}{2}}\leq C \bmo^{\sfrac{1}{4}} |z|^{1+\sfrac{\gamma_0}{2}}\, .
\]
Recall now that $\Psi (0) =0$, $D\Psi (0) = 0$ and $\|D^2\Psi\|_{C^0} \leq C \bmo^{\sfrac{1}{2}}$. Considering that
\begin{equation}\label{e:formula_per_phi}
\varphi_j (z,w) = (\bar\varphi_j (z,w), \Psi (z, \bar\varphi_j (z,w)))
\end{equation}
(where $\bar\varphi_j (z,w)$ is the vector consisting of the first $\bar{n}$ components of $\hat\varphi_j (z,w)$), we easily conclude
\[
|\varphi_j (z,w)| \leq C \bmo^{\sfrac{1}{4}} |z|^{1+\sfrac{\gamma_0}{2}} + C \|D^2\Psi\|_{C^0} |z|^2 \leq C \bmo^{\sfrac{1}{4}} |z|^{1+\sfrac{\gamma_0}{2}}\, .
\]
This gives \eqref{e:stima_C0} and the continuity of $\varphi_j$, since by definition $\varphi_j (0,0)=0$.

For $(z,w)\in H$ we next write
\begin{align}
\hat\varphi_j (z,w) &=\Big(g_H \chi_H  + \sum_{L\in \sP^j (H)} 
g_L \chi_L\Big) (z,w) = g_H (x) + \sum_{L\in \sP^j (H)} (g_L - g_H) \chi_L\,  (z,w)\, ,
\end{align}
because $H$ does not meet the support of $\vartheta_L$ for any $L\in \sP^j$ which does not meet $H$.
Using the Leibniz rule, \eqref{e:p_unita} and the estimates of 
Proposition~ \ref{p:main_est}, for $l\in \{1,2,3\}$ we get
\begin{align*}
&\|D^l \hat\varphi_j - D^l g_H\|_{C^0 (H)} + C_0 \sum_{0\leq i \leq l} 
\sum_{L\in \sP^j (H)} \|g_L-g_H\|_{C^i (H)} \ell (L)^{i-l}\\ 
\leq & C \bmo^{\sfrac{1}{2}} \d (H)^{\gamma_0+1-l} + C \bmo^{\sfrac{1}{2}} \d (H)^{2(1+\beta_0)\gamma_0-\beta_2-2} \sum_{0\leq i\leq l} \ell (H)^{3+\kappa-i} \ell (H)^{i-l}\\
\leq&  C \bmo^{\sfrac{1}{2}} \d (H)^{\gamma_0+1-l}\, . 
\end{align*}
Again using the formula \eqref{e:formula_per_phi} and the estimate $\|\Psi\|_{C^{3, \eps_0}}\leq \bmo^{\sfrac{1}{2}}$ (together with $D\Psi (0)=0$ and $\Psi (0)=0$) we easily reach
\eqref{e:stima_C3}.  In fact we can also see that
\begin{equation}\label{e:aggiuntiva}
\|D^l \varphi_j - D^l g_H\|_{C^0 (H)} \leq C \bmo^{\sfrac{1}{2}} \d (H)^{2(1+\beta_0)\gamma_0-\beta_2-2} \ell (H)^{3+\kappa-l}\, .
\end{equation}
With an argument entirely similar we obtain
\begin{equation}\label{e:Holder_locale}
[D^3 \varphi_j]_{\kappa, H} \leq C \bmo^{\sfrac{1}{2}} \d (H)^{\gamma_0-2}\, .
\end{equation}
Thus, pick any two points $(z,w), (z'w')\in \mathcal{A}_k$. If they belong to the same cube $H\in \sP^j$ with $H\subset \mathcal{A}_k$, then
\begin{align}
|D^3 \varphi_j (z,w)- D^3 \varphi_j (z',w')| \leq & C \bmo^{\sfrac{1}{2}} \d (H)^{-2} d ((z',w'), (z,w))^\kappa\nonumber\\
\leq & C \bmo^{\sfrac{1}{2}} 2^{2k} d ((z',w'), (z,w))^\kappa\, .\label{e:Holder1}
\end{align}
If they do not belong to the same cube, then let $H, L\in \sP^j$ be two cubes contained in $\mathcal{A}_k$ such that
$(z,w)\in H$ and $(z',w')\in L$.
Next observe that, by our choice of the cut-off functions $\vartheta_J$, $\varphi_j = g_H$ in a neighborhood of $(z_H, w_H)$ and $\varphi_j = g_L$ in a neighborhood of $(z_L, w_L)$.   We can then estimate, using Proposition \ref{p:main_est}(iv) and \eqref{e:Holder_locale}
\begin{align}
&|D^3 \varphi_j (z,w)-D^3\varphi_j (z',w')|\leq |D^3 \varphi_j (z,w) - D^3 g_H (z_H, w_H)|\nonumber \\
&\qquad\qquad + |D^3 g_H (z_H, w_H) - D^3 g_L (z_L, w_L)| + |D^3 \varphi_j (z_L, w_L) - D^3 \varphi_j (z',w')|\nonumber\\
\leq& C \bmo^{\sfrac{1}{2}} \d (H)^{-2} \left(\ell (H)^\kappa + d ((z_H, w_H), (z_L, w_L))^\kappa + \ell (L)^\kappa)\right)\nonumber\\ 
\leq &C \bmo^{\sfrac{1}{2}} \d (H)^{-2} d ((z,w), (z',w'))^\kappa\leq C \bmo^{\sfrac{1}{2}} 2^{2k} d ((z',w'), (z,w))^\kappa\, .\label{e:Holder2}
\end{align}
From \eqref{e:Holder1} and \eqref{e:Holder2} we conclude \eqref{e:stima_Hoelder} and thus the proof of Theorem \ref{t:cm}(i). 

\subsection{Proof of (ii)} The first statement is an obvious consequence of the construction algorithm: indeed note that, if $i,j,k$, $L$ and $H$ are as in the statement
then $\sP^j (L) = \sP^k (L)$ and moreover $\chi_J = 0$ on $H$ for any $J\in \sP^j\setminus \sP^j (L)$ and for any $J\in \sP^k\setminus \sP^k (L)$. Then it turns out that 
$\hat\varphi_j = \hat\varphi_k$ on $H$, which in turn obviously implies that $\varphi_j$ and $\varphi_k$ coincide on $H$.

As for the second statement (ii), if we can show that there is a uniform limit $\phii$ for $\varphi_j$, the $C^3$ convergence and the regularity of $\phii$ will follow from the estimates
of point (i). Fix a point $(z,w)\neq 0$ and let $H\in \sP^j$ which contains it. If $H\in \sW^i$ and $i\leq j-2$, then $\hat\varphi_{j+1}$ and $\hat\varphi_j$ coincide on it. Otherwise we can assume that $H\in \sC^{j-1}\cup \sC^j$. In this case we can estimate
\[
|\varphi_j (z,w) - \varphi_j (z_H, w_H)|\leq C \bmo^{\sfrac{1}{2}} \d (H)^\kappa \ell (H) \leq C 2^{-j}\, .
\]
A similar estimate holds for $\varphi_{j+1}$: notice that we can choose $L\in \sP^{j+1}$ such that $(z,w)\in L$ and $L$ is either $H$ or a son of $H$. Moreover, we can estimate
\[
|\varphi_{j+1} (z,w)- \varphi_{j+1} (z_L, w_L)|\leq C 2^{-j}\, .
\]
Next, recall that $\varphi_{j} (z_H, w_H) = g_H (z_H, w_H)$ and that  $\varphi_{j+1} (z_L, w_L) = g_L (z_L, w_L)$. Since moreover $L= H$ or $L$ is a son of $H$, by Proposition \ref{p:main_est} we achieve
\[
|\varphi_{j+1} (z_L, w_L) - \varphi_j (z_H, w_H)| \leq C_0 \|D g_H\|_{C^0} \ell (H) + C \|g_H-g_L\|_{C^0} \leq C 2^{-j}\, .
\] 
Summarizing, we conclude that
\[
\|\varphi_{j+1}-\varphi_j\|_{C^0} \leq C 2^{-j}\, .
\]
The latter estimate gives that $\varphi_j$ is a Cauchy sequence in $C^0$ and thus that it converges uniformly to some $\phii$. 

We record in particular an important consequence which will be useful later: If $L\in \sW$, then
\begin{equation}\label{e:aggiuntiva_10}
\|\phii - g_L\|_{C^j (L)} \leq C \bmo^{\sfrac{1}{2}} d(L)^{2(1+\beta_0) \gamma_0 - \beta_2} \ell (L)^{3+\kappa}\, .
\end{equation}
Indeed we have already shown the estimate \eqref{e:aggiuntiva} (which corresponds to \eqref{e:aggiuntiva_10} with $\varphi_j$ in place of $\phii$) whenever $j\geq k$ and $L\in \sW^k$, but on the other hand we know now that $\phii = \phi_j$ on $L$ for $j$ large enough.
Finally, recall that the graph of $g_L$ coincides with the graph of $h_L$ and that the graph of $\phii$ coincide with that of $g_L$ in an open subset of $L$. By the estimates on $h_L$, this implies that on such set the distance between any tangent $\pi$ to the graph of $\phii$ and $\pi_L$ is bounded by $\bmo^{\sfrac{1}{2}}$. Given the bounds on the second derivative of $\phii$ we easily conclude that
\begin{equation}\label{e:aggiuntiva_11}
|\pi - \pi_L| \leq C \bmo^{\sfrac{1}{2}} \d (L)^{\sfrac{\gamma_0}{2} - 1} \ell(L) \qquad \mbox{$\forall L\in \sW$ and any tangent $\pi$ to $\gr (\phii|_{B_{64 \ell (L)} (p_L)})$.}
\end{equation}

\subsection{Proof of (iii)} Observe first that, if $(z,w)$ does not belong to some $H\in \sW$, then $\Phii (z,w)$ is necessarily a point in the support of
$T$ and we can estimate
\begin{equation}\label{e:nell'intorno}
|\phii (z,w) - u (z,w)|\leq c_s |z|^a\, .
\end{equation}
To see this note that for every $j\geq N_0$ there is $H_j\in \sS^j$ such that $(z,w)\in H_j$.
Observe that $\varphi_j (z_{H_j}, w_{H_j}) = g_{H_j} (z_{H_j}, w_{H_j})$ and that 
\[
\lim_{j\to \infty} \big(d ((z_{H_j}, w_{H_j}), (z,w)) + | g_{H_j} (z_{H_j}, w_{H_j}) - \phii (z,w)|\big) = 0\, .
\]
But we also have, by \eqref{e:sfava_tanto},
\[
\lim_{j\to\infty} |(z_{H_j}, g_{H_j} (z_{H_j}, w_{H_j})) - p_{H_j}| = 0\, .
\]
On the other hand, since
\[
|p_{H_j} - (z_{H_j}, u (z_{H_j}, w_{H_j}))|\leq c_s |z_{H_j}|^a\, ,
\]
we then conclude \eqref{e:nell'intorno} taking the limit in $j\to \infty$.

From now on we therefore assume that $(z,w)\in H$ for some $H\in \sW$.

\medskip

{\bf Step 1.} In this step we show that 
\begin{equation}\label{e:stop}
64 r_H \leq \frac{1}{2} \d (H)^{(b+1)/2}\, .
\end{equation}
 In fact we claim that this is the case
for any $H\in \sW$. First of all we observe that it suffices to show it for $H\in \sW_e\cup \sW_h$: given indeed
any $H\in \sW_n$ we find a chain of cubes $H= H_l, H_{l-1}, \ldots, H_i$ with the properties that 
\begin{itemize}
\item $H_k \cap H_{k+1} \neq \emptyset$;
\item $\ell (H_{k})= 2\ell (H_{k+1})$;
\item $H_l\in \sW_n$ for any $l\geq i+1$ and $H_i\in \sW_e\cup \sW_h$.
\end{itemize}
It is easy to see that, provided $N_0$ is larger than a geometric constant, $\frac{1}{2} \d (H)\leq \d (H_i) \leq 2 \d (H)$. 
Since $\ell (H) \leq \frac{1}{2} \ell (H_i)$, it suffices to show $64 r_H \leq \frac{1}{16} \d (H_i)^{(b+1)/2}$.

Next, assume $H\in \sW_e$. Then we know that
\begin{equation}\label{e:eccesso_da_sotto}
\bE (T_H, \bB_H)> C_e \bmo \d (H)^{2\gamma_0-2 + 2\delta_1} \ell (H)^{2-2\delta_1} \geq C_e \bmo^{\sfrac{1}{2}} \d (H)^{2\gamma_0-2} \ell (H)^2\, .
\end{equation}
Now recall that $d=|z_H|\leq 2\sqrt{2} \d (H)$ and that, for $N_0$ large enough depending on $M_0$, we have $64 r_H \leq \frac{d}{2}$.  Moreover, if $r_H$ were larger than $\frac{1}{16} d^{(b+1)/2}$, then by \eqref{e:effetto_energia} there would be
a $\pi$ such that (recall that $C_i^2\leq \bmo$)
\[
\bE (T_H, \bB_H, \pi) \leq \bmo \d (H)^{2\gamma-2} r_H^2 \, .
\]
By Lemma \ref{l:tecnico}(i), we then would have
\begin{equation}\label{e:eccesso_da_sopra}
\bE (T_H, \bB_H )\leq C_0 \bmo \d (H)^{2\gamma- 2} r_H^2 + C_0 \bmo r_H^2\leq C (M_0)\bmo \d (H)^{2\gamma_0-2} \ell (H)^2\, 
\end{equation}
(recall that $\gamma_0 < \gamma$). Thus we conclude that \eqref{e:eccesso_da_sopra} contradicts \eqref{e:eccesso_da_sotto}, provided $C_e \geq C(M_0)$. 

It remains to show \eqref{e:stop} when $H\in \sW_h$. Assume therefore that $r_H \geq \frac{1}{2} d^{(b+1)/2}$. Notice that, by \eqref{e:effetto_energia}, we know
\begin{align}
\bE (T_H, \bB_H, \pi_H)= &\bE (T_H, \bB_H) \leq \bar C \bmo \d (H)^{2\gamma-2} \ell (H)^2
\end{align}
where the constant $\bar{C}$ does not depend on $H$. We thus conclude from Lemma \ref{l:tecnico}(iii) that
\begin{equation}\label{e:tilt}
|\pi - \pi_H|\leq \bar C \bmo^{\sfrac{1}{2}} \d (H)^{\gamma-1} \ell (H)\, .
\end{equation}
We next wish to estimate $\bh (T_H, \bB_H, \pi)$. $\pi$ is tangent to $\bG_u$ at $q_H := (z_H, u(z_H,w_H))$. Recall that $|p_H-q_H|\leq c_s |d|^a$. Fix a point $p\in \bB_H \cap \supp (T_H)$ and recall that there is a point $p'$ in $\gr (u)\cap \bV_H$ such that
$|p-p'|\leq 2^a \bmo^{\sfrac{1}{2}} d^a$, since $|\p_{\pi_0} (p')|\geq \frac{d}{2}$. Obviously $|\p_{\pi} (p')|\leq 2 r_H$ and since $\pi$ is tangent to $\gr (u)$ at $q_H$, we have the estimate
\[
|\p_\pi^\perp (p')|\leq C_0 \bmo^{\sfrac{1}{2}} d^{\alpha-1} |\p_\pi (p')|^2 \leq \bar{C} \bmo^{\sfrac{1}{2}} \d (H)^{\alpha-1} \ell (H)^2\, .
\]
We can therefore estimate
\[
|\p_\pi^\perp (p)| \leq \bar{C} \bmo^{\sfrac{1}{2}} \d(H)^{\alpha-1} \ell (H)^2 + \bar{C} \bmo^{\sfrac{1}{2}} \d (H)^a\, .
\]
This implies the estimate
\begin{equation}\label{e:height_stop}
\bh (T_H, \bB_H, \pi) \leq \bar{C} \bmo^{\sfrac{1}{2}} \d(H)^{\alpha-1} \ell (H)^2 + \bar{C} \bmo^{\sfrac{1}{2}} \d (H)^a\, .
\end{equation}
Using now Lemma \ref{l:tecnico} and \eqref{e:tilt} we then estimate
\begin{equation}\label{e:height_stop2}
\bh (T_H, \bB_H) \leq \bar{C} \bmo^{\sfrac{1}{2}} \d(H)^{\alpha-1} \ell (H)^2 + \bar{C} \bmo^{\sfrac{1}{2}} \d (H)^a + \bar C\bmo^{\sfrac{1}{2}} \d (H)^{\gamma-1} \ell (H)^2 \, ,
\end{equation}
where $\bar{C}$ depends upon $M_0, N_0$ and $C_e$, but not upon $C_h$.

On the other hand, since $H\in \sW_h$, we then have
\begin{equation}\label{e:height_stop3}
\bh (T_H, \bB_H)> C_h \bmo^{\sfrac{1}{4}} \d (H)^{\gamma_0-\beta_2}\ell (H)^{1+\beta_2}\, .
\end{equation}
By our choice of the exponents it is obvious that the first and third summand in \eqref{e:height_stop2} are smaller than a fraction (say $\frac{1}{4}$) of $C_h \bmo^{\sfrac{1}{4}} \d (H)^{\gamma_0-\beta_2}\ell (H)^{1+\beta_2}$, provided that $C_h$ is chosen large enough. Recalling that we are assuming $\ell (H)\geq \bar{C} \d (H)^{(1+b)/2}$,
to achieve the same conclusion with the second summand we need
\[
\frac{1+b}{2} (1+\beta_2) - \beta_2 + \gamma_0 < a\, .
\]
However, since $a> b$, the latter inequality is implied by \eqref{e:cost_0}, and we reach a contradiction.

\medskip

{\bf Step 2}. Recall that we have fixed $(z,w)\in H$ with $H\in \sW$ and that our aim is to establish \eqref{e:cornuto!}. From the previous step we know
that $\ell (H) \leq C |z|^{(1+b)/2}$ and that $\d (H) \leq C_0 |z|$. Assume $H\in \sW^j$ and pick any $k\geq j+2$. By Theorem \ref{t:cm}(ii), we know that $\phii= \varphi_k$ on $H$.
Recalling the arguments above (in particular \eqref{e:interpo_5}), we also have
\[
\|\varphi_j - g_H\|_{C^0} \leq \sum_{L\in \sP^k (H)} \|g_H-g_L\|_{C^0} \leq C \bmo^{\sfrac{1}{2}} \d (H)^{\gamma_0-2-\beta_2} \ell (H)^{3+\kappa} 
\leq C \bmo^{\sfrac{1}{2}} d^{\gamma_0-2 + (3+\kappa) (b+1)/2}
\]
Since $\gamma_0-2 + (3+\kappa) (b+1)/2 > \gamma_0 +3\frac{b}{2} - \frac{1}{2} > \gamma_0+ b$, it suffices then to show that
\begin{equation}\label{e:cornuto2}
|u (z,w)- g_H (z,w)|\leq C \bmo^{\sfrac{1}{4}} |z|^{a'}\, .
\end{equation}
We next consider both $u$ and $g_H$ as two functions defined on $\pi_0$ and having defined the ball $B := B_{r_H} (z_H, \pi_0)$, our goal is indeed to show that
\[
\|u - g_H\|_{C^0 (B)} \leq  C \bmo^{\sfrac{1}{4}} \d (H)^{a'}\, .
\]
Recall next that the graph of $g_H$ is indeed a subset of the graph of the tilted interpolating function $h_H$. If $v:B_{8r_H} (p_H, \pi_H)\to \pi_H^\perp$ is the function which gives 
the graph of $u$ in the system of coordinates $\pi_H\times \pi_H^\perp$ and we set $B' := B_{6r_H} (p_H, \pi_H)$, we then claim that it suffices to show
\begin{equation}\label{e:cornuto3}
\|v-h_H\|_{C^0 (B')} \leq C \bmo^{\sfrac{1}{4}} \d (H)^{a'}\, .
\end{equation}
In fact let $p = (\zeta, g_H (\zeta))\in \pi_0\times \pi_0^\perp$ and let $\omega\in \pi_H$ be such that $p= (\omega, h_H (\omega))\in \pi_H\times \pi_H^\perp$.
Consider also $q= (\zeta, u (z))$ and $q'= (\omega, v (\omega))$ and let $\zeta'\in \pi_0$ such that $q'= (\zeta', u (\zeta'))$. Let $\mathcal{T}$ be the triangle with vertices
$q$, $p$ and $q'$. The angle $\theta_p$ at $p$ can be assumed to be small, because $|\pi_H-\pi_0|\leq C \bmo^{\sfrac{1}{2}}$. On the other hand the angle
$\theta_q$ at $q$ is close to $\frac{\pi}{2}$, since the Lipschitz constant of $u$ is small. Thus the angle $\theta_{q'}$ is also close to $\frac{\pi}{2}$. From the Law of Sines
applied to the triangle $\mathcal{T}$ we then conclude
\begin{equation}
|u (\zeta)-g_H (\zeta)| = |p-q| = \frac{\sin \theta_{q'}}{\sin \theta_q} |p-q'|\, .
\end{equation}
By choosing $\eps_2$ small we then reach
\[
\|u-g_H\|_{C^0 (B)} \leq 2 \|v-h_H\|_{C^0 (B')}\, .
\]
As usual, we assume now to have shifted the origin so that $p_H = 0$. Recall that $\Psi_{p_H} (0)=0$ and $D \Psi_{p_H} (0) =0$, so that we can estimate
\[
\|h_H - \etaa\circ f_H\|_{C^0 (B')} \leq C_0 \|\bar{h}_H - \etaa \circ \bar{f}_H\|_{C^0} + C \bmo^{\sfrac{1}{2}} \ell (H)^2\, .
\]
Using now Proposition \ref{p:stime} we then conclude
\begin{equation}\label{e:C0-con-media}
\|h_H - \etaa\circ f_H\|_{C^0 (B')} \leq C \bmo \d (H)^{2 \gamma_0 -2} \ell (H)^3 + C \bmo^{\sfrac{1}{2}} \ell (H)^2\, .
\end{equation}
Since $\ell (H) \leq \d (H)^{(1+b)/2}$, we again see that \eqref{e:cornuto3} can be reduced to the estimate
\begin{equation}\label{e:cornuto4}
\|\etaa \circ f_H - v\|_{C^0 (B')} \leq C \bmo^{\sfrac{1}{4}} \d (H)^{a'}\, .
\end{equation}
We will in fact show such estimate in the ball $\hat{B} := B_{8r_H} (p_H, \pi_H)$. Consider a point $p\in \supp (T_H)\cap \bC_{8r_H} (p_H, \pi_H)$ and
let $p= (\zeta, \eta)\in \pi_0\times \pi_0^\perp$. We also let $q$ be the point $(\zeta, u (\zeta))$ and $q' = (\omega, v (\omega))\in \pi_H\times \pi_H^\perp$, 
where $\omega=\p_{H} (p)$. The argument above can be applied literally to the triangle $\mathcal{T}$ with vertices $p$, $q$ and $q'$ to conclude that
\[
|p-q'|\leq 2 |p-q| \leq C \bmo^{\sfrac{1}{2}} \d (H)^a\, .
\]
Recall that, except for a set of points $\omega\in A$ of measure no larger than $C \bmo \d (H)^{2\gamma_0-2} \ell (H)^4$, the slice $\langle T, \p_{\pi_H}, \omega\rangle$
coincides with the slice $\langle \bG_{f_H}, \p_{\pi_H}, \omega\rangle$. Thus on the set $A$ we obviously have
\[
|\etaa\circ f_H (\omega) - v (\omega)|\leq C \bmo^{\sfrac{1}{2}} \d(H)^a\, .
\]
Now, for any point $\omega\not\in A$ there is a point $\omega'\in A$ at distance at most $d (H)^{\gamma_0-1} \ell (H)^2$. Since both $\Lip (v)$ and $\Lip (\etaa \circ f_H)$ are controlled by $\bmo^{\sfrac{1}{2}}$, this gives the estimate
\[
\|\etaa\circ f_H-v\|_{C^0 (B')}\leq C \bmo^{\sfrac{1}{2}} \d (H)^a + C \d (H)^{\gamma_0-1} \ell (H)^2\, .
\]
On the other hand, since $\ell (H) \leq C \d (H)^{(b+1)/2}$ and $a > b+\gamma_0$ (recall \eqref{e:cost_2}), we easily see that
\[
\|\etaa \circ f_H -v\|_{C^0 (B')} \leq C \bmo^{\sfrac{1}{2}} \d (H)^{\gamma_0 + b}\, .
\]
This completes the proof of \eqref{e:cornuto4} and hence of \eqref{e:cornuto!}

\section{The construction of the approximating map $N$}

In this section we prove Corollary \ref{c:cover} and Theorem \ref{t:approx}.

\subsection{Proof of Corollary \ref{c:cover}} Statement (i) is an obvious consequence of \eqref{e:separation} and \eqref{e:cornuto!}. 
As for statement (ii), the argument is the same given in the proof of Lemma \ref{l:density} for the existence of the nearest point projection $\p: \bV_{u,a} \cap \bC_1 \to \gr (u)$.

For what concerns (iii), let $L\in \sW$, denote by $p_L = (z_L, w_L)$ its center and set $p := \Phii (q)$ We start by observing that $\supp (\langle T, \p, p\rangle)\subset \supp (T_J)$ for the ancestor $J\in \sC^{N_0}$ of $L$, given estimates on $u$ and the definition of $T_J$. We next claim that
\begin{equation}\label{e:claim_intorno}
\supp (\langle T, \p, p\rangle)\subset \bB_{r_L} (p)\, .
\end{equation}
Assuming this for the moment, recall that, by \eqref{e:aggiuntiva_10},
\[
\|\phii - g_L\|_{C^0 (L)} \leq C\, \bmo^{\sfrac12}\,
\d (L)^{2(1+\beta_0)\gamma_0-\beta_2 -2}\,\ell(L)^{3+\kappa }\leq C \bmo^{\sfrac{1}{2}} \d (L)^{\sfrac{\gamma_0}{2} - \beta_2}
\ell (L)^{1+\beta_2}
\]
(where in the last inequality we have used that $\ell (L) \leq \d (L)$ and $\frac{\gamma_0}{4} > \beta_2$).  Recall also that the graph of $g_L$ coincides with that of $h_L$ and, by \eqref{e:sfava_tanto}, 
\[
\|h_L - \eta\|_{C^0 (B_{6 r_H} (p_H, \pi_H))} \leq C \bmo^{\sfrac{1}{4}} \d (L)^{\sfrac{\gamma_0}{2}-\beta_2} \ell (L)^{1+\beta_2}\, ,
\]
where $(\xi, \eta) \in \pi_L \times \pi_L^\perp$ are the coordinates for $p_L$, cf. \eqref{e:sfava_tanto}. Since $\supp (T_J) \cap \bC_{8r_L} (p_L, \pi_L)\subset \supp (T_L)$ for every ancestor $J$ of $L$, we must then have $\supp (\langle T, \p, p\rangle)\subset \supp (\langle T, \p, p\rangle)\cap \bB_{r_L}(p)\subset  \supp (T_L) \cap \bC_{8r_L} (p_L, \pi_L)$. Recalling 
that $p_L\in \supp (T_L)$ and that, by Proposition \ref{p:tilting opt}, we have the bound
\[
\bh (T_L, \bC_{8r_L} (p_L, \pi_L)) \leq C \bmo^{\sfrac{1}{4}} \d (L)^{\sfrac{\gamma_0}{2} - \beta_2} \ell (L)^{1+\beta_2}\, ,
\] 
we conclude that no point of $\supp (\langle T, \p, p\rangle)$ can be at distance larger than 
\[
C \bmo^{\sfrac{1}{4}} \d (L)^{\sfrac{\gamma_0}{2} - \beta_2} \ell (L)^{1+\beta_2}
\] 
from the graph of $h_L$. Putting all these estimates together, no point of $\supp (\langle T, \p, p\rangle)$ can be at a distance larger than $C \bmo^{\sfrac{1}{4}} \d (L)^{\sfrac{\gamma_0}{2} - \beta_2} \ell (L)^{1+\beta_2}$ from $\gr (\phii)$. Since for every $p'\in \supp (\langle T, \p, p\rangle)$ the point $p$ is the closest in the graph of $\phii$, this completes the proof of (iii), provided we show \eqref{e:claim_intorno}.

If \eqref{e:claim_intorno} is false, there is a $p'\in \supp (\langle T, \p, p)$ and an ancestor $J$ of $L$ with largest sidelength among those for which $|p'-p|\geq r_J/2$. Let $\pi$ be the tangent to $\cM$ at $p$ and observe that we have the estimates $|\pi-\pi_J| \leq C \bmo^{\sfrac{1}{2}}$ and $|\pi-\pi_0|\leq C \bmo^{\sfrac{1}{2}}$. The second bound is a trivial consequence
of the estimates on $\phii$, whereas the first is a consequence of \eqref{e:aggiuntiva_11} and $|\pi_L-\pi_J|\leq C \bmo^{\sfrac{1}{2}}$, which in turn follows from Proposition \ref{p:tilting opt}.
If $J$ were an element of $\sS^{N_0}$, Assumption \ref{induttiva}(Hor) would imply $|p'-p|\leq C\bmo^{\sfrac{1}{4}}\,r_J^{1+\gamma_0}$. If $J\not\in \sS^{N_0}$ and we let $H$ be the father of $J$, we then conclude that $q,p,p'\in \bB_H$ and thus we have 
$|p'-p|\leq C_0 \bh (T, \bB_H)\leq C \bmo^{\sfrac{1}{4}} \ell (H)^{1+\beta_2}$ by \eqref{e:ht_whitney}. In both cases this would be incompatible with $|p'-p|\geq r_J=\sfrac{r_H}{2}$, provided $\eps_2 \leq c (\beta_2, \delta_2, M_0, N_0, C_e, C_h)$.

\medskip

We next prove (iv). Fix a point $(z,w)\in \gira$ which belongs to $\bGam$ and set $p:= (z, \phii (z,w)) = \Phii (z,w)$. To prove our statement
we claim in fact that:
\begin{align}
&\mbox{$Q\a{T_p \cM}$ is the unique tangent cone to $T$ at $p$}\label{e:claim-1000}\\
&\supp (T) \cap \p^{-1} (\{p\}) = \{p\}\label{e:claim-1001}.
\end{align}
By construction there is an infinite chain $L_{N_0}\supset L_{N_0+1} \supset \ldots \supset L_i \supset \ldots$ where $(z,w)\in L_i\in \sS^i$ for every $i$.
Set $\pi_i := \pi_{L_i}$. By our construction and the estimates of the previous sections, it is obvious that $\pi_{L_i}\to \pi = T_p \cM$. In fact since
$|\pi_{L_i} - \pi_{L_{i+1}}|\leq C \bmo^{\sfrac{1}{2}} |z|^{\gamma_0+\delta_1-1} \ell (L_i)^{1-\delta_1}$ by Proposition \ref{p:tilting opt}(iv), we easily infer
\begin{equation}\label{e:bound-1002}
|\pi-\pi_{L_i}|\leq C \bmo^{\sfrac{1}{2}} |z|^{\gamma_0+\delta_1-1} \ell (L_i)^{1-\delta_1}\, .
\end{equation}
On the other hand by the height and excess bounds \eqref{e:ex_whitney} and \eqref{e:ht_whitney}, it is also obvious that $T_{p_{L_i}, r_{L_i}}$ converges, in $\bB_1$, to $Q \a{\pi}$. Since $r_{L_i}/r_{L_{i+1}} = 2$ and $p_{L_i}\to p$ (in fact $|\Phii(z,w)-p_{L_i}|\leq C 2^{-i}$), \eqref{e:claim-1000} is then obvious.

Assume now that \eqref{e:claim-1001} is false and let $p'\in \supp (\langle T, \p, p\rangle)$. Again by the width of $\bV$ it turns out that $p'\in \supp (T_{L_{N_0}})$. Let $j$ be the integer such that $2^{-j-1} |z| \leq |p-p'|\leq 2^{-j} |z|$. From Assumption \ref{induttiva}(Hor) it follows that, if $\eps_2$ is sufficiently small, then certainly $j\geq N_0+2$. This means that there is an $L_i$ such that $p'\in \bB_{L_i}$ and obviously $\ell (L_i) \leq C |z| 2^{-j}$. Recall that $\supp (T_{L_{N_0}})\cap \bB_{L_i}\subset \supp (T_{L_i})$ On the other hand, by \eqref{e:bound-1002}, we have 
\[
|p-p'|\leq (1+ C |\pi_{L_i} - \pi|) \bh (T_{L_i}, \bB_{L_i}) \leq C \bmo^{\sfrac{1}{4}} \d (L_i)^{\sfrac{\gamma_0}{2} - \beta_2} \ell (L_i)^{1+\beta_2}\leq C \bmo^{\sfrac{1}{4}} |z|^{1+\sfrac{\gamma_0}{2}} 2^{-j}\, .
\]
Since the constant $C$ depends upon the parameters $C_h, C_e, M_0$ and $N_0$, but not upon $\eps_2$, the latter bound contradicts $|p-p'|\geq 2^{-j-1} |z|$ provided $\eps_2$ is chosen sufficiently small. 

\subsection{Proof of Theorem \ref{t:approx}: Part I} We set $F(p) = Q\a{p}$ for $p\in \Phii (\bGam)$. 
For every $L\in \sW^j$ consider the $\pi_L$-approximating function $f_L: \bC_{8r_L} (p_L, \pi_L)\to
\Iq (\pi_L^\perp)$ of Definition~\ref{d:pi-approximations}
and $K_L\subset B_{8r_L} (p_L, \pi_L)$ the projection on $p_L+\pi_L$ of $\supp (T_L)\cap \gr (f_L)$.
In particular we have $\bG_{f_L\vert_{K_L}}= T_L\res(K_L \times \pi_L^\perp)$.
We then denote by $\mathscr{D} (L)$ the portions of the supports of $T_L$ and $\gr (f_L)$ which differ:
\[
\mathscr{D} (L) := (\supp (T_L) \cup \gr (f_L))\cap \big[(B_{8r_L} (p_L, \pi_L) \setminus K_L)\times \pi_L^\perp\big]\, . 
\]
Observe that, by \cite[Theorem 1.5]{DSS2} and Proposition \ref{p:app} and our choice of the parameters, 
we have, for $E:= \bE (T_L, \bC_{32 r_L} (p_L, \pi_L))$,
\begin{align}\label{e:poco_scazzo}
\cH^m (\mathscr{D} (L)) \leq \|T\| (\mathscr{D} (L)) &\leq C E^{\beta_0} (E + \ell (L)^2 \bmo) \ell(L)^2 \notag\\
&\leq C \bmo^{1+\beta_0} \d (L)^{(1+\beta_0)(2\gamma_0-2+2\delta_1)} \ell (L)^{2+(1+\beta_0)(2-2\delta_1)}\, .
\end{align}
Let $\cL$ be the Whitney region in Definition~\ref{d:cm} and
set $\cL':= \Phii (J)$ where $J$ is the cube concentric to $L$ with
$\ell (J) = \frac{9}{8} \ell (L)$.
Observe that the graphical structure of $\Phii$, our choice of the constants and condition (NN) ensure that 
\begin{align}
&L\cap H = \emptyset \quad \iff \quad
\mathcal{L}' \cap \mathcal{H}' = \emptyset\qquad \forall H, L \in \sW\, ,\label{e:non_si_toccano}\\
&\Phii (\bGam) \cap \cL' = \emptyset \qquad \forall L\in \sW\, .\label{e:nonsitoccano_2}
\end{align}
We then apply \cite[Theorem 5.1]{DS2} to the map $f_L$, the plane $\pi_L$ and the (appropriate portion of the) center manifold $\cM$ as a graph over $\pi_L$
 to obtain Lipschitz maps $F_L: \cL' \to \Iq (\bU)$, $N_L : \cL'\to \Iq (\R^{m+n})$ with the following properties:
\begin{itemize}
\item $F_L (p) = \sum_i \a{p+(N_L)_i (p)}$,
\item $(N_L)_i (p) \perp T_p \cM$ for every $p\in \cL'$ 
\item and $\bG_{f_L}\res (\p^{-1} (\cL')) = \bT_{F_L} \res (\p^{-1} (\cL'))$.
\end{itemize}
For each $L$ consider the set $\sW (L)$ of elements in $\sW$ which have a nonempty intersection with $L$.
We then define the set $\cK$ in the following way:
\begin{equation}\label{e:def_cK}
\cK = (\cM \cap \bC_{2r}) \setminus \Big(\bigcup_{L\in \sW} \Big(\cL' \cap \bigcup_{M\in \sW (L)} \p (\mathscr{D} (M))\Big)\Big)\, . 
\end{equation}
In other words $\cK$ is obtained from $\cM$ by removing in each $\cL'$ those points $x$ for which there is
a neighboring cube $M$ such that the
slice of $\bT_{F_M}$ at $x$ (relative to the projection $\p$) does
not coincide with the slice of $T$.
Observe that, by \eqref{e:nonsitoccano_2}, $\cK$ contains necessarily $\bGam$.
Moreover, recall that $\Lip(\p)\leq C$, that the cardinality of $\sW (L)$ is at most 12 and that each element of $\sW (L)$ has side-length at most twice that
of $L$. Thus \eqref{e:poco_scazzo} implies
\begin{align}
|\cL\setminus \cK|\leq |\cL'\setminus \cK|\leq &\sum_{M\in \sW (L)} \sum_{H\in \sW (M)} \| T_H\| (\mathscr{D} (H))\nonumber\\
\leq & C \bmo^{1+\beta_0} \d (L)^{(1+\beta_0)(2\gamma_0-2+2\delta_1)} \ell (L)^{2+(1+\beta_0)(2-2\delta_1)} \, .\label{e:bound_cK}
\end{align}
By \eqref{e:non_si_toccano}, if $J$ and $L$ are such that $\cJ'\cap \cL'\neq \emptyset$, then $J\in \sW (L)$ and therefore $F_L= 
F_J$ on $\cK\cap (\cJ'\cap \cL')$. We can therefore define a unique map on $\cK$ by simply setting
$F (p) = F_L (p)$ if $p\in \cK\cap \cL'$.
Notice that $\bT_F = T\res \p^{-1} (\cK)$, which implies two facts. First, by Corollary \ref{c:cover}(iii) we also have that $N (p) := \sum_i \a{F_i (p) -p}$ enjoys the bound
\[
\|N|_{\cL\cap \cK}\|_{C^0} \leq C \bmo^{\sfrac14}\, \d (L)^{\sfrac{\gamma_0}{2}-\beta_2}\, \ell (L)^{1+\beta_2}.
\]
Secondly,
\begin{align}
\|T\| (\p^{-1} (\cL\setminus \cK)) \leq &\sum_{M\in \sW (L)} \sum_{H\in \sW (M)} \| T_H\| (\mathscr{D} (H))\nonumber\\
\leq & C \bmo^{1+\beta_0} \d (L)^{(1+\beta_0)(2\gamma_0-2+2\delta_1)} \ell (L)^{2+(1+\beta_0)(2-2\delta_1)}\, .\label{e:bound_scazzo_corrente}
\end{align}
Finally, notice that, by the $C^2$ estimate on $\phii$ and \eqref{e:interpo_4}, $\cM$ is given on $\bC_{r_L} (p_L, \pi_L)$ as the graph of a map $\phii': B_{r_L} (p_L, \pi_L)\to \pi_L^\perp$ with $\|D\phii'\|_{C^0} \leq C \bmo^{\sfrac{1}{2}} \d (H)^{\gamma_0-1+\delta_1} \ell (H)^{1-\delta_1}$ and $\|D^2 \phii'\|_{C^0} \leq C \bmo^{\sfrac{1}{2}} \d (H)^{\gamma_0-1}$. Hence, the Lipschitz constant of $N_L$ can be 
estimated using \cite[Theorem 5.1]{DS2} as
\begin{equation}\label{e:Lip_1}
 \Lip(N_L)\leq C\, \left(\|D^2\phii'\|_{C^0}\,\|N\|_{C^0}+\|D\phii'\|_{C^0}+\Lip(f_L)\right)
\leq C\,\left(\bmo\,\d(L)^{\gamma_0}\,\ell(L)^{\gamma_0}\right)^{\beta_0}\,,
\end{equation}
so that our map has the Lipschitz bound of \eqref{e:Lip_regional}.
We next 
extend $F$ and $N$ to the whole center manifold and conclude \eqref{e:err_regional}
from \eqref{e:bound_scazzo_corrente} and \eqref{e:bound_cK}.
The extension is achieved in three steps:
\begin{itemize}
\item we first extend the map $F$ to a map $\bar{F}$ taking values
in $\Iq (\bV)$;
\item we then modify $\bar{F}$ to achieve the form $
\hat{F} (x) = \sum_i \llbracket x+\hat{N}_i (x)\rrbracket$
with $\hat{N}_i (x) \perp T_x \cM$ for every $x$;
\item in the cases (a) and (c) of Definition \ref{d:semicalibrated} we finally modify $\hat{F}$ to reach the desired extension $F (x) =
\sum_i \a{x+ N_i (x)}$, with $N_i (x) \perp T_x \cM$ and
$x+ N_i (x) \in \Sigma$ for every $x$.
%\footnote{{\color{red} Why not use $\mathbf{\Phi}$ to introduce a global trivialization of the normal bundle of $\cM$? In particular the second two bullets above are simple.} {\color{blue} We could do this. However it takes care only of the second bullet: in the third we still need to express $\Sigma$ as a graph in the trivialization, since $x+N_i (x)$ must be in $\Sigma$. It will shorten the argument (and allow to merge the last two bullets), but not by much.}}
\end{itemize}

\medskip

{\bf First extension}. We use on $\cM$ the coordinates
induced by its graphical structure,
i.e.~we work with variables in flat domains.
% which is obviously a smooth reerization. Given the
%nature of the Lipschitz and $L^\infty$ bounds, we can treat the map $F$ has if its domain were flat. 
Note that the domain parameterizing the
Whitney region for $L\in \sW$ is then the cube concentric to $L$ and
with side-length $\frac{17}{16} \ell (L)$.  
The multivalued map $N$ is extended to a multivalued $\bar{N}$ inductively to appropriate
neighborhoods of the skeleta of the Whitney decomposition (a similar argument has been
used in 
\cite[Section 1.2.2]{DS1}). The extension of $F$ will obviously be 
$\bar{F} (x) = \sum_i \llbracket x+\bar{N}_i (x)\rrbracket$.
The neighborhoods of the skeleta are defined in this way:
\begin{enumerate}
\item if $p$ belongs to the $0$-skeleton, we let $L\in \sW$ be (one of) the smallest cubes 
containing it and define $U^p := B_{\ell(L)/16} (p)$;
\item if $\sigma= [p,q]\subset L$ is the edge of a cube and $L\in \sW$ is (one of) the smallest cube intersecting $\sigma$, we then define $U^\sigma$ to be the neighborhood of size $\frac{1}{4}\frac{\ell(L)}{16}$ of $\sigma$ minus the
closure of the unions of the $U^r$'s, where $r$ runs in the $0$-skeleton.
\end{enumerate} 
Denote by $\bar{U}$ the closure of the union of all
these neighborhoods and let $\{V_i\}$ be the connected components of the complement.
For each $V_i$ there is a $L_i\in \sW$ such that $V_i\subset L_i$.
Moreover, $V_i$ has distance $c_0 \ell(L)$ from $\partial L_i$, where $c_0$ is a geometric constant. 
It is also clear that if $\tau$ and $\sigma$ are two distinct facets of the same cube $L$ with the same
dimension, then the distance between any pair of points $x,y$ with $x\in U^\tau$ and $y\in U^\sigma$ is at least $c_0 \ell (L)$.
In Figure \ref{f:whitney} the various domains are shown in a piece of a $2$-dimensional decomposition.

\begin{figure}[htbp]
\begin{center}
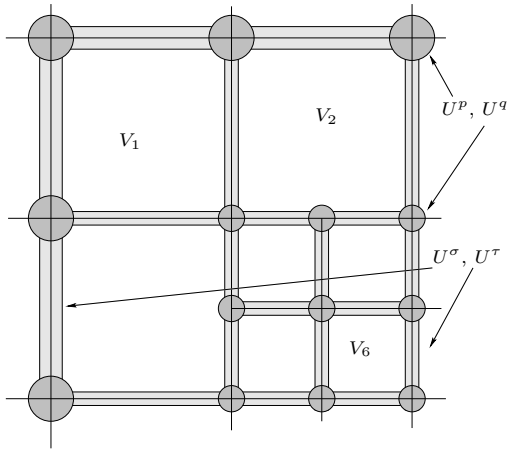
\caption{The sets $U^p$, $U^\sigma$ and $V_i$.}
\label{f:whitney}
\end{center}
\end{figure}

At a first step we extend $N$ to a new map $\bar{N}$ separately on each $U^p$, where $p$ are 
the points in the $0$-skeleton. In particular, fix $p\in L$ and let ${\rm St} (p)$ be the union of all cubes which contain $p$. Observe that the Lipschitz constant of 
$N|_{\cK\cap {\rm St} (p)}$ is smaller than 
\[
C \left(\bmo\,\d (L)^{\gamma_0} \, \ell (L)^{\gamma_0}\right)^{\beta_0}
\]
and that 
\[
|N| \leq C \bmo^{\sfrac14} \d (L)^{\sfrac{\gamma_0}{2}-\beta_2}
\ell (L)^{1+\beta_2}\, .
\]
We can therefore extend the map $N|_{\cK\cap {\rm St} (p)}$ to $U^p \cup (\cK \cap {\rm St} (p))$ at the price
of enlarging the Lipschitz constant and the height bound by a multiplicative constant, using \cite[Theorem 1.7]{DS1}. 
Being the $U^p$ disjoint, the resulting map, for which we use the symbol $\tilde{N}$, is well-defined.

It is obvious that this map has the desired height bound in each Whitney region. We therefore 
want to estimate its
Lipschitz constant.
Consider $L\in \sW$ and $H$ concentric to $L$ with side-length $\ell (H) = \frac{17}{16} \ell (L)$.
Let $x,y\in H$. If $x,y\in U^p \cup (\cK \cap {\rm St} (p))$ for some $p$, then there is nothing to check.
%If $y\in U^p$ for some $p$ and $x\not\in \bigcup_q U^q$,
%then $x\in {\rm St} (p)$ and
% \[
% \cG (\tilde{N}(x), \tilde{N}(y)) \leq C \left(\bmo\d(L)^{\gamma_0}\ell(L)^{\gamma_0}\right)^{\beta_0}  |x-y|\, .
% \]
% The same holds when $x,y\in U^p$. 
If $x\in U^p$ and $y\in U^q$ with $p\neq q$, observe however that this would imply that $p,q$ are both vertices of $L$. Given
that $L\setminus \cK$ has much smaller measure than $L$ there is at least one point $z\in L\cap \cK$. It
is then obvious that 
\[
\cG (\bar{N} (x), \bar{N} (y)) \leq \cG (\bar{N} (x), \bar{N} (z)) +
\cG (\bar{N} (z), \bar{N} (y)) \leq C \left(\bmo \d(L)^{\gamma_0}\,\ell(L)^{\gamma_0}\right)^{\beta_0}  \ell (L),
\]
and, since $|x-y|\geq c_0 \ell (L)$, the desired bound readily follows. 
Observe moreover that, if $x$ is in the closure of some $U^q$, then we can extend the map continuously to it.
By the properties of the Whitney decomposition it follows that the union of the closures of the $U^q$ and of $\mathcal{K}$ is closed and thus, w.l.o.g., we can assume that the domain of this new $\bar{N}$ is in fact closed.

We can repeat this procedure with the edges of the skeleta, that is in the argument above we simply replace 
points $p$ with $1$-dimensional faces $\sigma$, defining ${\rm St} ( \sigma)$ as the union of the cubes which
contain $\sigma$. In the final step we then extend over the domains $V_i$'s:
this time ${\rm St} (V_i)$
will be defined as the union of the cubes which intersect the cube $L_i\supset V_i$. 
The correct height and Lipschitz bounds follow from  the same arguments. Since the algorithm is applied
$3$ times, the original constants have been enlarged by a geometric factor.

\medskip

{\bf Second extension.}
For each $x\in \cM$ let $\p^\perp (x, \cdot) : \R^{m+n}\to \R^{m+n}$ 
be the orthogonal projection on $(T_x\cM)^\perp$ and set
$\hat{N} (x) = \sum_i \llbracket \p^\perp (x,\tilde{N}_i (x))\rrbracket$. Obviously $|\hat{N} (x)| \leq |\tilde{N} (x)|$,
so the $L^\infty$ bound is trivial. We now want
to show the estimate on the Lipschitz constant. 
To this aim, fix two points $p,q$ in the same Whitney region associated to $L$ and parameterize the corresponding
geodesic segment $\sigma\subset \cM$ by arc-length $\gamma :[0, d (p,q)] \to\sigma$, where $d(p,q)$ 
denotes the geodesic distance on $\cM$. Use \cite[Proposition~1.2]{DS1}
to select $Q$ Lipschitz functions $N'_i: \sigma \to \bU$ such that
$\tilde{N}|_{\gamma} = \sum \a{N'_i}$ and 
$\Lip (N'_i) \leq \Lip (\tilde{N})$. Fix a frame $\nu_1, \ldots, \nu_n$
on the normal bundle of $\cL\subset \cM$ with the property that $\|\nu_i\|_{C^0 (\cL)}\leq C \|D\phii\|_{C^0}
\leq C\,\bmo^{\sfrac12}\,\d (L)^{\gamma_0}$ and $\|D \nu_i\|_{C^0 (\cL)}\leq C \|D^2\phii\|_{C0}\leq
\bmo^{\sfrac12} \d (L)^{\gamma_0-1}$ (which is possible by~\cite[Appendix A]{DS2}, indeed we do this in $\cM\setminus \{0\}$, where our manifold is $C^{3,\gamma_0}$). We have
$\hat{N} (\gamma (t)) = \sum_i \llbracket\hat{N}_i (t)\rrbracket$, where
\[
\hat{N}_i (t) = \sum [\nu_j (\gamma (t)) \cdot N'_i (\gamma (t))] \, \nu_j (\gamma(t)).
\]
Hence we can estimate
\begin{align*}
\left|\frac{d \hat{N}_i}{dt}\right| &\leq  C \sum_j[ \|D\nu_j\| \|N'_i\|_{C^0}+\Lip(N'_i)]
\leq C \left(\bmo\,\d (L)^{\gamma_0}\,\ell(L)^{\gamma_0}\right)^{\beta_0}\,.
\end{align*}
Integrating this inequality we find
\[
\cG (\hat{N} (p), \hat{N} (q)) \leq 
\sum_{i=1}^Q |\hat{N}_i (d (p,q))-\hat{N}_i(0)|\leq C \left(\bmo\,\d (L)^{\gamma_0}\,\ell(L)^{\gamma_0}\right)^{\beta_0} d (p,q) \, .
\]
Since $d(p,q)$ is comparable to $|p-q|$, we achieve the desired
Lipschitz bound.

\medskip

{\bf Third extension and conclusion.} We still need to modify the map $\hat{N}$ in the cases (a) and (c) of Definition \ref{d:semicalibrated}. 
For each $x\in \cM\subset \Sigma$ consider the orthogonal
complement $\varkappa_x$ of $T_x \cM$ in $T_x \Sigma$. Let $\cT$ be the fiber bundle 
$\bigcup_{x\in \cM\setminus \{0\}} \varkappa_x$ and observe that, by the regularity of both $\cM\setminus\{0\}$
 and $\Sigma$, there is a $C^{2,\gamma_0}$ trivialization (argue as in \cite[Appendix A]{DS2}).
It is then obvious that there is a $C^{0,\gamma_0}$ map $\Xi: \cT \to \R^{m+n}$
with the following property: for each $(x,v)$, $q:= x+\Xi (x,v)$ is the only point in $\Sigma$ which is
orthogonal to $T_x \cM$ and such that $\p_{\varkappa_x} (q-x) = v$. Let us
denote by $\Omega (x, q)$ the map $\Xi (x, \p_{\varkappa_x} (q))$. This map extends to a $C^{0, \gamma_0}$ 
map at the origin with the estimates
\begin{align}
|D_x \Omega (x, q)|\leq C \bmo^{\sfrac{1}{2}} |x|^{\gamma_0-1} \qquad \forall x\in \gira\setminus \{0\}\quad \forall q \;\mbox{with $|q|\leq 1$}\label{e:bound_degenere_1}\\
|D^2_x \Omega (x, q)|\leq C \bmo^{\sfrac{1}{2}} |x|^{\gamma_0-2} \qquad \forall x\in \gira\setminus \{0\}\quad \forall q \;\mbox{with $|q|\leq 1$}\label{e:bound_degenere_2}
\end{align}
We then set $N (x) = \sum_i \llbracket\Omega (x, \hat{N}_i (x))\rrbracket$. 
Obviously, $N(x) = \hat{N} (x) = 0$ for $x\in \cK$, simply because in this case $x+\hat N_i (x) = x$ belongs
to $\Sigma$. 

In order to show the Lipschitz bound, notice that, by the regularity of $\mathcal{M}$ and $\Sigma$,
\begin{equation}\label{e:Lip_1bis}
|\Omega (x, q) - \Omega (x, p)|\leq C\, |q-p|\, .
\end{equation}
Moreover, since $\Omega (x, 0) = 0$ for every $x\in \cM\subset \Sigma$, we have $D_x \Omega (x, 0) =0$. 
We therefore conclude that $|D_x \Omega (x, q)|\leq C\bmo^{\sfrac{1}{2}} |x|^{\gamma_0-1} |q|$ and hence that
\begin{equation}\label{e:Lip_2}
|\Omega (x, q) - \Omega (y, q)|\leq C\bmo^{\sfrac{1}{2}} |x|^{\gamma_0-1}  |q| |y-x|\, .
\end{equation}
Thus, fix two points $x,y$ in any Whitney region $\cL$ and let us assume that $\cG (\hat{N} (x), \hat{N} (y))^2 =
\sum_i |\hat{N}_i (x) - \hat{N}_i (y)|^2$ (which can be achieved by a simple relabeling).
We then conclude
\begin{align}
\cG (N (x), N(y))^2 &\leq 2 \sum_i |\Omega (x, \hat{N}_i (x)) - \Omega (x, \hat{N}_i (y))|^2 
+ 2 \sum_i |\Omega (x, \hat{N}_i (y)) - \Omega (y, \hat{N}_i (y))|^2\nonumber\\
&\leq C \,\bmo^{\sfrac12}\, \cG (\hat{N} (x), \hat{N} (y))^2 + C\,|x|^{2\gamma_0-2} \sum_i |\hat{N}_i (y)|^2 |x-y|^2\nonumber\\
&\leq C \left(\bmo \,\d(L)^{\gamma_0}\,\ell(L)^{\gamma_0}\right)^{2\beta_0} |x-y|^2\\
&\qquad+ C \bmo  \d (L)^{2\gamma_0-2+\gamma_0-2\beta_2} \ell (L)^{2+2\beta_2} |x-y|^2\notag\\
&\leq C\,\left(\bmo \,\d(L)^{\gamma_0}\,\ell(L)^{\gamma_0}\right)^{2\beta_0} |x-y|^2 \,,
\end{align}
which proves the desired Lipschitz bound. 
Finally, using the fact that
$\Omega (x,0) =0$, we have $|\Omega (x,v)|\leq C |v|$ and the $L^\infty$ bound readily follows. 

\subsection{Proof of Theorem \ref{t:approx}, Part II} In this section we show the estimates \eqref{e:Dir_regional} and \eqref{e:av_region}. We start with the first one. Fix a Whitney region $\cL$ and a corresponding square $L\in \sW$. First consider
the cylinder $\bC:= \bC_{8r_L} (p_L, \pi_L)$, the interpolating function $g_L$ and the tilted interpolating function $h_L$. Denote by $\vec{\cM}$ the unit $m$-vector 
orienting $T\cM$ and by $\vec\tau$ the one orienting $T \bG_{h_L} = T\bG_{g_L}$. 
Recalling that $g_L$ and $\phii$ coincide in a neighborhood of $(z_L, w_L)$ of $L$, by 
Theorem \ref{t:cm} we have
\[
\sup_{p\in \cM\cap \bC} |\vec{\tau} (z_L, g_L (z_L, w_L)) -\vec{\cM} (p)|\leq C \|D^2 \phii\|_{C^0} \,\ell (L)
\leq C \bmo^{\sfrac{1}{2}} \d (L)^{\gamma_0-1}\ell (L).
\]
On the other hand recalling \eqref{e:interpo_4} in Proposition \ref{p:main_est}, we have
\[
|\pi_L - \vec{\tau} (z_L, g_L (z_L, w_L))| \leq C \bmo^{\sfrac{1}{2}} \d (L)^{\gamma_0-1+\delta_1} \ell (L)^{1-\delta_1}\, .
\] 
This in turn implies that 
\begin{equation}\label{e:planesss}
\sup_{\bC\cap \cM} |\vec{\cM}-\pi_L| \leq C \bmo^{\sfrac12} \d (L)^{\gamma_0-1 +\delta_1} \ell (L)^{1-\delta_1}. 
\end{equation}
Therefore, we can estimate
\begin{align}
\int_{\p^{-1} (\cL)} |\vec{\bT}_F (x) -& \vec{\cM} (\p (x))|^2 \,d\|\bT_F\| (x)\nonumber\\ 
\leq\; &C \int_{\p^{-1} (\cL)} |\vec{T} (x) - \vec{\cM} (\p (x))|^2 \,d\|T\| (x) + C \bmo^{1+\beta_0} \d (L)^{2(1+\beta_0)\gamma_0-2}
\ell (L)^{4}\nonumber\\
\leq\; &C \int_{\p^{-1} (\cL)} |\vec{T} (x) - \pi_L|^2\, d\|T\| (x) + C \bmo \d(L)^{2\gamma_0-2+2\delta_1} \ell (L)^{4 -2\delta_1}\, .\label{e:eccesso_storto}
\end{align}
Since $\p^{-1} (\cL) \cap \supp (T_{L}) \subset \bC$, the integral in \eqref{e:eccesso_storto} 
is bounded by $C \ell (L)^2 \bE (T_L, \bC, \pi_L)$.
By \cite[Proposition 3.4]{DS2} and using Proposition \ref{p:app} we then conclude
\begin{align*}
\int_{\cL} |DN|^2 &\leq C \int_{\p^{-1} (\cL)} |\vec{\bT}_F (x) - \vec{\cM} (\p (x))|^2 \,d\|\bT_F\| (x) + 
C \|A_{\cM}\|_{C^0 (\bC_{\d (L))}\setminus \bC_{\d (L)/4})}^2 \int_{\cL} |N|^2\\
&\leq C \bmo \, \d (L)^{2\gamma_0-2+2\delta_1}\,\ell(L)^{4-2\delta_1} + C \bmo\,\d (L)^{2\gamma_0-2} \ell (L)^{4 +2\beta_2}\, ,
\end{align*}
where we have used $\|A_{\cM}\|_{C^0 (\bC_{\d (L))}\setminus \bC_{\d (L)/4})} \leq C \bmo^{\sfrac{1}{2}}\dist(L,0)^{\gamma_0-1}$.
This shows \eqref{e:Dir_regional}.

\medskip

We finally come to \eqref{e:av_region}. First observe that, by \eqref{e:Lip_regional} and \eqref{e:err_regional},
\begin{align}\label{e:sfava1}
\int_{\cL \setminus \cK} |\etaa \,\circ N|\, 
&\leq C \, \bmo^{\sfrac14} \d (L)^{\sfrac{\gamma_0}{2}-\beta_2} \ell(L)^{1+\beta_2} |\cL\setminus \cK|\notag\\
&\leq C \bmo^{1+\beta_0+\sfrac14 } \d (L)^{(1+\beta_0)(2\gamma_0-2+2\delta_1)+(\sfrac{\gamma_0}{2}-\beta_2)}\ell(L)^{3+\beta_2+(1+\beta_0)(2-2\delta_1)}\, .
\end{align}
Fix now $p\in \cK$. 
Recalling that $F_L (p) = \sum_j \a{p+ N_j (p)}$ is given
by \cite[Theorem 5.1]{DS2} applied to the map $f_L$, we can use \cite[Theorem 5.1(5.4)]{DS2}
to conclude
\begin{align}
|\etaa \circ N_{L} (p)| \leq {}& C\, |\etaa \circ f_{L} (\p_{\pi_{L}}(p)) - \p_{\pi_{L}}^\perp (p)| + 
C \, \Lip (f_L) \, |T_p \cM - \pi_{L}| \, |N_{L}| (p)\nonumber\\
\stackrel{\eqref{e:planesss}}{\leq} {}& C |\etaa \circ f_{L} (\p_{\pi_{L}}(p)) - \p_{\pi_{L}}^\perp (p)| \notag\\
&+ C \bmo^{\sfrac{1}{2} + \beta_0}\, \d (L)^{\beta_0(2\gamma_0-2+2\delta_1)+\gamma_0-1+\delta_1} \ell (L)^{1-\delta_1+\beta_0(2-2\delta_1)}\\
&\qquad \cdot \big[\cG (N_{L}(p), Q \a{\etaa \circ N_{L} (p)}) + Q|\etaa \circ N_{L}| (p)\big]\, .\nonumber
\end{align}
Note that with a slight abuse of notation we have denoted by $\p_{\pi_L}$ the orthogonal projection onto $p_L+\pi_L$ (rather than onto $\pi_L$; alternatively we could shift the origin so that $p_L=0$). 
For $\eps_{2}$ sufficiently small (depending only on $\beta_2, \gamma_2, M_0, N_0, C_e, C_h$), we then conclude that
\begin{align}
|\etaa \circ N_{L} (p)|& \leq  C\, |\etaa \circ f_{L} (\p_{\pi_{L}}(p)) - \p_{\pi_{L}^\perp} (p)|\nonumber \\
&+ C  \bmo^{\sfrac{1}{2} + \beta_0}\,
\d (L)^{\beta_0(2\gamma_0-2+2\delta_1)+\gamma_0-1+\delta_1}\,\ell (L)^{1-\delta_1+\beta_0(2-2\delta_1)}
\cG (N_L(p), Q \a{\etaa \circ N_L (p)}) \label{e:da_integrare_10}
\end{align}
Let next $\phii': p_L+\pi_{L}\to \pi_{L}^\perp$ such that $\bG_{\phii'} = \cM$. 
Applying \cite[Lemma B.1]{DS4} we conclude that
\[
\int_{\cK \cap \cV} |\etaa\circ f_{L} (\p_{\pi_{L}}(p)) - \p_{\pi_{L}^\perp} (p))| \leq 
\int_{\p_{\pi_{L}} (\cK \cap \cV)} |\etaa \circ f_{L} (x) - \phii' (x)|
\leq C\|g_{L}- \phii\|_{L^1 (H)}\, ,
\]
where $H$ is a cube concentric to $L$ with side-length $\ell (H) = \frac{9}{8} \ell (L)$.
Next assume $L\in \sW^j$ and let $k\geq j+2$.
Consider the subset $\sP^k (L)$ of all cubes in $\sP^k$ which intersect $L$ and recall that $\phii$ coincides with the map $\varphi_k$ on $H$ (recall Definition \ref{d:glued} and Theorem \ref{t:cm}). Thus we can estimate
\begin{align}\label{e:L1_importante}
\|\phii - g_{L}\|_{L^1 (H)} 
&\leq C \sum_{L'\in \sP^k (L)} \|g_{L'} - g_{L}\|_{L^1(B_{r_L} (p_{L}, \pi_0))} \notag\\
&\leq C \bmo\,\d (L)^{2(1+\beta_0)\gamma_0-2-\beta_2} \ell (L)^{5+\kappa}\, ,
\end{align}
where in the last inequality we used \eqref{e:interpo_5}. We then conclude
\[
\|\phii- g_{L}\|_{L^1(H)} \leq C \bmo \dist(L)^{2(1+\beta_0)\gamma_0-2-\beta_2} \ell (L)^{5+\kappa} 
\]
and \eqref{e:av_region} follows integrating \eqref{e:da_integrare_10} over $\cV\cap \cK$ and using \eqref{e:sfava1}.

\section{Separation and splitting before tilting}

\subsection{Vertical separation} 
In this section we prove Proposition \ref{p:separ} and Corollary \ref{c:domains}.

\begin{proof}[Proof of Proposition \ref{p:separ}]
Let $J$ be the father of $L$. 
By Lemma \ref{l:tecnico2} and Proposition \ref{p:whitney}, Theorem \ref{t:height_bound} can
be applied to the cylinder $\bC := \bC_{32r_J} (p_J, \pi_J)$. 
Moreover, 
$|p_J-p_L|\leq C_0 \ell (J)$, where $C_0$ is a geometric constant, and $r_J = 2 r_L$. 
Thus, if $M_0$ is
larger than a geometric constant, 
we have
$\bB_L \subset \bC_{31 r_J} (p_{J}, \pi_{J})$. 
Denote by $\q_{L}$, $\q_{J}$ the projections $\p_{\pi_{L}^\perp}$
and $\p_{\pi_{J}^\perp}$ respectively. Since $L \in \sW_h$, 
there are two points $p_1, p_2\in
\supp (T_{L}) \cap \bB_{L}$ such that 
\[
|\q_{L} (p_1-p_2)| \geq C_h \bmo^{\sfrac{1}{4}}\,\d (L)^{\sfrac{\gamma_0}{2}-\beta_2} \ell (L)^{1+\beta_2}\,.
\]
On the other hand, recalling Proposition \ref{p:tilting opt}(iv), $|\pi_{J}- \pi_{L}|\leq \bar C \d (L)^{\gamma_0-1+\delta_1}
\ell (L)^{1-\delta_1}$, where $\bar C$ depends upon all the parameters except $C_h$ and $\eps_2$. Thus,
\begin{align*}
|\q_{J} (p_1-p_2)| &\geq |\q_{L} (p_1-p_2)| - C_0 |\pi_{L}-\pi_{J}| |p_1-p_2|\\
&\geq C_h \bmo^{\sfrac{1}{4}} \d (L)^{\sfrac{\gamma_0}{2}-\beta_2}\, \ell (L)^{1+\beta_2} - 
\bar C \bmo^{\sfrac{1}{2}} \d (L)^{\gamma_0-1+\delta_1}\,\ell (L)^{2-\delta_1}\\
&\geq C_h \bmo^{\sfrac{1}{4}} \d (L)^{\sfrac{\gamma_0}{2}-\beta_2}\, \ell (L)^{1+\beta_2} - 
\bar C \bmo^{\sfrac{1}{2}} \d (L)^{\sfrac{\gamma_0}{2}-\beta_2}\,\ell (L)^{1+\beta_2}\, ,
\end{align*}
where $C_0$ is a geometric constant and $\bar{C}$ a constant which does not depend on $C_h$ and $\eps_2$.
Hence, if $\eps_2$ is sufficiently small, we actually conclude
\begin{equation}
|\q_{J} (p_1-p_2)| \geq \frac{15}{16} C_h \bmo^{\sfrac{1}{4}} \d (L)^{\sfrac{\gamma_0}{2}-\beta_2}\, \ell (L)^{1+\beta_2}\, .
\end{equation}
Set $E:= \bE (T_J, \bC_{32 r_J} (p_{J}, \pi_{J}))$ and apply Theorem \ref{t:height_bound} to $T_J$ and $\bC$: the union of the corresponding
``stripes'' $\bS_j$ contain the set $\supp (T_J) \cap \bC _{32 r_{J} (1 - C E^{\sfrac{1}{24}} |\log E|)} (p_{J}, \pi_{J}))$,
where $C$ is a geometric constant. We can therefore assume that they contain
$\supp (T_{L}) \cap \bC_{31r_{J}} (p_{J}, \pi_{J})$. The width of these stripes is bounded as follows:
\begin{align*}
\sup \big\{|\q_{J} (x-y)| :x,y\in \bS_j\big\}
&\leq C_0 \,E^{\sfrac{1}{4}} r_{J} \leq C_0\, C_e^{\sfrac{1}{4}} \bmo^{\sfrac{1}{4}} \d (L)^{(2\gamma_0-2+2\delta_1)/4} \ell (L)^{1+ (2-2\delta_1)/4}\\
&\leq C_0 C_e^{\sfrac14}\bmo^{\sfrac14} \d (L)^{\sfrac{\gamma_0}{2}-\beta_2}\, \ell(L)^{1+\beta_2}
\end{align*}
where $C_0$ is a geometric constant. So, if $C^\sharp$ is chosen large enough, we actually conclude that $p_1$
and $p_2$ must belong to two different stripes, say $\bS_1$ and $\bS_2$. 
By Theorem \ref{t:height_bound}(iii) we conclude that all points in $\bC_{31 r_{J}} (p_{J}, \pi_J)$ have density $\Theta$ strictly
smaller than $Q-1$ (because in our case we know that $\Theta$ is {\em everywhere} integer-valued), thereby implying (S1). Moreover, by choosing $C^\sharp$ appropriately, we achieve that
\begin{equation}\label{e:vert_separ}
|\q_{J} (x-y)| \geq  \frac{7}{8} C_h \bmo^{\sfrac{1}{4}} \,\d (L)^{\sfrac{\gamma_0}{2}-\beta_2}\, \ell (L)^{1+\beta_2} 
\quad \forall x\in \bS_1, y\in \bS_2\, .
\end{equation} 
Assume next there is $H\in \sW_n$ with $\ell (H) \leq \frac{1}{2} \ell (L)$ and
 $H\cap L\neq\emptyset$. 
From our construction and Proposition \ref{p:tilting opt}(iv) it follows that $\ell (H) = \frac{1}{2} \ell (L)$, $\d (H) \leq 2 \d (L)$, $\bB_{H}\subset
\bC_{31 r_J} (p_{J}, \pi_{J})$ and $|\pi_{H} -\pi_{J}|\leq \bar C\bmo^{\sfrac{1}{2}}\, \d (L)^{\gamma_0-1+\delta_1} \ell (H)^{1-\delta_1}$, 
with $\bar C$ which does not depend upon $C_h$ and $\eps_2$. Hence choosing $\eps_2$ sufficiently small we conclude
\begin{align}
|\q_H (x-y)| 
&\geq  \frac{3}{4} C_h \bmo^{\sfrac{1}{4}}\,\d(L)^{\sfrac{\gamma_0}{2}-\beta_2}\, \ell (L)^{1+\beta_2}\notag \\
&\geq \frac{3}{2}\Bigl(\frac{1}{2}\Bigr)^{\sfrac{\gamma_0}{2}} C_h \bmo^{\sfrac{1}{4}}\, \d (H)^{\sfrac{\gamma_0}{2}-\beta_2} \ell (H)^{1+\beta_2} \notag \\
&\geq \frac{5}{4} C_h \bmo^{\sfrac{1}{4}}\, \d (H)^{\sfrac{\gamma_0}{2}-\beta_2} \ell (H)^{1+\beta_2}\qquad \forall x\in \bS_1, y\in \bS_2\,\label{e:vert_separ_10},
\end{align}
where the latter inequality holds because $\gamma_0\leq \log_2 \frac{6}{5}$. Now, recalling Proposition \ref{p:tilting opt}, if $\eps_{2}$ is sufficiently small, $\bC_{30 r_{H}} (p_{H}, \pi_{H}) \cap
\supp (T_{H}) \subset \bB_{H}$ and $\supp (T_J)\cap \bB_H \subset \supp (T_H)$. Moreover, by Theorem \ref{t:height_bound}(ii) ,
\[
(\p_{\pi_{J}})_\sharp (T_{J}\res (\bS_j\cap \bC_{30 r_{H}} (p_{H}, \pi_{J}))) = Q_j \a{B_{30 r_{H}} (p_{H}, \pi_{J})}
\quad \text{for } \, j=1,2,\;\; Q_j\geq 1.
\] 
A simple argument already used several other times allows to conclude that indeed
 \[
(\p_{\pi_{H}})_\sharp (T_{H}\res (\bS_j\cap \bC_{30 r_{H}} (p_{H}, \pi_{H}))) = Q_j \a{B_{30 r_{H}} (p_{H}, \pi_{H})}
\quad \text{for } \, j=1,2,\;\; Q_j\geq 1.
\] 
%for both $i\in \{1,2\}$, with $Q_i\geq 1$.
Thus, $\bB_{H}\cap \supp (T_H)$ must necessarily contain two points $x,y$ with 
\[
|\q_H (x-y)|\geq \frac{5}{4} 
C_h \bmo^{\sfrac{1}{4}}\, \d(H)^{\sfrac{\gamma_0}{2}-\beta_2} \ell (H)^{1+\beta_2}.
\]
But then the refining in $H$ should have stopped because of condition (HT) and so $H$ cannot belong to $\sW_n$.

\medskip

Coming to (S3), set $\Omega:= \Phii (B_{8 \ell (L)} ((z_L,w_L))$ and observe that $\p_\sharp (T\res (\p^{-1} (\Omega)\cap \bS_i)) = Q_i \a{\Omega}$. Thus, for each $p\in \cK\cap \Omega$,
the support of $p + N(p)$ must contain at least
one point $p+N_1 (p) \in \bS_1$ and at least one point $p+N_2 (p)\in \bS_2$. 
Now,
\begin{equation}\label{e:separazione parametrica}
|N_1 (p)- N_2 (p)| \geq  \frac{7}{8} C_h \bmo^{\sfrac{1}{4}} \d (L)^{\sfrac{\gamma_0}{2}-\beta_2} \ell (L)^{1+\beta_2} - C_0 \ell (L) \,
|T_p \cM - \pi_{J}|\, .
\end{equation}
Recalling, however, Proposition~\ref{p:main_est} and
that $\cM$ and $\gr (g_{J})$ coincide on a nonempty open set, we easily conclude that (see for instance the proof of \eqref{e:Dir_regional})
$|T_p \cM - \pi_{J}|\leq C \bmo^{\sfrac{1}{2}} \d (L)^{\gamma_0-1 + \delta_1} \ell (L)^{1-\delta_1}$
and, via \eqref{e:separazione parametrica},
\[
\cG \big(N (p), Q \a{\etaa \circ N (p)}\big) \geq \frac{1}{2} |N_1 (p)- N_2 (p)|\geq \frac{3}{8} C_h \bmo^{\sfrac{1}{4}}\d (L)^{\sfrac{\gamma_0}{2}-\beta_2} \ell (L)^{1+\beta_2}\, .
\]
Next observe that, by the property of the Whitney decomposition, any cube touching $B_{4 \ell (L)} ((z_L, w_L))$ has sidelength at most $4 \ell (L)$. Thus the sum of $\ell (H)^2$ over all such $H$ is the $2$-dimensional measure of a region of diameter comparable to $\ell (L)$ and from \eqref{e:err_regional} we infer
\[
|\Omega\setminus \cK| \leq C \bmo^{1+\beta_0} \d (L)^{(1+\beta_0)(2\gamma_0-2+2\delta_1)}\ell (L)^{2+(1+\beta_0)(2-2\delta_1)}\, .
\]
So, for every point $x\in \Omega$ there exists $q\in \cK\cap \Omega$ which has geodesic distance to $x$ at most 
$C \bmo^{\sfrac{1}{2}+\sfrac{\beta_0}{2}}\,\d (L)^{(1+\beta_0)(\gamma_0-1+\delta_1)}  \ell(L)^{1+ (1+\beta_0)(1-\delta_1)}$. Given the Lipschitz bound for $N$ and the choice $\beta_2 \leq \frac{1}{4}$,
we then easily conclude (S3):
\begin{align*}
\cG (N (x), Q \a{\etaa \circ N (x)}) 
\geq &\frac{3}{8} C_h \bmo^{\sfrac{1}{4}}\,\d (L)^{\sfrac{\gamma_0}{2}-\beta_2} \ell (L)^{1+\beta_2}\\
&\quad- C \bmo^{\sfrac{1}{2}+\sfrac{3\beta_0}{2}}  \d (L)^{(\sfrac{3\beta_0}{2}+1)\gamma_0-\beta_2}\ell (L)^{1+\beta_2}\\
\geq & \frac{1}{4} \,C_h \bmo^{\sfrac14}\,\d (L)^{\sfrac{\gamma_0}{2}-\beta_2} \ell (L)^{1+\beta_2}\, ,
\end{align*}
where again we need $\eps_2 < c (\beta_2, \delta_2, M_0, N_0, C_e, C_h)$ for a sufficiently small $c$. 
\end{proof}

\begin{proof}[Proof of Corollary \ref{c:domains}]
The proof is straightforward. Consider any $H\in \sW^j_n$. By definition it has
a nonempty intersection with some cube $J\in\sW^{j-1}$: this cube cannot belong to $\sW_h$ by
Proposition \ref{p:separ}. It is then either an element of $\sW_e$ or an element
$H_{j-1}\in \sW^{j-1}_n$. Proceeding inductively, we then find a chain $H = H_j, H_{j-1}, \ldots,
H_i =: L$, where $H_{\bar{l}}\cap H_{\bar{l}-1} \neq\emptyset$ for every $\bar{l}$, $H_{\bar l} \in \sW^{\bar{l}}_n$ for every $\bar{l}>i$
and $L= H_i\in \sW^i_e$. 
Observe also that
\[
|x_H-x_L| \leq \sum_{\bar{l}=i}^{j-1} |x_{H_{\bar{l}}} - x_{H_{\bar{l}+1}}| \leq \sqrt{2}\, \ell (L) 
\sum_{\bar{l}=0}^\infty 2^{-\bar{l}}
\leq 2\sqrt{2} \,\ell (L)\, . 
\]
It then follows easily that $H\subset B_{3\sqrt{2} \ell (L)} (L)$.
\end{proof}

\subsection{Unique continuation for Dir-minimizers}
We recall for completeness the following two Propositions, whose proof can be found in \cite{DS4}, cf. Lemmas 7.1 and 7.2 therein.

\begin{lemma}[Unique continuation for $\D$-minimizers]\label{l:UC}
Let $n\in \mathbb N\setminus \{0\}$ be fixed.
For every $\eta \in (0,1)$ and $c>0$, there exists $\gamma>0$ with the following property.
If $w: \R^2\supset B_{2\,r} \to \Iq(\R^n)$ is Dir-minimizing,
$\D(w, B_r)\geq c$ and  $\D(w, B_{2r}) =1$,
then
\[
\D (w, B_s (q)) \geq \gamma \quad \text{for every 
$B_s(q)\subset B_{2r}$ with $s \geq \eta\,r$}.
\]
\end{lemma}

In the sequel we fix $\lambda>0$ such that 
\begin{equation}\label{e:lambda}
(1+\lambda)^{4}< 2^{\delta_1}\, .
\end{equation}

\begin{proposition}[Decay estimate for $\D$-minimizers]\label{p:harmonic split}
Let $n\in \mathbb N \setminus \{0\}$.
For every $\eta>0$, there is $\gamma >0 $ with the following property.
Let $w: \R^2 \supset B_{2r} \to \Iq(\R^n)$ be Dir-minimizing 
in every $\Omega'\subset\subset B_{2r}$ such that
\begin{equation}\label{e.no decay}
\int_{B_{(1+\lambda) r}} \cG \big(Dw, Q \a{D (\etaa \circ w) (0)}\big)^2 \geq 2^{\delta_1-4} \D (w, B_{2r})\, .
\end{equation}
Then, if we set $\tilde{w} = \sum_i \a{w_i - \etaa\circ w}$, we have
\begin{align}
\gamma \, \D (w, B_{(1+\lambda) r}) & \leq \int_{B_s (q)} \cG (Dw, Q \a{D (\etaa \circ w)})^2\nonumber\\
& \leq \frac{1}{\gamma \, r^2} \int_{B_{s}(q)} \cG (w, Q \a{\etaa \circ w})^2  \quad\forall \;B_s(q) \subset B_{2\,r} \;\text{with }\, s\geq \eta \,r\, .\label{e.harm split 2}
\end{align}
\end{proposition}

\subsection{Splitting before tilting: Proof of Proposition \ref{p:splitting}}\label{ss:splitting}
As customary we use the convention that constants denoted by $C$ depend upon all the parameters but $\eps_2$, whereas constants denoted by $C_0$ depend only upon $n,\bar{n}$ and $Q$.

Given $L\in \sW^j_e$, let us consider
its ancestors $H\in \sS^{j-1}$ and $J\in \sS^{j-6}$, which exists thanks to Proposition \ref{p:whitney}. 
Set $\ell = \ell(L)$, $\pi = \pi_H$ and $\bC := \bC_{8r_J} (p_J, \pi)$,  and
let $f:B_{8r_J} (p_{HJ}, \pi)\to \Iq (\pi^\perp)$ be the $\pi$-approximation of Definition \ref{d:pi-approximations},
which is the result of \cite[Theorem 1.5]{DSS2} applied to $T_J$ in $\bC_{32r_J} (p_J, \pi)$
(recall that Proposition~\ref{p:tilting opt} ensures the applicability of \cite[Theorem 1.5]{DSS2} in the latter cylinder).

The following are simple consequences of Proposition~\ref{p:tilting opt} and Proposition~\ref{p:app}:
\begin{gather}
E := \bE (T_{J}, \bC_{32 r_J} (p_{J}, \pi)) \leq
% C \bmo \,\ell (J)^{2-2\delta_2}\leq 
C \bmo\, \d (L)^{2\gamma_0-2+2\delta_1} \,\ell^{2-2\delta_1}\, ,\label{e:eccesso J}\\
\bh (T_J, \bC, \pi_{H}) 
%\leq C\bmo^{\sfrac{1}{2m}} \ell (J)^{1+\beta_2}
\leq C\,\bmo^{\sfrac{1}{4}}\,\d (L)^{\sfrac{\gamma_0}{2}-\beta_2} \ell^{1+\beta_2},\label{e:altezza J}\\
c\,C_e\,\bmo\,\d (L)^{2\gamma_0-2+2\delta_1} \ell^{2-2\delta_1} \leq E,\label{e:eccesso alto}
\end{gather}
where \eqref{e:eccesso alto} follows from $\bB_{L} \subset \bC$, $L\in \sW_e$ and $\sfrac{r_L}{r_J}=2^{-6}$. 
In particular the positive constants $c$ and $C$ do not depend on $\eps_{2}$.
We divide the proof of Proposition~\ref{p:splitting} in three steps.

\medskip

{\bf Step 1: decay estimate for $f$.}
Let $2\rho:= 64 r_H - \underline{C} \bmo^{\sfrac{1}{4}} \d (L)^{\sfrac{\gamma_0}{2}-\beta_2} \ell^{1+\beta_2}$:
since $p_{H} \in \supp (T_{J})$, it follows from \eqref{e:altezza J} that, upon having chosen $\underline{C}$ appropriately,
$\supp (T_J) \cap \bC_{2\rho} (p_{H}, \pi)\subset \supp (T_H)\cap \bB_{H}\subset \bC$. Observe in particular that $\underline{C}$
does not depend on $\eps_2$, although it depends upon the other parameters.
Thus, setting $B=B_{2\rho}(p_H, \pi)$, using the Taylor expansion in \cite[Corollary 3.3]{DS2} and the estimates
in \cite[Theorem 1.5]{DSS2},  we get
\begin{align}
\D (B, f) &\leq 2 |B|\, \bE (T_J, \bC_{2\rho} (p_H, \pi)) +
C \bmo^{1+\beta_0}\,\d (L)^{(1+\beta_0)(2\gamma_0-2+2\delta_1)} \ell^{2+(1+\beta_0)(2-2\delta_1)}\nonumber\\
& \leq 2 \omega_2 (2\rho)^2 \bE (T_H, \bB_{H}) + C\bmo^{1+\beta_0}\d (L)^{(1+\beta_0)(2\gamma_0-2+2\delta_1)} \ell^{2+(1+\beta_0)(2-2\delta_1)}\, .\label{e:Dir-Ex}
\end{align}
%Consider next the cylinder $\bC_{64 r_L} (p_{L}, \pi_{H})$. 
Recall that $|p_H-p_{HL}|\leq |p_{H}-p_{L}| \leq C_0 \ell (H)$, where $C_0$ is a geometric constant (cf.~Proposition \ref{p:tilting opt}),
and set $\sigma:= 64 r_L+ C \ell (H) =  32 r_H + C \ell (H)$.
If $\lambda$ is as in \eqref{e:lambda} and $M_0$ is sufficiently
large (thus fixing a lower bound for 
$M_0$ which depends only on $\delta_1$) we reach
\[
\sigma \leq \left(\frac{1}{2} +\frac{\lambda}{4}\right) \,64\, r_H
\leq \left(1+\frac{\lambda}{2}\right) \rho + \bar{C} \bmo^{\sfrac14}\,\d (L)^{\sfrac{\gamma_0}{2}-\beta_2}\, \ell^{1+\beta_2}\, .
\]
In particular, choosing $\eps_2$ sufficiently small we get $\sigma \leq (1+\lambda) \rho$ and $\bB_{L}\subset
\bC_{(1+\lambda)\rho}(p_{L},\pi)=:\bC'$. Define $B'=B_{(1+\lambda)\rho}(p_H,\pi)$. Set
$A := \mint_{B'} D (\etaa \circ f)$,
$\bar{A}: \pi_{H}\to \pi_{H}^\perp$ the linear map $x\mapsto A\cdot x$
and $\bar \pi$ for the plane corresponding to $\bG_{\bar{A}}$.
Using \cite[Theorem 3.5]{DS2} and \cite[Theorem 5.2]{DSS2}, we estimate
\begin{align}
& {\textstyle{\frac{1}{2}}} \int_{B'} \cG (Df, Q \a{A})^2 \geq  |B'| \, \bE (T_{J}, \bC', \bar \pi) -  C\bmo^{1+\beta_0}\d (L)^{(1+\beta_0)(2\gamma_0-2+2\delta_1)} \ell^{2+(1+\beta_0)(2-2\delta_1)}\nonumber\\
\geq\; & |B'| \bE (T_J, \bB_{L}, \bar \pi) -  C\bmo^{1+\beta_0}\d (L)^{(1+\beta_0)(2\gamma_0-2+2\delta_1)} \ell^{2+(1+\beta_0)(2-2\delta_1)}\nonumber\\
\geq\; & \omega_2 ((1+\lambda)\rho)^2 \bE (T_L, \bB_{L})
 -  C\bmo^{1+\beta_0}\d (L)^{(1+\beta_0)(2\gamma_0-2+2\delta_1)} \ell^{2+(1+\beta_0)(2-2\delta_1)}. \label{e:da_sotto}
\end{align}
Next, considering that $\bB_{H}\supset
\bB_{L}$ and that, by $L\in \sW^j_e$,
\[
\bE(T_{L}, \bB_{L}) \geq C_e \bmo\,\d (L)^{2\gamma_0-2+2\delta_1} \ell^{2-2\delta_1},
\]
we conclude from \eqref{e:Dir-Ex} and \eqref{e:da_sotto} that 
\begin{align}
\D (B, f) &\leq 2 \omega_2 (2\rho)^2 (1+ C \bmo^{\beta_0}) \bE (T_{H}, \bB_{H}) \, .\label{e:Dir-Ex-bis}\\
\int_{B'} \cG (Df, Q \a{A})^2&\geq 2 \omega_2 ((1+\lambda)\rho)^2 (1- C \bmo^{\beta_0})
\bE (T_{L}, \bB_{L})\, .\label{e:da_sotto_bis}
\end{align}

\medskip

{\bf Step 2: harmonic approximation.} From now on, to simplify our notation,
we use $B_s (y)$ in place of $B_s (y, \pi)$ (and recall that $\pi = \pi_H$). Consistently with \cite{DSS1,DSS2, DS3} we introduce the parameter $\bOmega$, which equals 
\begin{itemize} 
\item $\bA = \|A_\Sigma\|_{C^0}$ in case (a) of Definition \ref{d:semicalibrated};
\item $\max \{\|d\omega\|_{C^0}, \|A_\Sigma\|_{C^0}\}$ in case (b);
\item $C_0 R^{-1}$ in case (c).
\end{itemize}
From \eqref{e:eccesso alto} we infer that, for any $\eps_{32}>0$, if $\bar{r}$ is chosen sufficiently small, we have
\begin{equation}\label{e:problematica}
8 r_J\, \bOmega \leq C \ell (L) \bmo^{\sfrac{1}{2}} \leq 
\varepsilon_{32}\,C_e^{\sfrac12} \bmo^{\sfrac{1}{2}}\, \d (L)^{\gamma_0-1+\delta_1} \ell (L)^{1-\delta_1}\leq \eps_{32} E^{\sfrac12},
\end{equation}
because $\ell (L) \leq \d (L) \leq \bar{r}$.
Therefore, for every positive $\bar{\eta}$,
we can apply \cite[Theorem 1.6]{DS3} (in case (a) of Definition \ref{d:semicalibrated}) and \cite[Theorem 3.1]{DSS2} (in the cases (b) and (c) of Definition \ref{d:semicalibrated}) to the cylinder $\bC$
and achieve a map $w: B_{8r_J} (p_{HJ}, \pi)
\to \Iq (\pi^\perp)$ of the form $w=(u,\Psi(y,u))$ (in fact $w=u$ in case (b) of definition \ref{d:semicalibrated}) for a  $\D$-minimizer $u$ and such that
\begin{gather}
(8\,r_J)^{-2} \int_{B_{8r_J} (p_{HJ})} \cG (f,w)^2 + \int_{B_{8r_J} (p_{HJ})} (|Df|-|Dw|)^2 \leq \bar{\eta} \,E \,(8\,r_J)^2,\label{e:armonica_vicina}\\
\int_{B_{8r_J} (p_{HJ})} |D (\etaa \circ f) - D (\etaa \circ w)|^2 \leq \bar{\eta}\, E \, (8\,r_J)^2\, .
\label{e:armonica_vicina_2}
\end{gather}
In the cases (a) and (c) of Definition \ref{d:semicalibrated},
by the chain rule we have $D(\Psi(y,u(y)))=\sum_j\a{D_x\Psi(y,u_j(y))+D_v\Psi(y, u_j(y))\cdot Du_j(y)}$, so that
\begin{equation}\label{e:aggiuntiva_20}
 \int_{B_{(1+\lambda)\rho}(p_H)}|D(\Psi(y,u))|^2\leq C_0 \bmo \int_{B_{(1+\lambda)\rho}(p_H)}|Du|^2+C_0 \bmo \rho^{4}, 
\end{equation}
where $C_0$ is a geometric constant.
Consider now $\tilde A :=\mint_{B'} D(\etaa\circ w)$, and observe that, since 
$D\, (\etaa\circ u)=\etaa\circ Du$ is harmonic, we have $A':= D\, (\etaa\circ u) (p_H)=\mint_{B'} \etaa\circ Du$, where $B'= B_{(1+\lambda) \rho} (p_H)$.  
We can use \eqref{e:armonica_vicina}, \eqref{e:armonica_vicina_2} and \eqref{e:aggiuntiva_20}, together with \eqref{e:da_sotto_bis} to infer, for $\eps_2$ small enough,
\begin{align}
& \int_{B'} \cG \big(Du, Q \a{\etaa\circ Du (p_H)}\big)^2\notag\\
= & \int_{B'}\cG \big(Du, Q\a{A'})^2\big)^2
\stackrel{\eqref{e:aggiuntiva_20}}{\geq} \int_{B'} \cG (Dw, Q \llbracket\tilde A\rrbracket)^2 - C_0 \bmo \rho^4\notag\\ 
\stackrel{\eqref{e:armonica_vicina}\&\eqref{e:armonica_vicina_2}}{\geq} & \int_{B'}\cG \big(Df, Q {\a{A}}\big)^2- C_0 \bmo \rho^{4}-C_0 \bar{\eta} E\rho^2\nonumber\\
\stackrel{\eqref{e:da_sotto_bis}}{\geq} & 2 \omega_2 ((1+\lambda)\rho)^2 (1- C \bmo^{\beta_0})\bE(T_{L},\bB_{L})- C_0 \bmo \rho^{4}- C_0 \bar{\eta} E\rho^2.\label{e:non_decay_10}
\end{align}
Analogously, using  \eqref{e:Dir-Ex-bis} and \eqref{e:armonica_vicina}, we easily deduce
\begin{equation}\label{e:armonica_dall_alto}
\int_{B_{2\rho}(p_H)}|Du|^2\leq 2 \omega_2 (2\rho)^2 (1+ \bmo^{\beta_0}) \bE (T_{H}, \bB_{H})+C_0 \bmo \rho^{4}+ C_0 \bar{\eta} E\rho^2
\end{equation}
Now recall that, since $\d (L)= \d (H) = \d (J)$, and $L\in \sW_e$,
\begin{align*}
 \bE(T_{L},\bB_{L})
 &\geq C_e\bmo \d (L)^{2\gamma_0-2+2\delta_1}\ell(L)^{2-2\delta_1} \geq 2^{2 \delta_1-2}\bE(T_H, \bB_{H})\, ,
 \end{align*}
and combining this with \eqref{e:armonica_dall_alto} and \eqref{e:non_decay_10} we achieve
\begin{equation}\label{e:concl_11}
\int_{B'} \cG \big(Du, Q \a{D(\etaa \circ u) (p_H)}\big)^2
\geq  (2^{2\delta_1-4}-C \bmo^{\beta_0})\,\int_{B_{2\rho}(p_H)}|Du|^2  - C_0 \bmo \rho^{4}- C_0 \bar{\eta} E\rho^2  \, .
\end{equation}
To estimate the last two errors in terms of the energy of $u$ we use again we use \eqref{e:non_decay_10} to estimate
\[
E\rho^2\leq C_0 \rho^2 \,\bE(T_L,B_L) \leq C_0 \int_{B'} |Du|^2+C_0\,\bmo\,\rho^4+C_0 \bar{\eta} E\rho^2 
\]
so that, for $\bar \eta\leq \sfrac {1}{2C_0}$ we have
\begin{equation}\label{e:err_arm_1}
E\rho^2\leq C_0 \,\rho^2 \bE(T_L,B_L) \leq C_0 \int_{B'} |Du|^2+C_0\,\bmo\,\rho^4\,.
\end{equation}
Next, using once again, $L\in \sW_e$ and this last inequality,
\begin{align*}
 \bmo \rho^4 
&\leq \frac{C_0\rho^2}{C_e} \d (L)^{2-2\gamma_0-2\delta_1} \bE(T_{L},\bB_{L})
%\stackrel{\eqref{e:da_sotto_bis}}{\leq}
%\frac{C_0}{C_e} \int_{B'}|Df|^2 
%\leq  \frac{C_0}{C_e} \int_{B'}|Du|^2 + \frac{C_0}{C_e} \bmo \rho^{4} + \frac{C_0}{C_e} \bar{\eta} E\rho^2
\leq \frac{C_0}{C_e} \int_{B'}|Du|^2 + \frac{C_0}{C_e} \bmo \rho^{4}\, ,
\end{align*}
which for $C_e$ bigger than a geometric constant implies
\begin{equation}\label{e:err_arm_2}
\bmo \rho^4 \leq \frac{C_0}{C_e} \int_{B'}|Du|^2\,.
\end{equation}
We can therefore combine \eqref{e:concl_11} with \eqref{e:err_arm_1} and \eqref{e:err_arm_2} to achieve
\begin{align}
\int_{B_{(1+\lambda) \rho} (p_H)} \cG \big(Du, Q \a{D(\etaa \circ u) (p_H)}\big)^2
&\geq  \Bigl(2^{2\delta_1-4}-\frac{C_0}{C_e}-C \bmo^{\beta_0}- C_0\bar{\eta}\Bigr)\int_{B_{2\rho}(p_H)}|Du|^2\, .
\end{align}
It is crucial that the constant $C$, although depending upon $\beta_2, \delta_2, M_0, N_0, C_e$ and $C_h$, does not depend on $\eta$ and $\eps_2$, whereas $C_0$ depends only upon $Q, \bar{n}$ and $n$. 
So, if $C_e$ is chosen sufficiently large, depending only upon $\lambda$ (and hence upon $\delta_1$), we can require that
$2^{2\delta_1 -4} - \frac{C_0}{C_e} \geq 2^{3\delta_1/4 -4}$. We then require $\bar{\eta}$ and $\eps_2$ to be sufficiently small
so that  $2^{3\delta_1/4 -4} - C\,\bmo^{\beta_0}-C \bar{\eta} \geq 2^{\delta_1-4}$.

We can now apply Proposition \ref{p:harmonic split} to $u$ and conclude
\begin{align}
\hat{C}^{-1} \int_{B_{(1+\lambda) \rho} (p_H)} |Du|^2 \leq \int_{B_{\ell/8} (q)} \cG (Du, Q \a{D (\etaa\circ u)})^2
\leq \hat{C} \ell^{-2} \int_{B_{\ell/8} (q)} \cG (u, Q \a{\etaa\circ u})^2\, ,\label{e:aggiuntiva_99}
\end{align}
for any ball $B_{\ell/8} (q) = B_{\ell/8} (q, \pi) \subset B_{2\rho} (p_{H})$,
where $\hat{C}$ depends upon $\delta_1$ and $M_0$. 
In particular, being these constants are independent of
$\eps_2$ and $C_e$.

Next recall that $L\in \sW_e$, therefore
\begin{equation}\label{e:aggiuntiva_100}
\bmo \d (L)^{2\gamma_0-2+2\delta_1}\, \ell^{4-2\delta_1}\leq \tilde C \ell^2 \,\bE (T, \bB_L)\, .
\end{equation}
On the other hand by \eqref{e:non_decay_10}, \eqref{e:err_arm_1} and \eqref{e:err_arm_2}, a suitable choice of $\eps_2$, $\eta$ and $C_e$ yields
\begin{equation}\label{e:aggiuntiva_101}
\ell^2 \bE (T, \bB_L) \leq \int_{B_{(1+\lambda) \rho} (p_H)} |Du|^2 \leq C \int_{B_{\ell/8} (q)} \cG (Du, Q \a{D (\etaa\circ u)})^2\, .
\end{equation}
In turn, \eqref{e:armonica_vicina_2} and the definition of $w$ implies
\[
\int_{B_{\ell/8} (q)} \cG (Du, Q \a{D (\etaa\circ u)})^2 \leq \int_{B_{\ell/8} (q)} \cG (Df, Q \a{D (\etaa\circ f)})^2 + \eta E
\]
and taking into account \eqref{e:err_arm_2} and \eqref{e:aggiuntiva_99} we can once again assume that, for an appropriate choice of the parameters $\eps_2$, $\eta$ and $C_e$ we achieve
\[
\int_{B_{\ell/8} (q)} \cG (Du, Q \a{D (\etaa\circ u)})^2 \leq 2 \int_{B_{\ell/8} (q)} \cG (Df, Q \a{D (\etaa\circ f)})^2\, .
\]
Finally, using again \eqref{e:armonica_vicina}, \eqref{e:err_arm_2} and \eqref{e:aggiuntiva_99} we conclude in analogous way that
\[
\int_{B_{\ell/8} (q)} \cG (u, Q \a{\etaa\circ u})^2 \leq 2\int_{B_{\ell/8} (q)} \cG (f, Q \a{\etaa \circ f})^2\, .
\]
Summarizing all the conclusions reached so far, we have
\begin{align}
\bmo \d (L)^{2\gamma_0-2+2\delta_1}\, \ell^{4-2\delta_1} &\leq \tilde C \ell^2 \,\bE (T, \bB_L) \leq \bar{C} \int_{B_{\ell/8} (q)} 
\cG (Df, Q\a{D (\etaa \circ f)})^2\notag\\
&\leq \check{C} \ell^{-2} \int_{B_{\ell/8} (q)} \cG (f, Q \a{\etaa \circ f})^2,\label{e:stima dritta}
\end{align}
where $\tilde C$, $\bar{C}$ and $\check{C}$ are constants which depend upon $\delta_1$, $M_0$ and $C_e$, but not
on $\eps_2$.

\medskip

{\bf Step 3: Estimate for the $\cM$-normal approximation.} We next complete the proof showing \eqref{e:split_1} and \eqref{e:split_2}. Now, consider any ball $B_{\ell/4} (q, \pi_0)$ with $q\in \gira$ and $\dist (L, q)\leq 4\sqrt{2} \,\ell$
and  let $\Omega:= \Phii (B_{\ell/4} (q, \pi_0))$. Recall that $\pi=\pi_H$ and by a slight abuse of notation let $\p_\pi$ be the projection onto $p_H+\pi$. 
Observe that $\p_{\pi} (\Omega)$ must contain a ball
$B_{\ell/8} (q', \pi)$, because of $\|D\phii\|_{C^0} \leq C\bmo^{\sfrac{1}{2}}$ and $|\pi_0-\pi|\leq C\bmo^{\sfrac{1}{2}}$, and in turn
it must be contained in $B_{8r_J} (p_{HJ}, \pi)$. 

Applying \cite[Lemma B.1]{DS4} let $\phii': B_{8r_J} (p_{HJ}, \pi)\to\pi^\perp$
be such that $\bG_{\phii'} =\a{\cM}\res \bC_{8r_J} (p_J, \pi)$ and let $\Phii'(z) = (z, \phii'(z))$. Since $D (\etaa \circ f) (z) =
\etaa \circ Df (z)$ for a.e. $z$, we obviously have
\begin{equation}
\int_{B_{\ell/8} (q', \pi)} \cG (Df, Q \a{D(\etaa \circ f)})^2 \leq  
C_0 \int_{B_{\ell/8} (q', \pi)} \cG (Df, Q \a{D\phii'})^2\, .
\end{equation}
Let now $\vec{\bG_f}$ be the orienting tangent $2$-vector to $\bG_f$ and $\tau$
 the one to $\cM$. For a.e. $z$ we have the inequality
\[
C_0 \sum_j |\vec{\bG}_f (z, f_j (z)) - \vec{\tau} (z, \phii' (z))|^2 \geq \cG (Df (z), Q \a{D\phii' (z)})^2\, ,  
\]
for some geometric constant $C_0$, because $|\vec{\bG}_f (z, f_j (z)) - \vec{\tau} (z, \phii' (z))|\leq \bmo^{\beta_0}$. Therefore,
using \cite[Theorem 5.2]{DSS2}, 
\begin{align}
\mint_{B_{\ell/8} (q', \pi)} &\cG (Df, Q \a{D\phii'})^2 \leq
C \mint_{\bC_{\ell/8} (q', \pi)} |\vec{\bG}_f (z) - \vec\tau (\Phii' (\p_{\pi_{H}} (z))|^2 d\|\bG_f\| (z)\nonumber\\
&\leq C \mint_{\bC_{\ell/8} (q', \pi)} |\vec{T_{L}} (z) - \vec\tau (\Phii' (\p_{\pi_{H}} (z))|^2 d\|T_{L}\| (z)\nonumber\\
&+ C \bmo^{1+\beta_0}\d (L)^{(1+\beta_0)(2\gamma_0-2+2\delta_2)} \ell^{2+(2-2\delta_2)(1+\beta_0)}\, .\label{e:quasi finale}
\end{align}
Now, thanks to the height bound \eqref{e:ht_whitney} and to the fact that 
$|\vec{\tau} - \pi_{H}|\leq C \bmo^{\sfrac{1}{2}} \d (L)^{\sfrac{\gamma_0}{2}-1} \ell$
in the cylinder $\hat\bC = \bC_{\ell/8} (q', \pi_{H})$ (recall \eqref{e:aggiuntiva_11}), we have the inequality
\[
|\p (z) - \Phii' (\p_{\pi} (z))|
\leq C \bmo^{\sfrac{1}{4} + \sfrac{1}{2}} \d (L)^{\gamma_0-\beta_2}\ell^{2+\beta_2}  \qquad
\forall z\in \supp (T)\cap \hat\bC\, .
\]
Using the estimate $|D^2 \phii' (\p_{\pi} (z))| \leq C \bmo^{\sfrac{1}{2}} \d (L)^{\sfrac{\gamma_0}{2}-1}$ (which is valid
for any $z\in \supp (T)\cap \hat\bC$ by \eqref{e:stima_C3} and \cite[(B.3)]{DS4}) 
we then easily conclude from \eqref{e:quasi finale} that
\begin{align*}
 &\mint_{B_{\ell/8} (q', \pi)} \cG (Df, Q \a{D\phii'})^2\\
\leq &C \mint_{\hat\bC} 
|\vec{T}_{L} (z) - \vec\tau (\p (z))|^2 d\|T_L\| (z)
+ C \bmo^{1+\beta_0}\d (L)^{2\gamma_0-2-2\beta_2} \ell^{2+ 2\beta_2}\, \\
\leq & C \mint_{\p^{-1} (\Omega)} |\vec{\bT}_{F} (z) - \tau (\p (z))|^2 d\|\bT_{F}\| (z) + 
C \bmo^{1+\beta_0} \d (L)^{2\gamma_0-2+2\delta_1} \ell^{2-2\delta_1},
\end{align*}
where we used \eqref{e:err_regional}.

Since, on the region where we are interested, namely $\Omega$, we have the bounds $|DN|\leq C \bmo^{\beta_0} \d (L)^{\beta_0\gamma_0}$,
$|N|\leq C \bmo^{\sfrac{1}{4}} \d (L)^{\sfrac{\gamma_0}{2}-\beta_2} \ell^{1+\beta_2}$
and $\|A_{\cM}\|^2 \leq C \bmo \d (L)^{\gamma_0-2}$,
applying now  \cite[Proposition 3.4]{DS2} we conclude
\begin{align*}
\mint_{\p^{-1} (\Omega)}
|\vec{\bT}_F (x) - \tau (\p (x))|^2 d\|\bT_F\| (x)
\leq & (1+ C \bmo^{2\beta_0}\d (L)^{2\gamma_0\beta_0} ) \int_{\Omega} |DN|^2\\
&\quad + C \bmo^{1+\sfrac{1}{2}} \d (L)^{2\gamma_0-2-2\beta_2}\ell^{2+2\beta_2}\, .
\end{align*}
Thus, combining the latter estimate with \eqref{e:stima dritta} and \eqref{e:quasi finale} we achieve
\begin{equation}
\bmo \,\d (L)^{2\gamma_0-2+2\delta_1} \ell^{2-2\delta_2} \leq C (1 + C \bmo^{2\beta_0} \d (L)^{2\gamma_0\beta_0}) \mint_{\Omega} |DN|^2
+ C \bmo^{1+\beta_0} \d (L)^{2\gamma_0-2+2\delta_1}\ell^{2-2\delta_2}\, .
\end{equation}
Since the constant $C$ might depend on $M_0, N_0, C_e$ and $C_h$ but not on $\eps_2$, we conclude that
for a sufficiently small $\eps_2$ we have
\begin{equation}
\bmo\d (L)^{2\gamma_0-2+2\delta_1} \ell^{2-2\delta_1} \leq C \mint_\Omega |DN|^2\, .
\end{equation}
But $\bE (T_{L}, \bB_{L}) \leq C \bmo\,\d (L)^{2\gamma_0-2+2\delta_1} \ell^{2-2\delta_2}$ and thus \eqref{e:split_1}  follows.

\medskip

We finally show \eqref{e:split_2}. Observe that $\p^{-1} (\Omega) \cap \supp (T) \supset \bC_{\ell/8} (q', \pi)\cap \supp (T_L)$
and, for an appropriate geometric constant $C_0$, 
$\Omega$ cannot intersect a Whitney region $\cL'$
corresponding to an $L'$ with $\ell (L') \geq C_0 \ell (L)$ or $\d (L')\geq 2 \d (L)$.
In particular, Theorem~\ref{t:approx} implies that
\begin{equation}\label{e:masses}
\|\bT_F - T_L\| (\p^{-1} (\Omega)) + \|\bT_F - \bG_f\| (\p^{-1} (\Omega)) \leq  C \bmo^{1+\beta_0} \,\d (L)^{(1+\beta_0)(2\gamma_0-2+2\delta_1)}\,\ell^{2+(1+\beta_0)(2-2\delta_1)}\, .
\end{equation}
Let now $F'$ be the map such that 
$\bT_{F'} \res (\p^{-1} (\Omega)) = \bG_f \res (\p^{-1} (\Omega))$ and let $N'$ be the corresponding
normal part, i.e. $F' (x) = \sum_i \a{x+N'_i (x)}$ (the existence of $F'$ is guaranteed by \cite[Theorem 5.1]{DS2}).
The region over which $F$ and $F'$ differ is contained in the projection
onto $\Omega$ of $(\im (F) \setminus \supp (T))
\cup (\im(F') \setminus \supp (T))$ and
therefore its $\cH^m$ measure is bounded as in \eqref{e:masses}.
Recalling the bound on $\|N\|_{C^0}$ given by \eqref{e:Lip_regional} and that 
\[
\cG (f, Q \a{\phii}) \leq C \bh (T_L, \bB_L)\, ,
\]
we easily conclude $|N|+|N'| \leq C \bmo^{\sfrac14} \d (L)^{\sfrac{\gamma_0}{2}-\beta_2} \ell ^{1+\beta_2}$, which in turn implies
\begin{equation}\label{e:ultima1}
\int_{\Omega} |N|^2 \geq \int_{\Omega} |N'|^2 - C \bmo^{1+\sfrac{1}{4}+\beta_0} \d (L)^{(1+\beta_0)(2\gamma_0-2+2\delta_1)+ \gamma_0 - 2 \beta_2} \ell^{4+2\beta_2+(2-2\delta_1)(1+\beta_0)}\, .
\end{equation}
On the other hand, 
applying \cite[Theorem 5.1 (5.3)]{DS2}, we conclude
\[
|N' (\Phii' (z))| \geq \frac{1}{2\sqrt{Q}}\, \cG (f(z), Q \a{\phii' (z)}) \geq 
\frac{1}{4\sqrt{Q}}\, \cG (f(z), Q\a{\etaa\circ f (z)})\,,
\]
which in turn implies
\begin{align}
\bmo \, \d (L)^{2\gamma_0-2+2\delta_1}\, \ell^{2-2\delta_2} &\stackrel{\eqref{e:stima dritta}}{\leq} C \ell^{-2} \int_{B_{\ell/8} (q', \pi)}  \cG (f, Q \a{\etaa \circ f})^2
\leq C \ell^{-2} \int_\Omega |N'|^2\label{e:ultima2}\, .
\end{align}
For $\eps_2$ sufficiently small, \eqref{e:ultima1} and \eqref{e:ultima2} lead to the second inequality of \eqref{e:split_2}, while
%\begin{equation}
%\bmo \ell^{m+2-2\delta_2} \leq C \ell^{-2} \int_\Omega |N|^2
%\end{equation}
%and hence 
the first one comes from Theorem~\ref{t:approx} and $\bE (T, \bB_L) \geq C_e \bmo \,\d (L)^{2\gamma_0-2+2\delta_1}\ell^{2-2\delta_2}$.

\section{Proof of Theorem \ref{t:cm_final}}

We are now ready to define the relevant objects of Theorem \ref{t:cm_final}. The center manifold $\cM$ is given by Theorem \ref{t:cm}: the fact that $\cM$ is a $b$-separated admissible $\bar Q$-branching is a simple consequence of the estimates in Theorem \ref{t:cm}. We then apply Proposition \ref{p:conformal} to find the map $\Psii$ which is a conformal parametrization of $\cM$ in a neighborhood of $0$ and, after a suitable scaling, we assume that it is defined on $\gira_{\bar Q, 2}$. 
Secondly we consider the normal approximation $N$ of the current $T$ on $\cM$ constructed in Theorem \ref{t:approx}. The relation $\bar{Q} Q = \Theta (T, 0)$ is obvious from the construction. Again, after scaling, we assume that:
\begin{itemize}
\item The radius $r_0$ of Theorem \ref{t:cm_final} is $4$;
\item $\Psii (\gira) \subset \bC_3 (0)$;
\end{itemize} 
Rather than call the rescaled current $S$, as it is done in the statement of Theorem \ref{t:cm_final}, we keep denoting it by $T$. 

The maps $\NN$ and $\FF$ are then defined as 
\begin{align}
\NN(z,w) :=& N(\Psii(z,w)) = \sum_i \a{N_i (\Psii (z,w))}\\
\FF (z,w) :=& \sum_i \a{\Psii (z,w) + \NN\, _i (z,w)} = \sum_i \a{\Psii (z,w) + N_i (\Psii (z,w))}\, .
\end{align}
By the estimate \eqref{e:stima_C3} it follows immediately that 
\[
|A_\cM (\zeta, \xi)| + |\zeta| |D_{\cM}A_{\cM} (\zeta,\xi)|\leq C \bmo^{\sfrac{1}{2}} |\zeta|^{\gamma_0 -1}
\]
 at any point $p=(\zeta,\xi)\in \cM$ with $\zeta\in \R^2\setminus 0$. On the other hand by \eqref{e:conformal2}, if we set $(\zeta, \xi) := \Psii (z,w)$, then we have
\begin{equation}\label{e:sandwich}
|z| - C \bmo^{\sfrac{1}{4}} |z|^{1+\gamma_0} \leq |\zeta| \leq |z| + C \bmo^{\sfrac{1}{2}} |z|^{1+\gamma_0}
\end{equation}
and thus the estimates in (i) follow. By construction $\NN\,_i (z,w) = N_i (\Psii (z,w))$ is orthogonal to $T_{\Psii (z,w)} {\cM}$, which shows (ii). 

The fact that $T$ is contained in a horned neighborhood of $\cM$ where the projection $\p$ is well defined is a consequence of Corollary \ref{c:cover}. 
Moreover, by \eqref{e:sandwich} we can assume $\Psii (B_r (0)) \subset \bC_{2r}$ (this is true for a sufficiently small $r$ and hence, after scaling, we can assume it holds for any $r\leq 1$). On the other hand, consider a cube $L$ of $\sW$ which intersects $B_{3r/2} (0)$. By construction its sidelength is necessarily smaller than $r$. Thus \eqref{e:Ndecay} is a simple consequence of \eqref{e:Lip_regional}, Corollary \ref{c:cover}(iii) and Corollary \ref{c:cover}(iv). 

We are left to show the three estimates claimed in point (iv) of Theorem \ref{t:cm_final}: the rest of the section is devoted to this task.

\subsection{The special covering} First of all consider the set $\Psii (B_r (0))$ and let $\mathcal{B}_r\subset \gira$ be defined by
\begin{equation}
\mathcal{B}_r := \{(z,w)\in \gira: \Phii (z,w) \in \Psii (B_r (0))\}\, .
\end{equation}
Observe that, by the estimates on $\Psii$, the following two facts are obvious for $r$ small:
\begin{itemize}
\item[(g1)] $\mathcal{B}_r$ is star-shaped with respect to the origin, more precisely if $q=(z,w)\in \partial \mathcal{B}_r$, then the geodesic segment $\sigma$ in $\gira$ joining $(0,0)$ and $q$ is contained in $\mathcal{B}_r$;
\item[(g2)] If $\bar{q}$ denotes the point on $\sigma$ at distance $\frac{r}{4}$ from the origin, the disk $B_{r/4} (\bar{q})$ is contained in $\mathcal{B}_r$.
\end{itemize}
We next select an (at most countable) family of triples
$\{(L_j, B_j,U_j)\}_{j \in \N}$ of subsets of $\gira_{\bar{Q}}$
with the following properties:
\begin{itemize}
\item[(c1)] The $L_j$'s are distinct cubes
of the Whitney decomposition with $L_j \in \sW_e \cup \sW_h$
and $L_j \subset \bar B_{2r+ 6\ell(L_j)}$;
\item[(c2)] $B_j = B_{\sfrac{\ell(L_j)}{4}}(z_j,w_j) \subset \mathcal{B}_r$
are disjoint balls such that $|z_{L_j}-z_j| \leq 7\, \ell(L_j)$;
\item[(c3)] $U_j$ is the union of an at most countable family of cubes $\bar \sW (L_j)\subset \sW$ where $H \subset B_{30 \ell ({L_j})}(z_{L_j},w_{L_j})$ for every $H\in \bar \sW (L_j)$ and $\cup_j \bar \sW (L_j)$ consists of all cubes in $\sW$ which intersect $\mathcal{B}_r$; in particular 
\begin{equation}\label{e:copre}
\mathcal{B}_r \subset \bGam \cup \bigcup_j U_j\, .
\end{equation}
\end{itemize}
To this aim we start by selecting all the cubes $L \in \sW_e \cup \sW_h$ such that either $L \cap \cB_r \neq \emptyset$ or there exists $H \in \sW_n$ in the domain of influence of $L$ with $H \cap \cB_r \neq \emptyset$, and we denote the collection
of such cubes by $\sW(r)$. Observe that,  $\ell (L)\leq C_0 2^{-N_0} r$ and thus, provided $N_0$ is chosen sufficiently large, we can assume that the ratio $\frac{\ell (L)}{r}$ is smaller than any fixed geometric constant. Moreover, by Corollary \ref{c:domains}, it is obvious that $L\subset B_{2r + 6 \ell (L)}$. 

The triples above are then chosen according to the following procedure:
\begin{itemize}
\item We start selecting recursively $\{L_j\}\subset \sW (r)$. $L_0$ is a cube with the largest sidelength in $\sW (r)$. Having chosen $\{L_0, \ldots, L_j\}$ we select $L_{j+1}$ as a cube with the largest sidelength among those $L\in \sW(r)$ such that $B_{15 \ell (L)} (z_L, w_L)\cap B_{15 \ell (L_i)} (z_{L_i}, w_{L_i}) = \emptyset $ for all $i\leq j$.
\item For every $L_j$ we use the geometric properties (g1) and (g2) to choose a ball $B_j$ as in (c2): for instance we consider $z_j :=\frac{z_{L_j}}{|z_{L_j}|}\big(|z_{L_j}| - \frac{7\sqrt{2}}{2}\,\ell_{L_j} \big)$ and let $(z_j,w_j)$ be the unique point of $\gira$ that belongs to the connected component of $\gira \cap ( B_{L_j}\times  \C) $ that contains $(z_{L_j},w_{L_j})$. The $B_j$'s are disjoint because they are contained in $B_{15 \ell (L_j)} (z_{L_j}, w_{L_j})$;
\item For what concerns $U_j$, we need to define $\bar \sW (L_j)$. First of all $L_j \in \bar \sW (L_j)$. We then consider any other $H\in \sW$ such that $H\cap \cB_r \neq \emptyset$ and we assign it to one (and only one) family $\bar \sW (L_j)$ according to the following rules:
\begin{itemize}
\item[(a)] If $H\in \sW_e \cap \sW_h$, then $H\in \sW (r)$ and we select one $L_j$ with largest sidelength such that $B_{15 \ell (L_j)} (z_{L_j}, w_{L_j}) \cap B_{15 \ell (H)} (z_H, w_H) \neq \emptyset$;
\item[(b)] If $H\in \sW_n$, then $H$ belongs to the domain of influence $\sW_n (L)$ of some $L\in \sW (r)$; we then assign $H$ to the family $\bar \sW (L_j)$ which already contains $L$.
\end{itemize}
\end{itemize}

\subsection{Estimates on $\mathcal U_j$ and $\bLambda$} Let $\mathcal{U}_j = \Phii (U_j)$ and $\cB_j := \Phii (B_j)$ and set, for notational convenience, $\d_j:=\d(L_j)$ and $\ell_j:=\ell(L_j)$.
As a simple consequence of Theorem~\ref{t:approx} and of Corollary \ref{c:cover}(iii)we deduce the
following estimates for every $j\in \N$:
\begin{gather}
\int_{\cU_j} |\etaa\circ N| \leq C\bmo \,\d_j^{2\gamma_0-2 + 2\beta_0\gamma_0 - \beta_2}\,\ell_j^{5+\sfrac{\beta_2}{4}}+ C \bmo^{\sfrac12+\beta_0}\,\d_j^{\gamma_0-1}\,\ell_j^{1+\beta_2} \int_{U_j} |N|\label{e:media}\\
\int_{\mathcal{U}_j} |DN|^2 \leq C \bmo\, \d_j^{2\gamma_0-2+2\delta_1}\,\ell_j^{4-2\delta_1},\label{e:Dirichlet_sopra}\\
\|N\|_{C^0 ({\cU}_j)} + \sup_{p\in \supp (T) \cap \p^{-1} (\cU_j)} |p- \p (p)| \leq C \bmo^{\sfrac{1}{4}}\, d_j^{\sfrac{\gamma_0}{2}-\beta_2} \ell_j^{1+\beta_2},\label{e:N_sopra}\\
\Lip (N|_{\cU_j}) \leq C \left(\bmo \d_j^{\gamma_0} \ell_j^{\gamma_0}\right)^{\beta_0},\label{e:Lipschitz}\\
\|T -\bT_F\|(\p^{-1} (\cU_j)) \leq C \bmo^{1+\beta_0}\,\d_j^{(1+\beta_0)(2\gamma_0-2+2\delta_1)} \ell_j^{2+(1+\beta_0)(2-2\delta_1)}.\label{e:errori_massa}
\end{gather}
Indeed, observe that $\d(H)\leq\d_j \leq 2 \d(H)$ for every $H\in \bar \sW(L_j)$
and $\sum_{H\in \bar \sW (J_i)} \,\ell (H)^2 \leq C \ell_j^2$,
because all $H\in \bar \sW (J_i)$ are disjoint and contained in a ball of radius
comparable to $\ell_j$. This in turn implies that
$\sum_{H\in \bar \sW (J_j)} \ell (H)^{2+\varepsilon} \leq C \ell_j^{2+\varepsilon}$, because 
$\ell (H)\leq \ell_j$ for any $H\in \bar \sW (L)$, and \eqref{e:media} - \eqref{e:errori_massa} follows easily because the exponents $5+\beta_2/4$, $4-2\delta_1$ and $2 + (1+\beta_0) (2-2\delta_1)$ are all larger than $2$.

Next we claim the following inequality for every $t>0$, where $\eta (t)$ and $C(t)$ are suitable positive functions, 
\begin{gather}
\sup_j \left(\bmo \, \d_j\,\ell_j\right)^t\leq C(t)\,\bLambda^{\eta(t)}(r)\,, \label{e:key2}
\end{gather}
Indeed, using Propositions~\ref{p:separ} and \ref{p:splitting} we have
\begin{align}
C_e \,\bmo\, \d_j^{2\gamma_0-2+2\delta_1}\,\ell_j^{4-2\delta_1}&
\leq C \,\int_{\cB_j}|D N|^2\quad \textup{if }L_j\in \sW_e\,,\label{e:cubi_E}\\
C_h^2 \,\bmo^{\sfrac12}\, \d_j^{\gamma_0-2\beta_2}\,\ell_j^{4+2\beta_2}  &
\leq C \int_{\cB_j} |N|^2\quad \textup{if }L_j\in \sW_h\,. \label{e:cubi_H}
\end{align}
% and the disjointness of $\cB_j$
On the other hand, since the $\cB_j$ are disjoint and contained in $\cB_r = \Psii (B_r)$, 
\[
\sum_j \int_{\cB_j}|D N|^2 \leq \int_{\cB_r} |DN|^2 = \int_{B_r} |D\NN|^2
\]
by conformality of $\Psii$ and
\[
\sum_j \int_{\cB_j}|N|^2\leq \int_{\cB_r} |N|^2 \leq C \int_{B_r} |\NN|^2
\]
by the Lipschitz regularity of $\Psii$. 
Thus \eqref{e:key2} follows easily by suitably choosing $C(t)$ and $\eta(t)$.

Observe therefore that \eqref{e:Lip_N} is an obvious consequence of \eqref{e:key2}, \eqref{e:Lipschitz} and the uniform bound on $|D\Psii|$ given in Proposition \ref{p:conformal}. 

\subsection{Proof of \eqref{e:media_pesata}} First of all observe that, by the bounds on $\Psii$ of Proposition \ref{p:conformal},
\[
\int_{B_r}  |\zeta|^{\gamma_0 -1} |\etaa \circ \NN(\zeta,\xi)|  \leq C \int_{\cB_r} |z|^{\gamma_0-1} |\etaa \circ N (z, w)|\, .
\]
On the other hand, since $U_j\subset B_{30 \ell_j} (z_{L_j}, w_{L_j})$ and $\frac{\d_j}{2}\leq |z_{L_j}|\leq 2\d_j$, we compute
\[
\int_{\cB_r} \,|z|^{\gamma_0-1} |\etaa \circ N(z,w)|
\leq C\sum_j\d_j^{\gamma_0-1}\int_{\cU_j}|\etaa\circ N(z,w)|\, .
\]
Now considering that $\d_j^{3\gamma_0-3+2\beta_0\gamma_0 - \beta_2}\,\ell_j^{5+\sfrac{\beta_2}{4}}\leq \d_j^{3\gamma_0-2}\,\ell_j^{4+2\beta_2}$ (recall $2\beta_2\leq\beta_0\gamma_0$), we have
\begin{align*}
\int_{\cB_r} &\,|z|^{\gamma_0-1} |\etaa \circ N(z,w)|\notag\\
&\stackrel{\eqref{e:media}}{\leq} C \sum_{j\in \N} \Big( \bmo\,\d_j^{3\gamma_0-2}\,\ell_j^{4+2\beta_2}+ 
C \underbrace{\bmo^{1/2+\beta_0}\,\d_j^{\gamma_0-1}\,\ell_j^{1+\beta_2} \int_{\cU_j} \frac{|N|}{|z|^{1-\gamma_0}}}_{=:A}\Big).
\end{align*}
We treat the second term in the summand above via Young's inequality inequality;
\begin{align*}
A &\leq 2\,\left(\bmo^{1/2+\beta_0}\,\d_j^{\gamma_0-1}\,\ell_j^{2+\beta_2}\right)^2 + 2\,\left(\ell_j^{-1}\int_{\cU_j} \frac{|N|}{|z|^{1-\sfrac{\gamma_0}{2}}}\right)^2\\
& \leq 2\,\bmo^{1+2\beta_0}\,\d_j^{2\gamma_0-2}\,\ell_j^{4+2\beta_2} + C\,\int_{\cU_j} \frac{|N|^2}{|z|^{2-\gamma_0}}\,,
\end{align*}
where in the second line we have used Cauchy-Schwartz and $|\cU_j|\leq C \ell_j^2$. Summarizing,
\begin{equation}\label{e:riassunto}
\bmo^{\eta_0}\int_{\cB_r} \,|z|^{\gamma_0-1} |\etaa \circ N(z,w)| \leq C\sum_j \Big(\bmo^{1+\eta_0} \,\d_j^{2\gamma_0-2}\,\ell_j^{4+2\beta_2}+C\bmo^{\eta_0}\,\int_{\cU_j} \frac{|N|^2}{|z|^{2-\gamma_0}}\Big)\, .
\end{equation}
Moreover, observe that, if $L_j\in \sW_h$, then by \eqref{e:cubi_H} and $\frac{\d_j}{2}\leq |z|\leq 2\d_j$,
\begin{equation}\label{e:altezza_10}
\bmo \,\d_j^{2\gamma_0-2}\,\ell_j^{4+2\beta_2}
% \leq \bmo^{\sfrac12} \,\d_j^{\gamma_0-2+2\beta_2}\,\bmo\,\d_j^{\gamma_0-2\beta_2}\ell_j^{4+2\beta_2}
\leq C\,\bmo^{\sfrac{1}{2}} \, \int_{\mathcal{U}_j} \frac{|N|^2}{|z|^{2-\gamma_0}}\,
\end{equation}
while, if $L_j\in \sW_e$, using \eqref{e:cubi_E} and \eqref{e:key2}, we deduce, for a suitable choice of $\eta_0$,
\begin{equation}\label{e:eccesso_10}
\bmo^{1+\eta_0} \,\d_j^{2\gamma_0-2}\,\ell_j^{4+2\beta_2}
\leq C \,\bmo^{\eta_0}\,\d_j^{\beta_2}\,\ell_j^{\beta_2}\int_{\cU_j} |DN|^2\,
\leq C\,\bLambda(r)^{\eta_0}\,\int_{\cU_j} |DN|^2\,.
\end{equation}
Using \eqref{e:riassunto}, \eqref{e:altezza_10} and \eqref{e:eccesso_10}, the conformality of $\Psii$ (which in particular leaves the Dirichlet energy invariant) and the bounds in Proposition \ref{p:conformal} we conclude
\[
\bmo^{\eta_0}\int_{B_r}  |\zeta|^{\gamma_0 -1} |\etaa \circ \NN(\zeta,\xi)| \leq C \bLambda (r)^{\eta_0} \bD (r) + C \int_{B_r} \frac{|\NN|^2 (\zeta,\xi)}{|\zeta|^{2-\gamma_0}}\, .
\]
However the later integral is precisely
\[
\int_0^t \frac{\bH (t)}{t^{2-\gamma_0}}\, .
\]
This shows \eqref{e:media_pesata}.

\subsection{Proof of \eqref{e:diff masse}} Observe that $\bT_F = \bT_{\FF}$. Thus using \eqref{e:errori_massa} we have
\begin{align*}
\|T-\bT_\FF\| (\p^{-1}(\Psii(B_r))) 
&\stackrel{\eqref{e:errori_massa}}{\leq}C\,\sum_{j\in \N} \bmo^{1+\beta_0} \, \d_j^{(2\gamma_0-2+2\delta_1)(1+\beta_0)}\, \ell_j^{2+(2-2\delta_1)(1+\beta_0)}.
\end{align*}
Now, if $L_j\in \sW_e$, then using \eqref{e:key2} and \eqref{e:cubi_E} with a suitable $\eta$, we have 
\begin{align*}
\bmo^{1+\beta_0} \, \d_j^{(2\gamma_0-2+2\delta_1)(1+\beta_0)}\, \ell_j^{2+(2-2\delta_1)(1+\beta_0)}
&\leq \left(\bmo \, \d_j^{\gamma_0}\, \ell_j^{\gamma_0} \right)^{\beta_0} \,\left(\bmo \, \d_j^{2\gamma_0-2+2\delta_1}\, \ell_j^{4-2\delta_1} \right)\\
&\leq C\,\bLambda^{\eta}(r)\,\int_{\cU_j}|DN|^2\,.
\end{align*}
On the other hand, if $L_j\in \sW_h$, then by \eqref{e:cubi_H} and our choice of the constants,
\begin{align*}
\bmo^{1+\beta_0} \, \d_j^{(2\gamma_0-2+2\delta_1)(1+\beta_0)}\, \ell_j^{2+(2-2\delta_1)(1+\beta_0)}& =
\bmo^{1+\beta_0} \, \d_j^{(2\gamma_0-2+2\delta_1)(1+\beta_0)}
\ell_j^{-2\delta_1+\beta_0(2-2\delta_1)-2\beta_2}\ell_j^{4+2\beta_2}
\\
&\leq \bmo^{\sfrac12+\beta_0}\,\d_j^{2\gamma_0\beta_0}\,
\bmo^{\sfrac12}\d_j^{\gamma_0-2\beta_2+\gamma_0 -2}\ell_j^{4+2\beta_2}
\\
&\leq \bmo^{\sfrac12+\beta_0}\,\d_j^{2\gamma_0\beta_0}\,\int_{\cU_j}\frac{|N|^2}{|z|^{2-\gamma_0}}
\end{align*}
where we used that $-2\delta_1+\beta_0(2-2\delta_1)-2\beta_2>0$.
Summing both contributions and arguing as in the previous paragraph we conclude the proof of \eqref{e:diff masse}.

\appendix

\section{Density and height bound}

In this appendix we record two estimates which are standard for area minimizing currents and can be extended with routine arguments to the three cases of Definition \ref{d:semicalibrated}. Both statements are valid for general $m$ without additional efforts and we therefore do not restrict to $m=2$ here. Consistently with \cite{DSS1,DS2} we introduce the parameter $\bOmega$, which equals 
\begin{itemize} 
\item $\bA = \|A_\Sigma\|_{C^0}$ in case (a) of Definition \ref{d:semicalibrated};
\item $\max \{\|d\omega\|_{C^0}, \|A_\Sigma\|_{C^0}\}$ in case (b);
\item $C_0 R^{-1}$ in case (c).
\end{itemize}

\begin{lemma}\label{l:density2}
There is a positive geometric constant $c(m,n)$ with the following property. If $T$ is a current as in Definition \ref{d:semicalibrated}, where $\bOmega\leq c (m,n)$, then
\begin{equation}\label{e:density}
\|T\| (\bB_\rho (p)) \geq \omega_m (\Theta (T, p) - \textstyle{\frac{1}{4}}) \rho^m \geq \omega_m \textstyle{\frac{3}{4}} \rho^m \qquad \forall p\in \supp (T), \forall r\in \dist (p, \partial U)\, .
\end{equation}
\end{lemma}

\begin{proof} 
By \cite[Proposition 1.2]{DSS1} $\|T\|$ is an integral varifold with bounded mean curvature in the sense of Allard, where $C_0 \bOmega$ bounds the mean curvature for some geometric constant $C_0$. It follows from Allard's monotonicity formula that $e^{C_0 \bOmega r} \|T\| (\bB_r (x))$  is monotone nondecreasing in $r$, from which the first inequality in \eqref{e:density} follows. The second inequality is implied by $\Theta (T, p)\geq 1$ for every $p\in \supp (T)$: this holds because the density is an upper semicontinuous function which takes integer values $\|T\|$-almost everywhere.
\end{proof}

For the proof of the next statement we refer to \cite[Theorem A.1]{DS4}: in that theorem $T$ satisfies the stronger assumption of being area minimizing (thus covering only case (a) of Definition \ref{d:semicalibrated}), but a close inspection of the proof given in \cite{DS4} shows that the only property of area minimizing currents relevant to the arguments is the validity of the density lower bound \eqref{e:density}.

\begin{theorem}\label{t:height_bound} Let $Q$, $m$ and $n$ be positive integers. Then there are
$\eps >0, c>0$ and $C$ geometric constants with the following property. Assume that $\pi_0 = \R^m\times \{0\}\subset \R^{m+n}$ and that:
\begin{itemize}
\item[(h1)] $T$ is an integer rectifiable $m$-dimensional current as in Definition \ref{d:semicalibrated} with $U = \bC_r (x_0)$ and $\bOmega \leq c$;
\item[(h2)] $\partial T\res \bC_r (x_0) = 0$, $(\p_{\pi_0})_\sharp T\res \bC_r (x_0) = Q\a{B_r (\p_{\pi_0} (x_0))}$ 
and $E:= \bE (T, \bC_r (x_0)) < \eps$.
%, where $\pi_0 = \R^m\times \{0\}$ and $\bC_r (x_0) = \bC_r (x_0, \pi_0)$.
\end{itemize}
Then there are $k\in \N$, points $\{y_1, \ldots, y_k\}\subset \R^{m+n}$ and positive integers $Q_1, \ldots, Q_k$ such that:
\begin{itemize}
\item[(i)] having set $\sigma:= C E^{\sfrac{1}{2m}}$, the open sets $\bS_i := \R^m \times (y_i +\, ]-r \sigma, r \sigma[^n)$
are pairwise disjoint and $\supp (T)\cap \bC_{r (1- \sigma |\log E|)} (x_0) \subset \cup_i \bS_i$;
\item[(ii)] $(\p_{\pi_0})_\sharp [T \res (\bC_{r(1-\sigma |\log E|)} (x_0) \cap \bS_i)] =Q_i  \a{B_{r (1-\sigma |\log E|)}(\p_{\pi_0} (x_0), \pi_0)}$ $\forall i\in \{1, \ldots , k\}$.
\item[(iii)] for every $p\in \supp (T)\cap \bC_{r(1-\sigma |\log E|)} (x_0)$ we have $\Theta (T, p) < \max \{Q_i\} + \frac{1}{2}$.
\end{itemize}
\end{theorem}

\section{Proof of Proposition \ref{p:conformal}}\label{a:conformal}

In order to prove the Proposition we recall the following classical fact about the existence of conformal coordinates. 
As in the rest of the paper, $e$ denotes the standard euclidean metric.

\begin{lemma}\label{l:conformal}
For every $k\in \N$ and $\alpha, \beta\in ]0,1[$ there are positive constants $C_0$ and $c_0$ with the following properties.
Let $g$ be a $C^{k,\beta}$ Riemannian metric on the unit disk $B_2 \subset \R^2$ with $\|g-e\|_{C^{0,\alpha}} \leq c_0$. Then there exists an orientation preserving diffeomorphism $\Lambda:\Omega\to B_2$ and a positive function $\lambda: \Omega \to \R$ such that
\begin{itemize}
\item[(i)] $\Lambda^\sharp g = \lambda e$;
\item[(ii)] $\|\Lambda - {\rm Id}\|_{C^{1,\alpha}} + \|\lambda -1\|_{C^{0,\alpha}}\leq C_0 \|g-e\|_{C^{0, \alpha}}$;
\item[(ii)] $\|\Lambda- {\rm Id}\|_{C^{k+1,\beta}} + \|\lambda -1\|_{C^{k,\beta}} \leq C_0 \|g-e\|_{C^{k,\beta}}$. 
\end{itemize}
\end{lemma}

Although the statement above is a well-known fact (and it follows, for instance, from the treatment of the problem given in \cite[Addendum 1 to Chapter 9]{Spivak4}),  we have not been able to find a classical reference for it. However a complete proof can be found in the Appendix of \cite{DeInSz}. 

\begin{proof}[Proof of Proposition \ref{p:conformal}] After rescaling we can assume that $\rho \geq 2^Q$.
We fix $Q$ and drop subscripts in $\gira_{Q, 2}$. Observe also that, if we rescale by a large factor $R$, the constants $C_i$ in Definition \ref{d:admissible} can then replaced by the constants $C_i R^{-\alpha}$. Hence, without loss of generality we can assume that $C_i$ is sufficiently small.

Let $\Phii: \gira \to \R^{n+2}$ be the graphical parametrization of the branching and recall that $g= \Phii^\sharp e$.
Fix a point $(z_0,w_0)\in \gira\setminus \{0\}$, let $r:= |z_0|/2$ and observe that on $B_r (z_0,w_0)$ we can use $z$ as a chart and  compute the metric tensor explicitely
as
\[
g_{ij} (z,w) = \delta_{ij} + \partial_i u (z,w) \partial_j u (z,w) =: \delta_{ij} + \sigma_{ij} \, .
\]
It then follows easily that
\begin{align}
|D^j \sigma (z)|\leq & C_0 C_i^2 |z|^{2\alpha - j}\qquad \mbox{for $j\in \{0,1,2\}$}\label{e:est_g_2}\\
[D^2 \sigma]_{\alpha, B_r (z_0, w_0)} \leq & C_0 C_i^2 r^{\alpha -2}\, .\label{e:est_g_3}
\end{align}

\medskip

{\bf Step 1.}  Next consider the map $\bW\colon \bC= \R^2\supset B_2 \to \gira$ defined by $\bW (z):=(z^{Q},z)$.  We set
\[
\bar{g} = \bW^\sharp g = (\Phii \circ \bW)^\sharp e\, .
\]
We then infer that (following Einstein's convention on repeated indices)
\[
\bar{g}_{ij}  (z) = Q^2 |z|^{2Q-2} \delta_{ij} + \sigma_{kl} (\bW (z)) \partial_i \bW_l \partial_j \bW_k\, ,
\]
and we set
\[
\tau (z) := (Q^2 |z|^{2Q-2})^{-1} \bar{g} (z)\, .
\]
We then easily see that
\[
|\tau (z) - e| \leq C_0|z|^{-(2Q-2)} |D \bW (z)|^2 |\bW (z)| \leq C_0 C_i^2 |z|^{2Q\alpha}\, .
\]
Differentiating the identity which defines $\tau$ we also get
\begin{align*}
|D\tau (z)| \leq &C_0 |z|^{-(2Q-1)} |D \bW (z)|^2 |\sigma (\bW (z))| + C_0 |z|^{-(2Q-2)} |D^2 \bW (z)| |D \bW (z)| |\sigma (\bW (z))|\\
& + C_0 |z|^{-(2Q-2)} |D\bW (z)|^2 |D\sigma (\bW (z))| |z|^{Q-1}\\
\leq & C_0 C_i^2 |z|^{2Q\alpha-1}\, .
\end{align*}
Analogous computations lead then to the estimates 
\begin{align}
|D^j(\tau-e)|(z)\leq & C_0 C_i^2 |z|^{2 Q\alpha-j}\qquad \mbox{for $j\in \{0,1,2\}$}\label{e:stimetau}\\
[D^2 \tau]_{\alpha, B_s (z)} \leq & C_0 C_i^2 |z|^{2 Q\alpha -2-\alpha}\qquad \mbox{for $s = |z|/2$.}\label{e:stimetau_2}
\end{align}
Interpolating between the $C^1$ and the $C^0$ bound, we easily conclude that
\[
[\tau]_{2Q\alpha, B_{2r}\setminus B_r}\leq C_0 C_i^2\, .
\]
Note in particular that $\tau$ (unlike $g$) can be extended to a nondegenerate $C^{0,Q\alpha}$ metric to the origin. 

Since $C_i$ can be assumed sufficiently small, we can apply Lemma \ref{l:conformal} to find an orientation preserving diffeomorphism $\Lambda\colon \Omega \to B_2$ and a function $\lambda: \Omega \to \R^+$ such that
\begin{align}
\Lambda^\sharp \tau = &\bar \lambda e\label{e:conforme}\\
\|\Lambda - {\rm Id}\|_{C^{1, 2Q\alpha}} + \|\bar \lambda - 1\|_{C^{0, 2Q\alpha}} \leq &C_0 C_i\label{e:conf_Hoelder}\, .
\end{align}
Observe that, without loss of generality, we can assume that $0\in \Omega$ and $\Lambda (0)=0$. In particular \eqref{e:conf_Hoelder} implies that,
for $C_i$ suitably small, $B_1 \subset \Omega$ and hence we will regard $\Lambda$ and $\lambda$ as defined on $B_1$. Next divide $\Lambda$ by $\bar \lambda (0)^{\sfrac{1}{2}}$ and keep, by abuse of notation, the same symbols for the resulting map and the resulting conformal factor in \eqref{e:conforme}. After this normalization we achieve that $\bar \lambda (0)=1$ and that the estimates \eqref{e:conf_Hoelder} still hold with a larger $C_0$. Moreover, $\bar\lambda (0)=1$ implies that $D\Lambda (0) \in SO (2)$: composing $\Lambda$ with an appropriate rotation we can then assume that $D \Lambda (0)$ is the identity. This implies that
\begin{align}
|\bar \lambda (z)-1| \leq & C_0 C_i |z|^{2Q\alpha}\label{e:factor_decay_1}\\
|D^j (\Lambda (z) - z)| \leq & C_0 C_i |z|^{1+2Q\alpha -j} \qquad \mbox{for $j\in \{0,1\}$}\, .\label{e:conf_decay_1}
\end{align}

\medskip

{\bf Step 2.} We next wish to estimates the higher derivatives of both $\Lambda$ and $\bar \lambda$. We adopt the following procedure. We fix a point
$p\neq 0$ and let $r:= |p|/2$. We then apply a simple scaling argument to rescale $B_r (p)$ to a ball of radius $2$ so that we can apply Lemma \ref{l:conformal}. If we rescale back to $B_r (p)$ it is then easy to see that we find maps $\Lambda_p: \Omega_p \to B_r (p)$, $\lambda_p: \Omega_p \to \R^+$ with the properties
properties:
\begin{align}
\Lambda_p^\sharp \tau = & \lambda_p g \label{e:conforme_2}\\
\|\Lambda_p - {\rm Id}\|_{C^{1,2Q\alpha}} + \|\lambda_p-1\|_{C^{0,2Q\alpha}}\leq & C_0 C_i \label{e:conf_Hoelder_2}\\
[\Lambda_p - {\rm Id}]_{3,\alpha} + [\lambda_p -1]_{2,\alpha} \leq & C_0 C_i r^{2Q\alpha-2-\alpha}\label{e:conf_higher_1}\, .
\end{align}
Define $\Xi := \Lambda \circ \Lambda_p^{-1}$
Moreover, its domain is $B_r (p)$. Since
\[
\sup_{z\in B_r (p)} |\partial_z (\Xi (z)- z)| \leq C_0 r^{2 Q\alpha}\, ,
\]
we easily conclude the higher derivative estimates
\begin{align*}
\|\partial^k_z (\Xi (z) -z)\| \leq C_0 C_i r^{2 Q\alpha -k} \qquad \mbox{for $k\in \{1,2,3,4\}$}\, ,
\end{align*}
which, by holomorphicity, are actually estimates on the full derivatives.
Since $\Lambda= \Xi \circ \Lambda_p$ we then easily conclude that
\begin{align}
|D^{j+1} \Lambda (z)| + |D^j (\bar \lambda (z) -1)|\leq & C_0 C_i |z|^{2Q\alpha -j} \qquad \mbox{for $j\in \{0,1,2\}$}\\
[D^3 \Lambda]_{\alpha,B_r (z)} + [D^2 \bar \lambda]_{\alpha , B_r (z)} \leq & C_0 C_i r^{2Q\alpha -2-\alpha} \qquad \mbox{for $r= |z|/2 >0$}\, .
\end{align}
Finally notice that
\begin{equation}\label{e:vera_relazione}
(\Lambda^\sharp \bar{g})\,  (z) = Q^2 |\Lambda (z)|^{2Q-2} \bar\lambda (z) e\, .
\end{equation}

\medskip

{\bf Step 3.} We are finally ready to define $\Psii := \Phii \circ \bW \circ \Lambda \circ \bW^{-1}$. First of all observe that
\[
(\Psii^\sharp e) (z,w) = ((\bW^{-1})^\sharp \Lambda^\sharp \bar{g} ) (z,w) =  \frac{|\Lambda (\bW^{-1} (z,w))|^{2Q-2}}{|z|^{2-2/Q}} \bar \lambda (\bW^{-1} (z,w)) e_Q =: \lambda (z,w) e_Q \,  .
\]
Since $|\bW^{-1} (z,w)| = |z|^{1/Q}$, we can also estimate
\begin{align*}
|\lambda (z,w) -1|\leq &\frac{|\Lambda (\bW^{-1} (z,w))|^{2Q-2}}{|z|^{2-2/Q}} |\bar{\lambda} (\bW^{-1} (z,w))-1|
+ C \frac{|\Lambda (\bW^{-1} (z,w))|^{2Q-2} - |z|^{2-2/Q}}{|z|^{2-2/Q}}\\
\leq& C_0 C_i^2 |\bW^{-1} (z,w)|^{2Q\alpha} + C_0 |z|^{-1/Q}
\left(|\Lambda (\bW^{-1} (z,w))| - |\bW^{-1} (z,w)|\right)\\
\leq & C_0 C_i^2 |z|^{2\alpha} + C_0 C_i^2 |z|^{-1/Q} |\bW^{-1} (z,w)|^{1+2Q\alpha}
\leq C_0 C_i^2 |z|^{2\alpha}\, .
\end{align*}
Similarly
\begin{align*}
|D\lambda (z,w)| \leq & C_0 |D\bar{\lambda} (\bW^{-1} (z,w))||z|^{-1} + C_0 \left|D \frac{|\Lambda  (\bW^{-1} (z,w))|^{2Q-2}}{|\bW^{-1} (z,w)|^{2Q-2}}\right|\\
\leq & C_0 C_i^2 |z|^{2\alpha-1} + C_0 \left| D \frac{|\Lambda (\bW^{-1} (z,w)|}{|\bW^{-1} (z,w)|}\right|
\end{align*}
and observe that
\begin{align*}
\left|D \frac{|\Lambda (\bW^{-1})|}{|\bW^{-1} |}\right| = &\left|\left(\frac{D\Lambda (\bW^{-1})}{|\Lambda (\bW^{-1})||\bW^{-1}|}
- \frac{|\Lambda (\bW^{-1})|}{|\bW^{-1}|^3}{\rm Id}\right) D\bW^{-1} \bW^{-1}\right|\\
\leq & C_0 |D\bW^{-1}| |\bW^{-1}|^{-1} \left(|D\Lambda (\bW^{-1}) - {\rm Id}| +  |\bW^{-1}|\left(|\Lambda (\bW^{-1}) - (\bW^{-1})|\right)\right)\\
\leq & C_0 C_i^2 |D\bW^{-1}||\bW^{-1}|^{2Q\alpha-2}\, .
\end{align*}
Recalling that $|D\bW^{-1} (z,w)|\leq |z|^{1/Q-1}$, $|\bW^{-1} (z,w)|= |z|^{1/Q}$, we conclude
\[
|D \lambda (z,w)|\leq C_0 C_i^2 |z|^{2\alpha-1}\, .
\]
The estimates on the second derivative and its H\"older norm follow from similar computations. 

We now come to the estimates on $\Psii$. Let $\bar \Lambda := \bW \circ \Lambda \circ \bW^{-1}$. Fix $(z_0, w_0)\neq 0$, let $r:= |z_0|/2$ and use $z$ as a local chart. It will then suffice to show that 
\begin{align}
|D^j (\bar \Lambda (z) -z)|\leq & C_0 C_i |z|^{1+\alpha -l}\qquad \mbox{for $j \in \{0,1,2,3\}$}\\
[D^3\bar \Lambda]_{\alpha, B_r (z_0, w_0)} \leq & C_0 C_i |z|^{-2}\, .
\end{align}
On the other hand since $\bar \Lambda (0,0) = (0,0)$, it actually suffices to show the first estimate for $j=1$ to obtain it in the case $j=0$.

We start computing the first derivatives:
\[
D \bar \Lambda = D \bW (\Lambda \circ \bW^{-1}) D\Lambda (\bW^{-1}) D\bW^{-1}\, .
\]
Recalling that $D\bW (\bW^{-1})D\bW^{-1} = {\rm Id}$, we estimate
\begin{align*}
|D\bar \Lambda (z) - {\rm Id}| \leq &
|D\bW (\Lambda (\bW^{-1} (z))) - D \bW (\bW^{-1} (z))| |D\Lambda (\bW^{-1} (z))| |D\bW^{-1} (z)|\\
&+|D\bW (\bW^{-1} (z))| |D\Lambda (\bW^{-1} (z))-{\rm Id}||D\bW^{-1} (z)|\\
\leq & C_0 |\bW^{-1} (z)|^{Q-1}|\Lambda (\bW^{-1} (z)) - \bW^{-1} (z)| |z|^{1/Q-1}\\
\leq&+C_0 C_i^2 |\bW^{-1} (z)|^{Q-1} ||\bW^{-1} (z)|^{2Q\alpha} |z|^{1/Q-1}\\
\leq&C_0 C_i^2 |\bW^{-1} (z)|^{Q+2Q\alpha} |z|^{1/Q-1} + C_0 C_i^2 |z|^{2\alpha}
\leq C_0 C_i^2 |z|^{2\alpha}\, .
\end{align*}
Similar computations give the estimates on the higher derivatives.  
\end{proof}

\bibliographystyle{plain}
%\bibliography{references-Cal}

\end{document}